\documentclass[10pt]{siamltex}

\usepackage{amsfonts,amsmath,amssymb}
\usepackage{siampaper}
\usepackage{graphics,url,color,epsfig}
\usepackage[hidelinks]{hyperref}

\usepackage{stmaryrd}

\newcommand{\abs}[1]{\left\vert#1\right\vert}
\newcommand{\px}[1][x]{\partial_{#1}}
\newcommand{\dx}[1][x]{\,{\rm d}#1}

\newcommand{\cpt}[2][t]{{}_CD^{#2}_{0,#1}}

\newcommand{\mfrac}[1][2]{\frac{1}{2}}
\newcommand{\sfrac}[2]{ {#1}/{#2}}

 \allowdisplaybreaks

\title{Second-order numerical methods  for multi-term  fractional
differential equations: Smooth and non-smooth solutions
\thanks{This work was supported by the MURI/ARO on ``Fractional PDEs for Conservation Laws and Beyond: Theory, Numerics and  Applications  (W911NF-15-1-0562)'',  and  also  by  NSF (DMS  1216437). The second author of this work was also partially supported by a
start-up fund from WPI.}
}

\author{Fanhai Zeng\thanks{Division of Applied Mathematics, Brown University, Providence RI, 02912
(fanhai\_zeng@brown.edu, george\_karniadakis@brown.edu).}
\and Zhongqiang Zhang\thanks{Department of Mathematical Sciences, Worcester Polytechnic Institute, Worcester  MA, 01609 (zzhang7@wpi.edu).}
 \and George Em Karniadakis$^\dag$
}

\begin{document}

\maketitle

\begin{abstract}
Starting with the asymptotic expansion of the error equation
of the shifted Gr\"{u}nwald--Letnikov formula,
we derive a new modified weighted shifted Gr\"{u}nwald--Letnikov (WSGL) formula by introducing appropriate correction terms.
We then apply one special case of the modified WSGL formula to solve multi-term fractional
ordinary and partial differential equations, and we prove the linear stability and
second-order convergence for both smooth and non-smooth solutions.
We show theoretically and numerically that
numerical solutions up to certain accuracy can be obtained with only a few correction terms.
Moreover, the correction terms can be tuned according to the fractional derivative orders without explicitly knowing
the analytical solutions.
Numerical simulations verify the theoretical results and demonstrate that the new
formula leads to better performance
compared to other known numerical approximations with similar resolution.
\end{abstract}




\section{Introduction}\label{sec1}
The aim of this work is to provide an effective numerical method to solve
multi-term  fractional differential equations (FDEs),
where more than one  fractional differential operator is involved,
with high-order accuracy  for both smooth and non-smooth solutions.

Multi-term  FDEs  are motivated by their
flexibility to describe complex multi-rate physical
processes, see  e.g. \cite{Diethelm-B10,JiangLiu-etal12b,Luchko11,Pod-B99}.
Moreover, it is not  straightforward to extend the known numerical methods for single-term FDEs to solve multi-term FDEs.
Specifically, we find that:
(i) Some numerical methods
for single-term FDEs can be  extended to multi-term FDEs, but their numerical
stability and convergence analysis are not easy to prove; (ii)
Very low accurate numerical solutions may be obtained by extending the majority of the known numerical
methods for single-term FDEs to multi-term FDEs due to their often unreasonable requirement on the
high regularity of the solutions.


Existing numerical methods for FDEs
can be broadly divided into two classes.
\begin{itemize}
\item[(a)] {\em FDEs with smooth solutions}: The majority of numerical methods
have been developed   for single-term
FDEs with smooth solutions, in which    the fractional derivative
operators in these equations are discretized by the  (shifted)
Gr\"{u}nwald--Letnikov (GL) formula \cite{ChenLTA07,GaoSS15,MeeTad04,Pod-B99},
the L1 method and its modification
\cite{LinXu07,ZhangSunLiao14},
the weighted shifted Gr\"{u}nwald--Letnikov (WSGL) formulas \cite{TianZD14,WangVong14b},
the (weighted) fractional central difference methods
\cite{CelikDuman12,DingLC14a,DingLC14b,ZhaoSH14},
the fractional linear multi-step methods
\cite{CueLubPal06,Lub86,Yuste06,Zeng14,ZengLLT13,ZengLLT15,ZengZK16},
the spectral approximations \cite{ZayKar14b,ZengZK15,ZhengLATS15},
and so on \cite{ChenDeng14,ForMR13,MusMcL13,Sousa12}.
Some of these methods have been extended to solve
multi-term   FDEs with smooth solutions,
such as  the L1 method in time with spatial discretization by
the finite difference method (see e.g. \cite{Liu-etal13,RenSun14})
and   finite element method (see e.g. \cite{JinLaz-etal15}),
the predictor-corrector method in time with finite difference methods
in space \cite{Liu-etal13,YeLiu14}, {and some others \cite{FordCon09,PdeasTamme11},
just to name a few.}

\item[(b)]  {\em FDEs with non-smooth solutions}: Generally speaking, the analytical solutions
to FDEs are not smooth in real applications. For example, even for smooth inputs,
the solutions to FDEs usually have a weak singularity at the boundary, see e.g.
\cite{Diethelm-B10,JinZhou14,Luchko11,MaoShen16,WangZhang15,ZhangZK15}.
Therefore, the aforementioned numerical methods will produce numerical
solutions of low accuracy when applied to solve these FDEs. In order to derive numerical schemes
of uniformly high-order convergence for FDEs with non-smooth solutions,
several approaches have been proposed:
\begin{itemize}
\item[(b1)]   Use nonuniform/refined grids in the discretization of the fractional operators, see e.g. \cite{McLMus07,MusMcL13,QuiYus13,ZhangSunLiao14}). With a geometrically graded mesh, one can adaptively resolve a weak singularity at the endpoint.
\item[(b2)]  Separate the solution $U$ of the considered FDE into two parts of
$U^{(1)}$  and $U^{(2)}$, satisfying
    $U=U^{(1)}+U^{(2)}$ (see e.g. \cite{Lub86,Zeng14,ZengLLT15}). Then,
    a numerical scheme is designed such that high accuracy is obtained for both $U=U^{(1)}$ and $U=U^{(2)}$.

 \item[(b3)] Use a non-polynomial (or singular) basis function to  capture the singularity
     of   the solutions to FDEs
 (see e.g. \cite{CaoHX03,ChenSW14,EsmSL11,ForMR13,JinZhou14,MaoShen16,ZayKar14b,ZengZK15,ZhangZK15}), such that
 a high-order scheme is effective.

\end{itemize}
\end{itemize}

We also note that many numerical methods  for FDEs
may impose some unreasonable restrictions on  the solutions.
For example, the L1 method (see e.g. \cite {LinXu07}) and the interpolation method
(e.g. \cite{Diethelm-B10,Sousa12}) require that the solution of the considered
FDE is sufficiently smooth such that the expected accuracy can be realized,
but as we already stated solutions of FDEs are usually not smooth, see e.g. \cite{Diethelm-B10}.
The second-order WSGL formula
(see, e.g. \cite{TianZD14,WangVong14b}) requires that the solution and its first (and/or second)
derivative have vanishing values at the boundary;  see also the corresponding works in
\cite{CelikDuman12,ChenDeng14,DingLC14a,DingLC14b,GaoSS15,ZengLLT13,ZhaoSH14,ZhouTD13}.
Consequently, the  convergence rate of these methods can be  low
even if the solutions are sufficiently smooth. In particular, we
  show numerically that the second-order
WSGL formula in \cite{TianZD14,WangVong14b}
does not exhibit global second-order  accuracy for smooth solutions;
see  numerical results in  Table \ref{tb1-1}.
The theoretical explanation can be found in Section 2.

To remove these restrictions, we will adopt  the approach (b2) to solve
multi-term FODEs and multi-term time-fractional anomalous diffusion equations
with smooth and non-smooth solutions. We note that
the analytical solutions of FODEs and
time-fractional differential equations  usually have the  form
\begin{eqnarray}\label{eq:ut}
U(t)=U^{(1)}(t)+U^{(2)}(t),\, U^{(1)}(t)=\sum_{r=0}^mc_mt^{\sigma_r},\,
 U^{(2)}(t)=c_{m+1}t^{\sigma_{m+1}}+u(t)t^{\sigma_{m+2}},  \quad
\end{eqnarray}
where $u(t)$ is a uniformly continuous function over the interval $[0,T]$, $T>0$ and
$0\leq \sigma_r<\sigma_{r+1},\, r=1,2,...,m$ and
$m$ is a   positive integer,
see e.g. \cite{CueLubPal06,Diethelm-B10,DieFF04,JiangLiu-etal12b,JiangLiu-etal12,LiLiu14,
Luchko11,McLMus07}. Moreover, $\sigma_1$, $\ldots$, $\sigma_m$ are explicitly known,  see e.g. \cite{Diethelm-B10,DieFF04}. The WSGL formula
in \cite{TianZD14,WangVong14b} is indeed of  second-order convergence for
$U^{(2)}(t)$ when $\sigma_{m+1}\geq2+\alpha$ ($\alpha$ is the fractional order) while it converges slowly for $U^{(1)}(t)$.
This  observation motivated us to
apply the the idea of \cite{Lub86} in  (b2) to obtain
desired high accuracy schemes.
Specifically, we introduce some corrections terms into
the  WSGL formula so that
the resulting modified WSGL formula is exact or highly accurate for $U^{(1)}(t)$
while the second-order convergence for $U^{(2)}(t)$ is maintained.

According to  \cite{Lub86}, the  correction terms are found using  the so-called starting weights and values to make the resulting formula exact for low regularity terms  $t^{\sigma_r}(1\leq r \leq m)$ in $U$ (see \eqref{eq:ut}). We need to solve a linear system to obtain the starting weights, whose coefficient matrix
is an ill-conditioned exponential Vandermonde matrix.
It has been pointed out in \cite{DieFord06} that  the accuracy in solving the corresponding weights
``may have serious adverse effects for the entire scheme''  when  $m$ is large.
However, it is shown in \cite{DieFord06,Lub86} that
the accuracy of Lubich's correction approach depends on the residual of the linear system for obtaining the starting weights, which was discussed in detail in \cite{DieFord06}.

Fortunately, the number of correction terms can be small and still obtain reasonable accuracy.
In this paper, we show   that several correction terms (less than ten) significantly improve  accuracy, regardless of
the regularity of the analytical solution.
Since we are using a few corrections terms, the condition number of the exponential Vandermonde matrices
is not too large and thus the linear system to derive the starting weights can be
solved accurately with double precision.
Moreover, even if the regularity indices $\sigma_r$ (see \eqref{eq:ut}) are unknown and the ``correction terms'' do not match the singularity of the analytical solutions to considered FDEs, we can still obtain
satisfactory accuracy, see  Example \ref{eg3-2}  and numerical results in Section \ref{sec:numerical}.
In particular, we present in Lemma \ref{lem:quad-2} a detailed error estimate of
the WSGL formula with correction terms, which explains why a few correction terms may lead to satisfactory accuracy.

We organize the paper as follows.
In Section \ref{eq:time-discretization}, we obtain the asymptotic error equation of the shifted
GL formula that leads to the second-order WSGL formula under mild conditions.
We then show that the WSGL formula with correction terms can lead to better accuracy.
In Section \ref{sec:multi-term-fode}, we apply one special case of the second-order WSGL formula
with correction terms to  multi-term FODEs and present the stability and convergence theory.
We further extend the second-order WSGL formula
with correction terms in Section \ref{sec:3} to the time discretization of
the multi-term time-fractional diffusion-wave equation together
with the spectral element method for spatial discretization.
Numerical results for smooth and  non-smooth solutions are included in Section \ref{sec:numerical}.
All the proofs of our lemmas and theorems are presented in Section \ref{sec:proof} before the conclusion in the last section.
In the Appendix we include some computational details, additional proofs, and also more numerical results.


\section{Finite difference approximations for fractional derivatives}\label{eq:time-discretization}

In this section, we examine the asymptotic behavior of the error equation of the
shifted  GL formula, which leads to the error estimate of the
WSGL formula \cite{TianZD14}.
Following Lubich's approach \cite{Lub86}, we then introduce correction terms to recover the global second-order
accuracy of the WSGL formula and obtain an error bound.

We first introduce  definitions of fractional integrals and derivatives.
The $\alpha$th-order  Caputo derivative operator is defined \cite{Pod-B99}
\begin{equation}\label{eq:cpt}
\cpt{\alpha}{U(t)}=D^{-(n-\alpha)}_{0,t}\left[D^nU(t)\right]
=\frac{1}{\Gamma(n-\alpha)}\int_{0}^t(t-s)^{n-\alpha-1}\frac{\dx[]^n}{\dx[s]^n}U(s)\dx[s],
\end{equation}
where $n-1<\alpha\leq n$, $n$ is a positive integer.  The fractional integral
  $D^{-\gamma}_{0,t}$ is
given by
\begin{equation}\label{eq:fint}
D^{-\gamma}_{0,t}U(t)={}_{RL}D^{-\gamma}_{0,t}U(t)
=\frac{1}{\Gamma(\gamma)}\int_{0}^t(t-s)^{\gamma-1}{U(s)}\dx[s],{\quad}\gamma>0.
\end{equation}

Let $\tau>0$ be a time stepsize and $n_T$ be a positive integer
with $\tau=T/n_T$ and $t_n=n\tau(n=0,1,...,n_T$).
The shifted GL formula (with $q$ shifts) reads  (see e.g. \cite{MeeTad04}):
\begin{equation}\label{balf}
\mathcal{B}^{\alpha,n}_{q}U
=\frac{1}{\tau^{\alpha}}\sum_{k=0}^{n+q}\omega^{(\alpha)}_{k}U(t_{n-k+q})
=\frac{1}{\tau^{\alpha}}\sum_{k=0}^{n+q}\omega^{(\alpha)}_{k}U^{n-k+q},
\end{equation}
where $\omega^{(\alpha)}_k=(-1)^k\binom{\alpha}{k}$.  We have the following error estimate for the  formula \eqref{balf}, the proof of which  can be found in Section \ref{sec:proof}.
\begin{lemma}[Error of the shifted GL  formula
\eqref{balf}]\label{lem:quad-accuracy}
Let $U(t)=t^{\sigma}\,(\sigma\geq0)$ and $\alpha$ be  a real number. Then
\begin{equation}\label{SGL}\begin{aligned}
\left[{}_{RL}D^{\alpha}_{0,t}U(t)\right]_{t=t_n}
=\mathcal{B}^{\alpha,n}_{q}U
-\tau\left(q-\frac{\alpha}{2}\right)\frac{\Gamma(\sigma+1)}{\Gamma(\sigma-\alpha)}t_{n}^{\sigma-1-\alpha}
+\tau^{2}R^{n,\alpha,\sigma},
\end{aligned}\end{equation}
where   $\mathcal{B}^{\alpha,n}_{q}$  is defined by \eqref{balf}, $n\geq |q|$, $q$ is an integer, and $R^{n,\alpha,\sigma}$ is bounded by
\begin{equation}\label{Rn}
\abs{R^{n,\alpha,\sigma}}\leq  C t_n^{\sigma-2-\alpha}.
\end{equation}
\end{lemma}
Here and throughout the paper,  the constant $C>0$ is independent of $n$ and $\tau$.

From Lemma \ref{lem:quad-accuracy},    we can eliminate the
term $t_{n}^{\sigma-1-\alpha}$ in \eqref{SGL} by a linear combination of
$\mathcal{B}_q^{\alpha,n}$ and $\mathcal{B}_p^{\alpha,n}$, $p\neq q$. Let
\begin{equation}\label{s31-1-2}
\mathcal{A}_{p,q}^{\alpha,n}U
=\frac{\alpha-2q}{2(p-q)}\mathcal{B}^{\alpha,n}_{p}U
+\frac{2p-\alpha}{2(p-q)}\mathcal{B}^{\alpha,n}_{q}U.
\end{equation}
Then  we obtain   the second-order WSGL formula in \cite{TianZD14},
\begin{equation}\label{s31-1} 
\left[{}_{RL}D_{0,t}^{\alpha}U(t)\right]_{t=t_n}=\mathcal{A}_{p,q}^{\alpha,n}U
+\tau^{2}R^{n,\alpha,\sigma}, \quad |R^{n,\alpha,\sigma}|\leq C t_n^{\sigma-2-\alpha}.
\end{equation}


Eq. \eqref{s31-1-2} is   proved in \cite{TianZD14,WangVong14b} to be of second-order convergence when
$U(0)=0$, ${}_{RL}D_{0,t}^{\alpha+2}U(t)$ and its Fourier transform belongs to  $L_1(\mathbb{R})$.
However,   these conditions are too restrictive since ${}_{RL}D_{0,t}^{\alpha+2}U(t)\notin L_1(\mathbb{R})$
for  $U(t)=t^{\sigma_r}$ with $0\leq \sigma_r\leq 1+\alpha$ when  $\alpha>0$. For example, when  $U(t)=t$, the remainder term in \eqref{s31-1}
is of the order of $\tau^2t_{n}^{-1-\alpha}$, which is  not
of second-order convergence globally as the accuracy is $O(\tau^{1-\alpha})$
when $t_n=\tau$.
Consequently, when \eqref{s31-1} is applied to solve FDEs with even smooth solutions,
we cannot  expect
a global second-order convergence unless   {the solutions
have  vanishing first-order derivatives}.
In particular,  there is almost no accuracy  at $t=0$ when  $1<\alpha<2$.
In practice, this large error near $t=0$ may lead to large accumulation of discretization errors and thus much larger error at the desired final time $t_n$; see  accuracy  tests of the formula \eqref{s31-1-2}
in Fig. \ref{fig2-2} of this section and numerical experiments
in Section \ref{sec:numerical}.

To improve the accuracy  near $t=0$,  we follow  Lubich's approach \cite{Lub86} by adding correction terms to \eqref{s31-1-2} such that the resulting formula  is indeed of second-order accuracy when $U(t)$ is of the form \eqref{eq:ut}.
 %
%
Specifically, we modify  \eqref{s31-1-2} such that
\begin{equation}\label{s31-1-3}\begin{aligned}
\left[{}_{RL}D_{0,t}^{\alpha}U(t)\right]_{t=t_n}=\mathcal{{A}}_{p,q}^{\alpha,n,m}U+R^n,\,\,\, \mathcal{{A}}_{p,q}^{\alpha,n,m}U=:\mathcal{A}_{p,q}^{\alpha,n}U
+\frac{1}{\tau^{\alpha}}\sum_{k=1}^mw_{n,k}^{(\alpha)}U(t_k).
\end{aligned}\end{equation}
In \eqref{s31-1-3}, the starting weights $\{w_{n,k}^{(\alpha)}\}$ are known at each time step as they can be determined  by setting   $R^n=0$ in \eqref{s31-1-3} for
$U(t)=t^{\sigma_r}\,(1\leq r \leq m)$.
Denote $\displaystyle\mathcal{A}_{p,q}^{\alpha,n}U=\frac{1}{\tau^{\alpha}}\sum_{k=0}^ng^{(\alpha)}_{n-k}U(t_k)$. Then the starting weights  can be solved from the following linear system of equations, see e.g. \cite{DieFord06,Lub86},
\begin{equation}\label{s31-8}
\sum_{k=1}^mw_{n,k}^{(\alpha)}k^{\sigma_r}
=\frac{\Gamma(\sigma_r+1)}{\Gamma(\sigma_r+1-\alpha)}n^{\sigma_r-\alpha}
-\sum_{k=0}^{n}g^{(\alpha)}_{n-k}k^{\sigma_r},\quad 1\leq r \leq m.
\end{equation}

The linear system \eqref{s31-8} has  an exponential Vandermonde type matrix  that is ill-conditioned when $m$ is  large
\cite{DieFord06}. The  large condition number of the matrix  may lead to  big roundoff errors of $w_{n,r}^{(\alpha)}$ $(1\leq r \leq m)$  when  computation is performed with double precision. For  $\sigma_k=k\alpha$,   we present the condition number of  the system \eqref{s31-8} in   Table \ref{tb2-1}. Hereafter we  choose   $(p,q)=(0,-1)$ in \eqref{s31-1-3} \footnote{
We choose hereafter in this paper   $(p,q)=(0,-1)$ in \eqref{s31-1-3}.}
 where the quadrature weights
 $g^{(\alpha)}_{k}$'s are defined by, see \cite{Lub86,TianZD14,WangVong14b},
\begin{equation}\label{g-k}
g^{(\alpha)}_{0}=\frac{2+\alpha}{2}\omega^{(\alpha)}_{0},{\quad}
g^{(\alpha)}_{k}=\frac{2+\alpha}{2}\omega^{(\alpha)}_{k}-\frac{\alpha}{2}\omega^{(\alpha)}_{k-1},\quad \omega^{(\alpha)}_k=(-1)^k\binom{\alpha}{k},\,k\geq 1.
\end{equation}
In fact, $g^{(\alpha)}_{k}$ satisfies $(1-z)^{\alpha}\left(1+\frac{\alpha}{2}
-\frac{\alpha}{2} z\right)=\sum_{k=0}^{\infty}g^{(\alpha)}_{k}z^k$, see \cite{Lub86}.

From Table \ref{tb2-1}, we observe that the condition number increases with $m$ and decreases with $\alpha$. However, for  $m$ and $\alpha$'s presented in Table \ref{tb2-1}, we can still have some reasonable accuracy for the starting weights.
In fact, the accuracy of \eqref{s31-1-3}  is determined somehow by
the  residual of \eqref{s31-8}.
 This observation has been made
 in  \cite{DieFord06,Lub86}.   We present the residual of the system \eqref{s31-8} in Table \ref{tb2-2}, where the residual  is computed by
$\max_{1\leq r \leq m,\,1\leq n \leq 100}|\sum_{k=1}^mw_{n,k}^{(\alpha)}k^{\sigma_r}
-\frac{\Gamma(\sigma_r+1)}{\Gamma(\sigma_r+1-\alpha)}n^{\sigma_r-\alpha}
+\sum_{k=0}^{n}g^{(\alpha)}_{n-k}k^{\sigma_r}|.$  When  $\alpha=0.05$ and $m=7$, the condition number is $1.84\times 10^{12}$  and the residual is   $1.19\times 10^{-6}$, see Table \ref{tb2-2}. In this case we can still obtain relatively high accuracy, see Figure \ref{eg3-2}(a) in Example \ref{eg3-2}.

\begin{table}[!h]
\caption{Condition numbers of the linear system \eqref{s31-8},
$\sigma_k=k\alpha\,(k\geq 1)$, and $(p,q)=(0,-1)$ in \eqref{s31-1-3}.\label{tb2-1}}
 \centering \footnotesize
\begin{tabular}{|c|c|c|c |c | c|c|c|c|}\hline
$\alpha$& $m=2$   &   $m=3$   &  $m=4$  &  $m=5$  &   $m=6$   &  $m=7$  &  $m=8$ \\ \hline
$0.05$ &1.15e+02&1.28e+04&1.41e+06&1.54e+08&1.69e+10&1.84e+12&2.02e+14 \\ \hline
$0.1$  &5.80e+01&3.20e+03&1.76e+05&9.72e+06&5.41e+08&3.04e+10&1.73e+12\\ \hline
$0.3$  &2.03e+01&3.87e+02&7.86e+03&1.73e+05&4.12e+06&1.04e+08&2.81e+09\\ \hline
\end{tabular}
\end{table}

\begin{table}[!h]
\caption{The residual  of the linear system \eqref{s31-8},
$\sigma_k=k\alpha\,(k\geq 1)$.\label{tb2-2}}
 \centering\footnotesize
\begin{tabular}{|c|c|c|c |c | c|c|c|c|}\hline
$\alpha$& $m=2$   &   $m=3$   &  $m=4$  &  $m=5$  &   $m=6$   &  $m=7$  &  $m=8$ \\ \hline
$0.05$ &1.66e-15&2.27e-13&7.27e-12&4.65e-10&2.60e-08&1.19e-06&6.86e-05 \\ \hline
$0.1$  &1.77e-15&2.84e-14&4.54e-13&5.82e-11&4.65e-10&1.11e-08&8.34e-07 \\ \hline
$0.3$  &1.11e-16&7.10e-15&5.68e-14&2.27e-13&1.81e-11&4.65e-10&6.05e-09 \\ \hline
\end{tabular}
\end{table}

In our numerical simulations of this work, we choose less than eight correction terms. It  works surprisingly   well with high accuracy  even when $\alpha$ is small. For example, in Figure  \ref{eg3-2}(a), we used six correction terms for $\alpha=0.05$, the accuracy for all $t\geq 0.2$ is at the level of  $10^{-8}$, which is two orders of magnitude lower than $\tau^2=10^{-6}$.
Hence, we are motivated to consider
the asymptotic behavior of $R^n$ in \eqref{s31-1-3} when $U(t)=t^{\sigma}$, addressing the practical effect of correction terms.
\begin{lemma}\label{lem:quad-2}
Let $U(t)=t^{\sigma}\,(\sigma\geq0)$ and $\alpha$ be a real number. Let $R^n$ be defined in  $\eqref{s31-1-3}$, $S_m^{\sigma}=\prod_{k=1}^m|\sigma-\sigma_k|$, $\sigma_{\max}=\max\{\sigma,\sigma_1,...,\sigma_m\}$. Then we have
\begin{equation}\label{Rn2}
 {\abs{R^n}}
\leq  {\tau^{\sigma-\alpha}} \left[CS_m^{\sigma}
\left(n^{\sigma-\alpha-2}\log^{m}(n)+n^{\sigma_{\max}-\alpha-2}\right)
+\widetilde{C}{n}^{\sigma_{\max}-2-d-\alpha}\right],
\end{equation}
where  $C$ and $\widetilde{C}$ are positive constants bounded and independent of $n$ and $\tau$.
\end{lemma}

\begin{remark}
Numerical results indicate that  {$R^n$} in \eqref{Rn2} can be bounded by
$  {|R^n|}\leq CS_m^{\sigma} {\tau^{\sigma-\alpha}}n^{\sigma_{\max}-\alpha-2}.$
See the supplementary material.
\end{remark}

From Lemmas \ref{lem:quad-2} and \ref{lm4.2},  we have a uniformly second-order approximation
$\mathcal{{A}}_{0,-1}^{\alpha,n,m}U$ of $\left[{}_{RL}D_{0,t}^{\alpha}U(t)\right]_{t=t_n}$
when  $U(t)$ satisfies \eqref{eq:ut} and $\sigma_{m+1}\geq 2+\alpha$, i.e.,
$|R^n|\leq C\tau^2t_n^{\sigma_{m+1}-2-\alpha}$.
In practice, especially with double precision computation, we take only small $m$ and thus   $\sigma_{m+1}\geq 2+\alpha$ may not hold.  In this case, we may not have  the global  second-order accuracy, but we still observe accuracy improvement at $t=0$ and  small errors far from $t=0$ due to the small coefficient $S_m^{\sigma}$ in \eqref{Rn2}.

Next, we check the accuracy of the discrete operator $\mathcal{{A}}_{0,-1}^{\alpha,n,m}$ in \eqref{s31-1-3}.
\begin{example}\label{eg3-2}
Use the formula \eqref{s31-1-3}  with $(p,q)=(0,-1)$
to numerically approximate ${}_{RL}D_{0,t}^{\alpha}U(t)$.  We consider two cases: \textbf{Case I:} $U(t)=t^{8\alpha}$, where we take $\sigma_k=k\alpha$, $k\le 8$. \textbf{Case II:} $U(t)=t^{8\alpha}+t^{9\alpha}+t^{10\alpha}+t^{11\alpha}$, where we take $\sigma_k=(k+7)\alpha$.
\end{example}


\begin{figure}[!h]
\begin{center}
\begin{minipage}{0.49\textwidth}\centering
\epsfig{figure=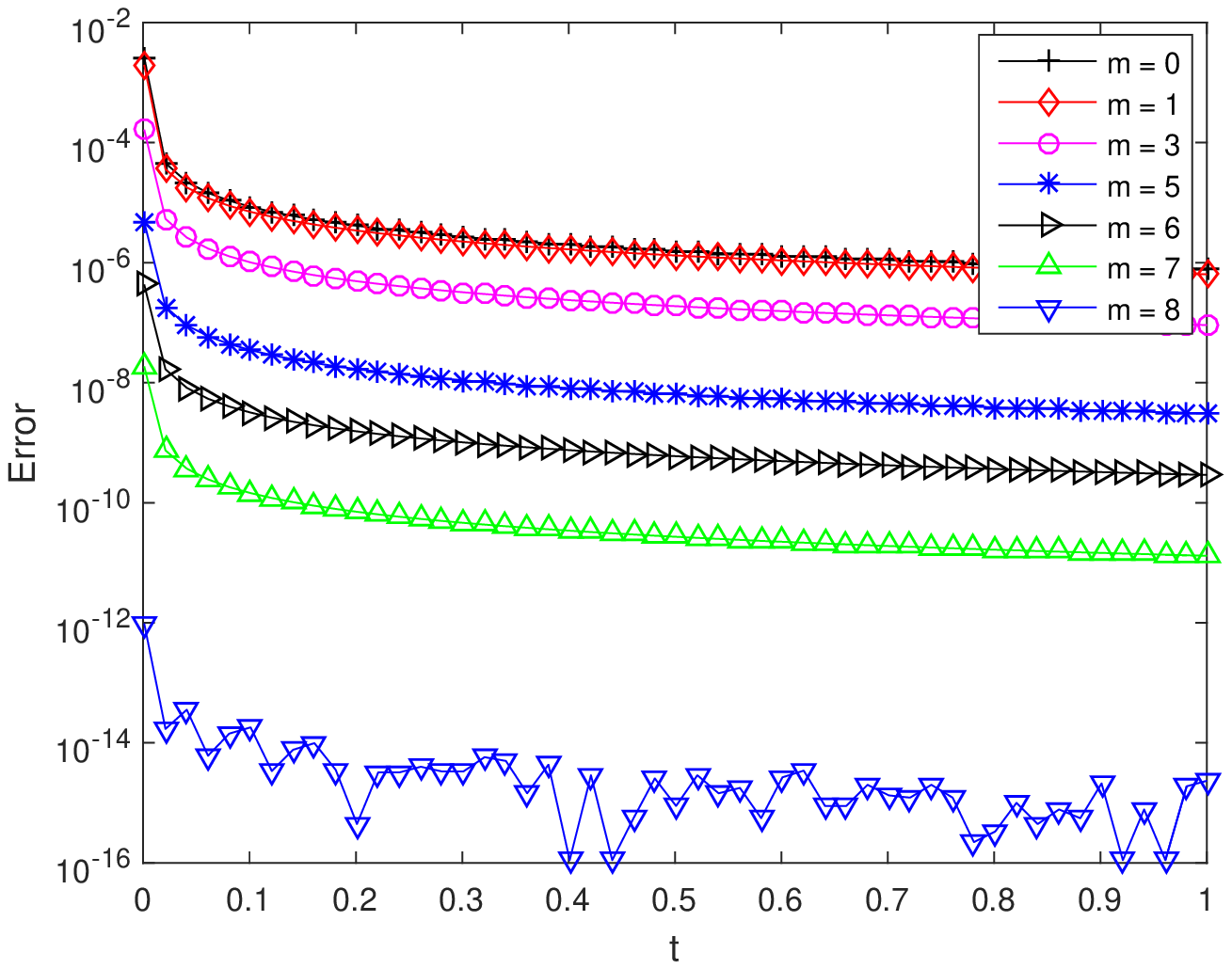,width=5.5cm}  \par{(a)   $\alpha = 0.05$.}
\end{minipage}
\begin{minipage}{0.49\textwidth}\centering
\epsfig{figure=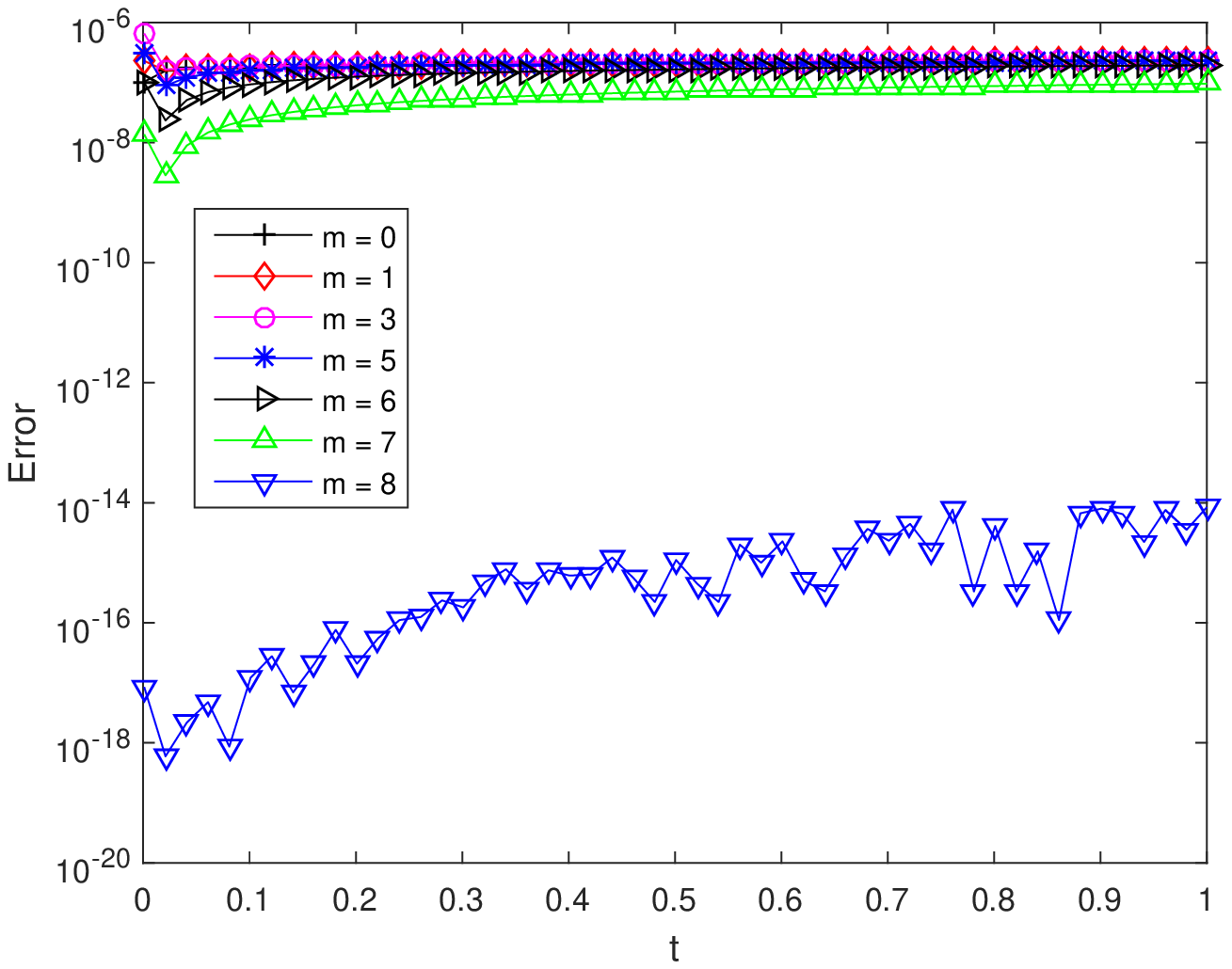,width=5.5cm}   \par{ {(b)} $\alpha = 0.3$.}
\end{minipage}
\end{center}
\caption{Pointwise errors of $\mathcal{A}_{0,-1}^{\alpha,n,m}$  for Example \ref{eg3-2},  Case I, $\tau=10^{-3}$;
small value of fractional order (left) and larger value (right).\label{fig2-2}}
\end{figure}

\begin{table}[!h]
\caption{Values of $S^{\sigma}_{m}$ for Case I, $\sigma=8\alpha$, and $\sigma_k=k\alpha\,(k\geq 1)$.\label{tb2-3}}
 \centering\footnotesize
\begin{tabular}{|c|c|c|c |c | c|c|c|c|}\hline
$\alpha$&   $m=1$   &  $m=3$  &  $m=5$  &   $m=6$   &  $m=7$  &  $m=8$ \\ \hline
$0.05$&3.50e-1&2.62e-2&7.88e-4&7.88e-5&3.94e-6&0 \\ \hline
$0.1$ &7.00e-1&2.10e-1&2.52e-2&5.04e-3&5.04e-4&0 \\ \hline
$0.3$ &2.10e-0&5.67e-0&6.12e-0&3.67e-0&1.10e-0&0 \\ \hline
\end{tabular}
\end{table}


\begin{figure}[!h]
\begin{center}
\begin{minipage}{0.49\textwidth}\centering
\epsfig{figure=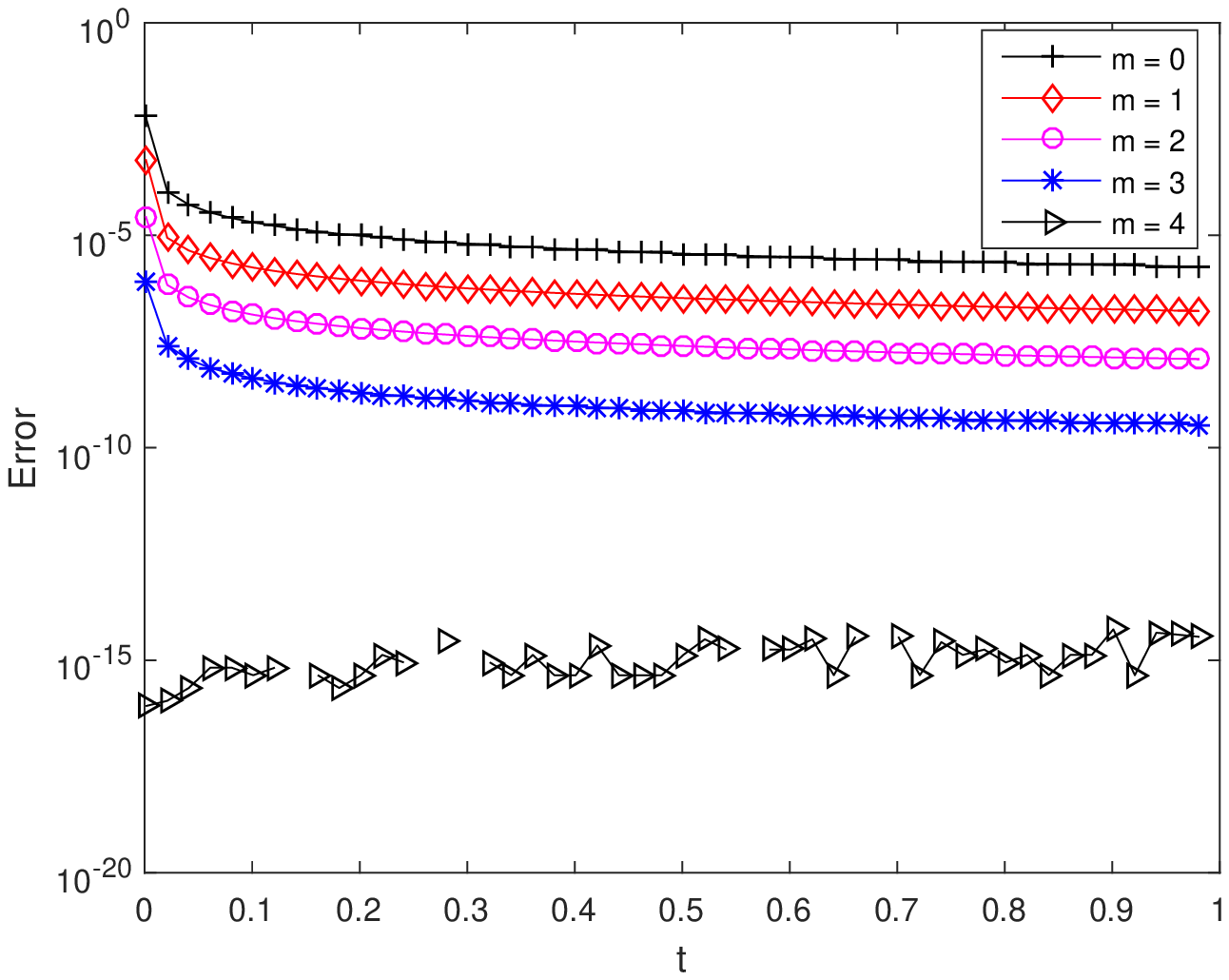,width=5.5cm}  \par{(a)   $\alpha = 0.05$.}
\end{minipage}
\begin{minipage}{0.49\textwidth}\centering
\epsfig{figure=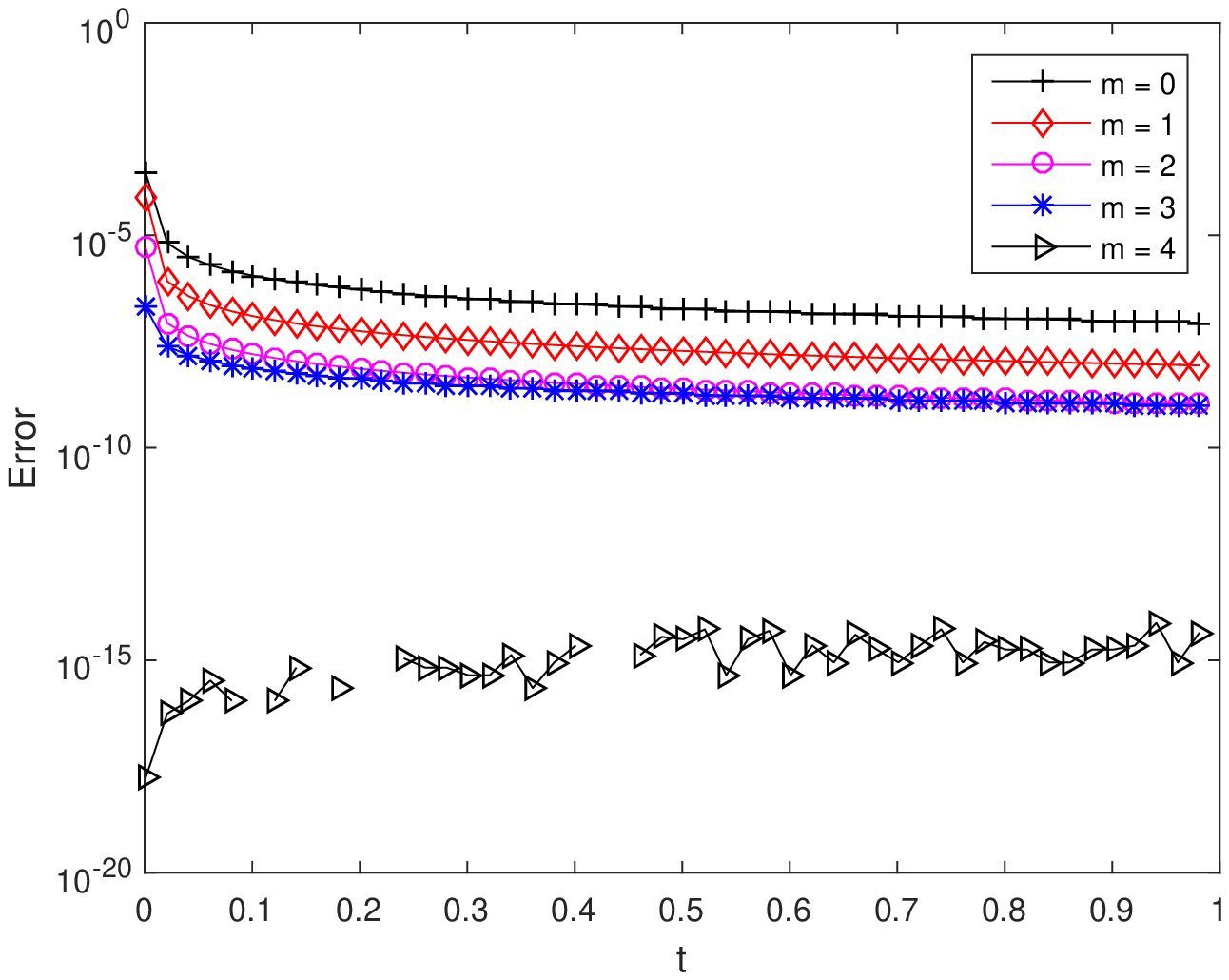,width=5.5cm}   \par{(b) $\alpha = 0.1$.}
\end{minipage}
\end{center}
\caption{Pointwise errors of $\mathcal{A}_{0,-1}^{\alpha,n,m}$  for Example \ref{eg3-2}, Case II, $\tau=10^{-3}$.\label{fig2-2-2}}
\end{figure}

The purpose of this example is to show that a small number of correction terms is sufficient to  yield relatively high accuracy whether $U(t)$ has high regularity or not.

We first consider Case I.
When {\em $\alpha$ is small}, the regularity of $U(t)$ is low. In such a case, we add only several correction terms but obtain satisfactory accuracy, see Fig. \ref{fig2-2} (a). The accuracy can be explained by
the estimate $S^{\sigma}_{m}\tau^2t_{n}^{\sigma-\alpha-2}$ in Lemma \ref{lem:quad-2}, especially
the factor $S^{\sigma}_{m}$ in the error estimate. Despite the accuracy from $\tau^2t_{n}^{\sigma-\alpha-2}$, the factor  $S_m^{\sigma}=7\times6\times \cdots\times(8-m) \alpha^m$  is very small when $\alpha$ is small and $m$ is large, see Table \ref{tb2-3}.
The small  factor $S^{\sigma}_{m}$  explains the high accuracy in Figs. \ref{fig2-2} (a).
%
When {\em $\alpha$ is relatively large},  we need only several terms  to achieve second-order   accuracy according to \eqref{s31-1-3} and Lemma \ref{lem:quad-2}.
In this case, the term $\tau^2 t_n^{\sigma-\alpha-2}$ is more pronounced than $S_m^\sigma$,  see Table \ref{tb2-3}. This effect is shown in Fig.  \ref{fig2-2} (b), where we observe  that increasing the number of correction terms does  not increase accuracy significantly  except for $m=8$, due to high regularity of $U(t)$.

For Case II, we choose $\sigma_k$ so that we match more terms of the singularity in $U(t)$; the pointwise errors are shown in Fig. \ref{fig2-2-2}. We obtain better accuracy   as the number of correction terms $m$ increases up to $4$ when we capture all the singularity of $U(t)$ that leads to
accuracy at the machine precision level.

In conclusion, we find that only a few number of corrections are needed to obtain high accuracy even when  $U(t)$ has low regularity at $t=0$. 
We also find that we do not have to match the singular terms in $U(t)$ when choosing correction terms.  In Section
\ref{sec:numerical}, we will present  numerical examples with some empirical guidelines to introduce correction terms where we do not know explicitly the singular terms in $U(t)$.

\section{Application to multi-term FODEs}\label{sec:multi-term-fode}
In this section, we apply the formula \eqref{s31-1-3} to the
discretization of multi-term  FODEs  of the form
\begin{eqnarray}
\sum_{j=1}^Q\nu_j {}_{C}D_{0,t}^{\alpha_j}Y(t)
 = f(t,Y),{\quad}t\in(0,T],T>0,{\quad}Y(0)=Y_0,\label{fode1}
\end{eqnarray}
where $\nu_1>0,\nu_j\geq0\,(2\leq j\leq Q)$, and $0<\alpha_{j+1}\leq\alpha_{j}\leq 1\,(1\leq j<Q)$.
The existence, uniqueness, and regularity of solutions to \eqref{fode1}
are investigated  in  \cite{Diethelm-B10,LiLiu14,Luchko11}.
If $f(t,Y(t))$ is smooth for $t\in(0,T]$ or $f(t,Y(t))=-Y(t)$, $\{\alpha_j\}$ are rational numbers
(see e.g. \cite{Diethelm-B10,JiangLiu-etal12b,JiangLiu-etal12,Luchko11}),
then the solution $Y(t)$ to \eqref{fode1} has the form
\begin{equation}\label{yt}
Y(t)-Y(0)=c_1t^{\sigma_1}+ c_{2}t^{\sigma_{2}}+ c_{3}t^{\sigma_{3}}+\cdots,{\quad}
\sigma_k<\sigma_{k+1},k>0.
\end{equation}

Using   ${}_{C}D_{0,t}^{\alpha_j}Y(t)= {{}_{RL}D_{0,t}^{\alpha_j}}(Y(t)-Y(0))$,
we   apply \eqref{s31-1-3} to \eqref{fode1}  that leads to
\begin{equation}\label{fode2}
\sum_{j=1}^Q {\nu_j}\mathcal{{A}}_{0,-1}^{\alpha_j,n,m_j}\widehat{Y}
= f(t_n,Y(t_n))+R^n,
\end{equation}
where $\widehat{Y}(t)=Y(t)-Y(0)$, $\mathcal{{A}}_{0,-1}^{\alpha_j,n,m_j}$
is defined by \eqref{s31-1-3},
and $m_j\,(1\leq j\leq Q)$ are suitable positive integers.
%
By \eqref{s31-1-3}, the truncation error $R^n$  in \eqref{fode2}
satisfies $R^n=\sum_{j=1}^QO(\tau^2t_{n}^{\sigma_{m_j+1}-2-\alpha_j})$.
 Let $y^n$ be the approximate solution
of $Y(t_n)$. Then we derive the  following fully discrete scheme for \eqref{fode1}
\begin{equation}\label{fode3}
\sum_{j=1}^Q   {\nu_j}\mathcal{{A}}_{0,-1}^{\alpha_j,n,m_j}\hat{y}
= f(t_n,y^n),{\quad}y^0=Y_0,
\end{equation}
where $\hat{y}=y-y^0$ and $\mathcal{{A}}_{0,-1}^{\alpha_j,n,m_j}$ is from \eqref{s31-1-3}.  Denote  $m=\max\{m_1,m_2,...,m_Q\}$.

\begin{remark}
We need  numerical values $y^k(k=1,2,...,m)$ to
proceed with the scheme.   Here we  solve  the  nonlinear system of $y^k(k=1,2,...,m)$ using \eqref{fode3} with $n=1,2,...,m$,
and we apply the Picard fixed-point iteration method.
Other high-order methods for  $y^k(k=1,2,...,m)$ can be applied here too.
\end{remark}

Next, we present the stability and convergence for \eqref{fode3}, the proofs of which are given
in Section \ref{sec:proof}.
\begin{theorem}[Linear stability]\label{thm3-2}
If $f(t,Y)=-\lambda Y(t),Re(\lambda)>0$, then the method \eqref{fode3} is unconditionally stable.
\end{theorem}
\begin{theorem}[Convergence]\label{thm3-1}
Let $y^n$ be the solution to \eqref{fode3} and $Y(t)$ be the solution to \eqref{fode1} satisfying
$Y(t)-Y(0)=c_1t^{\sigma_1}+c_2t^{\sigma_2}+...$,
where $\{c_k\}$ are constants and $0<\sigma_k<\sigma_{k+1},k>0$. Suppose that $f(t,Y)$ satisfies
the Lipschitz condition in its second  argument,
$m_j\,(j=1,2,...,Q)$ are suitable positive constants satisfying $\sigma_{m_j}\leq 2$.
Then there exists a positive constant $C$ independent of $n$ and $\tau$ such that
\begin{equation}\label{fode6}
|Y(t_n)-y^n|\leq C \bigg(\sum_{k=0}^{\max\limits_{1\leq j \leq Q}\{m_{j}\}}|Y(t_k)-y^k|
+\tau^{\min\big\{2,\min\limits_{1\leq j \leq Q}\{\sigma_{m_j+1}+\alpha_1-\alpha_j\}\big\}}\bigg).
\end{equation}
\end{theorem}


\begin{theorem}[Average error estimate]\label{theorem3-3} Let $q=\min_{1\leq j \leq Q}\{\sigma_{m_j+1}-\alpha_j\}$.
If $f(t,Y)=-Y(t)$ in Theorem \ref{thm3-1}, then  we have for $K\geq1$,
\begin{equation}\label{fode6-2}
\bigg(\tau\sum_{n=1}^K|e^n|^2\bigg)^{\frac{1}{2}}
\leq C\tau^{\min\{2,q+ \frac{1}{2}\}}
+C\max_{1\leq r \leq m}|e^r|
\sum_{j=1}^Q\tau^{\frac{1}{2}-\alpha_j}K^{\max\{0,\sigma_m-\alpha_j-\frac{3}{2}\}}.
\end{equation}
\end{theorem}

Two special cases of Theorem \ref{theorem3-3} are presented below. If there is no correction terms, i.e., $m_j=0$, then we have
\begin{eqnarray}
\bigg(\tau\sum_{n=1}^K|e^n|^2\bigg)^{1/2}
&\leq&C\tau^{\min\{2,q+ 0.5\}}=C\tau^{\min\{2, \sigma_1-\alpha_1+ 0.5\}}.
\label{apx-fode4-7}
\end{eqnarray}
If $m_j=m$ and $\max_{1\leq r \leq m}|e^r|\leq C\tau^{\sigma_{m+1}}$, then we have
\begin{eqnarray}
\bigg(\tau\sum_{n=1}^K|e^n|^2\bigg)^{1/2}
&\leq&C\tau^{\min\{2,\sigma_{m+1}-\alpha_1+0.5\}}.
\label{apx-fode4-8}
\end{eqnarray}


\section{Application to   time-fractional diffusion-wave equation}\label{sec:3}
In this  section, we consider the following  time-fractional
diffusion-wave  equation,
see e.g. \cite{CaiCZh14,Liu-etal13}:
\begin{equation}\label{wave}
\left\{\begin{aligned}
&\px[t]^2U+\nu\,{}_{C}D_{0,t}^{1+\alpha}U=\mu\,\px^2 U+f(x,t),
{\quad}(x,t){\,\in\,}\Omega{\times}(0,T],T>0,\\
&U(x,0)=\phi_0(x),{\quad}\px[t]U(x,0)=\psi_0(x),{\quad}x{\,\in\,}\bar{\Omega},\\
&{U(x,t)=0,{\quad}(x,t){\,\in\,} \partial\Omega\times(0,T]},
\end{aligned}\right.
\end{equation}
where $0<\alpha\leq1,\,\nu\geq 0,\mu>0$, $\Omega=(a,b)$.
We apply the quadrature formula \eqref{s31-1-3} in time and
spectral element method in space for the discretization of Eq. \eqref{wave}.
We also present  the rigorous stability
and convergence analysis of the present numerical scheme,
the proofs of which  can be found in the supplementary material.

The key assumption here is that the analytical solution $U(t)=U(x,t)$
to \eqref{wave} satisfies the following form
\begin{equation}\label{assu:fpde}
U(t)-U(0)-t\px[t]U(0)=\sum_{r=1}^{m}c_r(x)t^{\sigma_r}+ c_{m+1}(x)t^{\sigma_{m+1}}+\cdots,
\end{equation}
where $1<\sigma_r<\sigma_{r+1}$. Indeed, when $f(x,t)$ is smooth in time and $\alpha$ is rational,
the analytical solution of \eqref{wave} has the form as \eqref{assu:fpde},
see e.g. \cite{Diethelm-B10}.

\subsection{Time discretization}
Denote  $V(x,t)=\px[t]U(x,t)$, where $U(x,t)$ satisfies \eqref{assu:fpde}.
Then $V(t)=V(x,t)$ satisfies
$V(t)-V(0)=\sum_{r=1}^{m+1}d_r(x) t^{\sigma_r-1}+\cdots.$
Hence, we derive from \eqref{wave}
\begin{eqnarray}
\px[t]V(t)+\nu\,{}_{RL}D_{0,t}^{\alpha}(V(t)-V(0))&=&\mu\,\px^2 U(t)+f(t),\label{wave:1}\\
\px[t]\px^2 U(t)&=&\px^2 V(t).\label{wave:2}
\end{eqnarray}
The main task in the following is to construct a second-order approximation for
each differential operator in \eqref{wave:1}--\eqref{wave:2}, i.e.,
the first-order time derivative operator $\px[t]$
and the time-fractional derivative operator ${}_CD_{0,t}^{\alpha}$.

For simplicity,  we denote  $U^n=U^n(\cdot)=U(\cdot,t_n)$ and
 $\delta_tU^{n+\mfrac}=\frac{U^{n+1}-U^n}{\tau}.$
From \eqref{SGL}, we have $\delta_tg^{n+\mfrac}=\mathcal{B}^{1,n+1}_{0}g
=\mathcal{B}^{1,n}_{1}g$, which yields
\begin{equation}\label{s5:eq-0}
 \px[t]g(t_{n+1})+\px[t]g(t_{n})
 =\mathcal{B}^{1,n+1}_{0}g+\mathcal{B}^{1,n}_{1}g
 +O(\tau^2t_{n}^{\sigma_r-3}),{\quad}g(t)=t^{\sigma_r}.
\end{equation}
Let $\widetilde{U}(t)=U(t)-U(0)-tV(0)$. Then   from \eqref{SGL} and \eqref{s5:eq-0},  we have
\begin{equation}\label{s5:eq-1}
\frac{1}{2}\left[\px[t]\widetilde{U}(t_{n+1})+\px[t]\widetilde{U}(t_{n})\right]
=\delta_t\widetilde{U}^{n+\mfrac}
+\frac{1}{\tau}\sum_{r=1}^{m_1} u_{n,r}\widetilde{U}^r
+O(\tau^2t_{n}^{\sigma_{m_1+1}-3}),
\end{equation}
where $m_1\leq m$ and the starting weights $\{u_{n,r}\}$ are chosen such that \eqref{s5:eq-1} is exact for
$\widetilde{U}(t)=t^{\sigma_r}(1\leq r \leq m_1)$, which leads to 
\begin{equation}\label{s5:eq-2}
\sum_{k=1}^{m_1} u_{n,k}k^{\sigma_r}
=\frac{\sigma_r}{2}\Big((n+1)^{\sigma_r-1}+n^{\sigma_r-1}\Big)
-\Big((n+1)^{\sigma_r}-n^{\sigma_r}\Big)=O(n^{\sigma_r-3}).
\end{equation}
Note that $\widetilde{U}=U(t)-U(0)-tV(0)$. We have from \eqref{s5:eq-1}
\begin{equation}\label{s5:eq-3}
\frac{1}{2}\left[\px[t]{U}(t_{n+1})+\px[t]{U}(t_{n})\right]
=\delta_t{U}^{n+\mfrac}
+\frac{1}{\tau}\sum_{r=1}^{m_1} u_{n,r}(U^r-U^0-t_rV^0)
+O(\tau^2t_{n}^{\sigma_{m_1+1}-3}).
\end{equation}
We can similarly  obtain
\begin{equation}\label{s5:eq-4}
\frac{1}{2}\left[\px[t]{V}(t_{n+1})+\px[t]{V}(t_{n})\right]
=\delta_t{V}^{n+\mfrac}
 +\frac{1}{\tau}\sum_{r=1}^{m_2} v_{n,r}(V^r-V^0)
 +O(\tau^2t_{n}^{\sigma_{m_2+1}-4}),
\end{equation}
where $m_2\leq m$ and $\{v_{n,r}\}$ are chosen such that \eqref{s5:eq-4} is exact
for $V(t)=t^{\sigma_r-1}$, which yields
\begin{equation}\label{s5:eq-5}
\sum_{k=1}^{m_2} v_{n,k}k^{\sigma_r-1}
=\frac{\sigma_r-1}{2}\Big((n+1)^{\sigma_r-2}+n^{\sigma_r-2}\Big)
-\Big((n+1)^{\sigma_r-1}-n^{\sigma_r-1}\Big)=O(n^{\sigma_r-4}).
\end{equation}
From 
\eqref{s31-1-3},  we can also choose $m_3\leq m$ such that
\begin{equation}\label{s5:eq-5-2}
\left[{}_CD_{0,t}^{\alpha}V(t)\right]_{t=t_{k}}
=\mathcal{{A}}_{0,-1}^{\alpha,k,m_3}\widehat{V}
+O(\tau^2t_{k}^{\sigma_{m_3+1}-3-\alpha}),
\end{equation}
where $\widehat{V}=V-V^0$, $k=n,n+1$.

Combining \eqref{s5:eq-4} and \eqref{s5:eq-5-2}, we have the following time discretization
for \eqref{wave:1}
\begin{equation}\label{s5:eq-6}\begin{aligned}
&\delta_tV^{n+\mfrac}+\frac{1}{\tau}\sum_{r=1}^{m_2} v_{n,r}({V}^{r}-V^0)
+\frac{\nu}{2}\left(\mathcal{{A}}_{0,-1}^{\alpha,n+1,m_3}\widehat{V}
+\mathcal{{A}}_{0,-1}^{\alpha,n,m_3}\widehat{V} \right)\\
=&\mu\,\px^2 U^{n+\mfrac}+f^{n+\mfrac}+O(\tau^2t_{n}^{\sigma_{m_2+1}-4})
+O(\tau^2t_{n}^{\sigma_{m_3+1}-3-\alpha}).
\end{aligned}\end{equation}
From \eqref{s5:eq-3},  the time discretization of  Eq. \eqref{wave:2} is \begin{equation}\label{s5:eq-7}\begin{aligned}
\delta_t \px^2 U^{n+\mfrac} +\frac{1}{\tau}\sum_{r=1}^{m_1} u_{n,r}\px[x]^2(U^r-U^0-t_rV^0)
= \px^2 V^{n+\mfrac}+O(\tau^2t_{n}^{\sigma_{m_1+1}-3}).
\end{aligned}\end{equation}

\subsection{The fully discrete spectral element method}
Let us introduce some notations before presenting our fully discrete schemes.
Let  $\Omega=(a,b)$  and $M$ be a positive integer.
Let $\Pi=\{a=x_0<x_1<...<x_{M}=b\}$
be a partition of the interval
$\Omega$.
Denote $N=(N_1,N_2,...,N_M)$,  $N_i$ is a positive integer, and
\begin{equation*}
\Omega_i=(x_{i-1},x_i),{\quad}h_i=x_{i}-x_{i-1},{\quad}h=\max_{1\leq i\leq M}\{{h_i}/{N_i}\}.
\end{equation*}

Denote $\mathbb{P}_{K}(I)$ as the polynomial space defined on the
domain $I$ with degree no greater than $K$. The approximation spaces
$V_N,V_N^0$ are defined as follows:
\begin{equation*}
 \begin{aligned}
&V_N=\{ v{\,\in\,}C(\Omega):v|_{\Omega_i} {\,\in\,} \mathbb{P}_{N_i}(\Omega_i),~1\leq{i}\leq{M}\},~~~V_N^{0}=V_N{\cap}H_0^1(\Omega).
\end{aligned}
\end{equation*}

The inner product $(\cdot,\cdot)$ and norm $\|\cdot\|$ are defined by:
\begin{equation*}
(u,v)=\int_{\Omega}{u{v}}\dx,~~~~~~~\|u\|=\Big(\int_{\Omega}|u|^2\dx\Big)^{1/2},~~~{u,v\in{L^2(\Omega)}}.
\end{equation*}

From \eqref{s5:eq-6}--\eqref{s5:eq-7}, we present the
Legendre Galerkin spectral element method (LGSEM) for
\eqref{wave:1}--\eqref{wave:2}: For $n=0,1,...,n_T-1$  and  $\forall u,v\in V_N^0$,
we find $u_N^{n+1},v_N^{n+1}\in V_N^0$,  such that
 \begin{eqnarray}
 &&(\delta_tv_N^{n+\mfrac},v) +\frac{1}{\tau}\sum_{r=1}^{m_2} v_{n,r}(v_N^r-v_N^0,v)
 +\frac{\nu}{2}\bigg[\left(\mathcal{{A}}_{0,-1}^{\alpha,n+1,m_3}\hat{v}_N,v\right)
 +\left(\mathcal{{A}}_{0,-1}^{\alpha,n,m_3}\hat{v}_N,v\right)\bigg]\nonumber\\
 &&{\quad\quad\qquad}
 +\mu\,(\px u_N^{n+\mfrac},\px v) =(I_Nf^{n+\mfrac},v),\label{scheme2-1}\\
 &&(\delta_t\px[x]u_N^{n+\frac{1}{2}},\px[x]u)
 +\frac{1}{\tau}\sum_{r=1}^{m_1} u_{n,r}(\px[x](u_N^r-u_N^0-t_rv_N^0),\px[x]u)
 = (\px v_N^{n+\frac{1}{2}},\px u),\label{scheme2-2}\\
 &&u_N^0=P_N^{1,0}U(0),\qquad v_N^0=P_N^{1,0}V(0),\label{scheme2-3}
 \end{eqnarray}
in which $\hat{v}_N^{n}=v_N^{n}-v_N^0$, $\mathcal{{A}}_{0,-1}^{\alpha,n,m_3}$
 is defined by  \eqref{s31-1-3}, $m_k(k=1,2,3)$ are suitable positive integers,
 $\{u_{n,r}\}$ satisfy \eqref{s5:eq-2},
$\{v_{n,r}\}$ satisfy \eqref{s5:eq-5},
and $I_N$ is the Legendre--Gauss--Lobatto interpolation operator defined by
$$(I_Nu)(x^i_k)=u(x^i_k),~~\,k=0,1,...,N_i,\,\,i=1,...,M,\,\,u\in C(\bar{\Omega}),$$
where $\{x_k^i\}$   are the Legendre--Gauss--Lobatto points on $\bar{\Omega}_i$.

\begin{remark}
To get  $\{u_N^k\}$ and $\{v_N^k\}$  for $1\leq k \leq m,m=\max\{m_1,m_2,m_3\}$,
we can let $n=0,1,...,m-1$ in \eqref{scheme2-1}--\eqref{scheme2-3} and solve the resulting system.
We can also use other  high-order methods in time
to obtain $\{u_N^k\}$ and $\{v_N^k\}$ for  $1\leq k \leq m$.
\end{remark}

\subsection{Stability and convergence}\label{sec4-3}
This subsection presents
 the stability and convergence of the scheme \eqref{scheme2-1}--\eqref{scheme2-3}.
We  give the following stability result.
\begin{theorem}\label{thm:stability}
Suppose that $u_N^n$ and $v_N^n$  $(n=0,1,...,n_T)$ are solutions to
\eqref{scheme2-1}--\eqref{scheme2-3}.
If $\sigma_{m_1}\leq3$ and $\sigma_{m_2},\sigma_{m_3}\leq 4$,
then there exists a positive constant $C$ independent of $n,\tau$ and $N$ such that
\begin{eqnarray}
\|v_N^{n}\|^2 &+\mu\|\px[x]u_N^{n}\|^2\leq C\bigg(\|v_N^{0}\|^2 +\mu\|\px[x]u_N^{0}\|^2+\|\px[x]v_N^{0}\|^2
+ \sum_{r=1}^{m_1} \|\px[x]\delta_tu_N^{r-\mfrac}\|^2 \nonumber\\
&+  \sum_{r=1}^{m_2}\|\delta_tv_N^{r-\mfrac}\|^2+ \sum_{r=1}^{m_3}\|\delta_tv_N^{r-\mfrac}\|^2
+\tau\sum_{k=0}^{n}\|f^{k}\|^2\bigg).\label{s5:eq-8}
\end{eqnarray}
\end{theorem}

For the nonnegative integer $k$, $H^k(\Omega)$ is the Sobolev space
equipped with the norm $\|\cdot\|_{H^k(\Omega)}$ and  semi-norm
$|\cdot|_{H^k(\Omega)}$ defined by
$$|v|_{H^k(\Omega)}=\bigg(\sum_{l=0}^{k}\|\partial^{l}_x v\|^2\bigg)^{1/2}
~~\text{and}~~\|v\|_{H^k(\Omega)}=\bigg(\sum_{s=0}^{k}|v|^2_{H^s(\Omega)}\bigg)^{1/2},{\quad}v{\,\in\,}H^k(\Omega).$$
Next, we present the convergence analysis.
\begin{theorem}\label{thm:convergence}
Suppose that $n,n_T$ and $r$ are positive integers with $0\leq n\leq n_T$ and
$U(t)=U(x,t)$ is the solution to \eqref{wave} satisfying
{$U(t)-U(0)-t\px[t]U(0)=\sum_{r=1}^{m}c_rt^{\sigma_r}+u(t)t^{\sigma_{m+1}}$,
$\sigma_{r} < \sigma_{r+1}$, $u(t)\in C[0,T]$ for each  $x$,}
$V(x,t)=\px[t]U(x,t)$, $u_N^n$ and $v_N^n$
are the solutions to the scheme \eqref{scheme2-1}--\eqref{scheme2-3}, respectively,
$m_1,m_2,m_3\leq m$
with $\sigma_{m_1}\leq 3,\sigma_{m_2},\sigma_{m_3}\leq4$.
For fixed $t$, $U(t)\in H_0^1(\Omega)\cap H^r(\Omega)$, and $f\in C(0,T;H^r(\Omega))$.
If $\sum_{k=1}^{m_1}\|\px[x]\delta_t(u_N-U)^{k-\frac{1}{2}}\|^2
+\sum_{k=1}^{\max\{m_2,m_3\}}\|\delta_t(v_N-V)^{k-\frac{1}{2}}\|^2$
$\leq C\left(\tau^{2\min\{2,\sigma_{m_1+1}-0.5,\sigma_{m_2+1}-1.5,\sigma_{m_3+1}-0.5-\alpha\}}
+ h^{2r-2}\right)$,
then for small enough $\tau$, there exists a positive constant $C$
independent of $n,\tau$ and $h$, such that
\begin{eqnarray*}
\|\px (u_N^{n}-U(t_n))\|
&\leq& C\left(\tau^{\min\{2,\sigma_{m_1+1}-0.5,\sigma_{m_2+1}-1.5,\sigma_{m_3+1}-0.5-\alpha\}}
+ h^{r-1}\right).\label{s5:eq-9}
\end{eqnarray*}
\end{theorem}
\vskip -10pt
\begin{remark}\label{remark6}
If the analytical solution $U(x,t)$ to  \eqref{wave} is sufficiently smooth in time, then
the convergence rate of \eqref{scheme2-1}--\eqref{scheme2-3} in time is $O(\tau^{2})$ by choosing
$m_1=m_2=0$ and $m_3=2$. For smooth solutions with $m_1=m_2=0$, we also have:
(i) if $m_3=0$ in \eqref{scheme2-1}--\eqref{scheme2-3}, then the convergence rate in time
is $O(\tau^{1.5-\alpha})$;
(ii) if  $m_3=1$ in \eqref{scheme2-1}--\eqref{scheme2-3}, then the convergence rate in time
is $O(\tau^{2})$ for $\alpha<1/2$,  $O(\tau^{2.5-\alpha})$    for $\alpha>1/2$,
or  $O(\log(n)\tau^{2})$  for $\alpha=1/2$ at $t=t_n$.
\end{remark}

\begin{remark}
The methodology here can be readily  extended to
  two-term time-fractional subdiffusion equation and
more generalized multi-term time-fractional subdiffusion equations,
see e.g. \cite{JinLaz-etal15,Liu-etal13,RenSun14}.
\end{remark}

\section{Numerical examples}\label{sec:numerical}
In this section, we present some numerical simulations to verify our theoretical analysis presented in the previous sections.
\begin{example}\label{s6eg0}
Consider the following two-term FODE
\begin{equation}\label{sec6:eq0}
{}_{C}D^{2\alpha}_{0,t}Y(t)+\frac{3}{2}{}_{C}D^{\alpha}_{0,t}Y(t)=-\frac{1}{2}Y(t),{\quad}t{\,\in\,}
 {(0,T],T>0}\\
\end{equation}
subject to the initial condition $Y(0)=1$,   and $0<\alpha\leq1/2$.
\end{example}
The analytical solution of \eqref{sec6:eq0} is
$Y(t)=2E_{\alpha}(-t^{\alpha}/2)-E_{\alpha}(-t^{\alpha}),$
where      $E_{\alpha}(t)$ is the Mittag--Leffler function defined by
$E_{\alpha}(t)=\sum_{k=0}^{\infty}\frac{t^{k}}{\Gamma(k\alpha+1)}.$

To solve \eqref{sec6:eq0}, we apply the method  \eqref{fode3}  with $m_1=m_2=m$ and $\alpha_1=2\alpha,\alpha_2=\alpha$ in computation.
The maximum absolute error $\|e\|_{\infty}$  is measured by
$\|e\|_{\infty}=\max_{0\leq n \leq T/\tau}|e^n|,  e^n=Y(t_n)-y^n.$

First, we observe from Tables \ref{tb6-1-1}--\ref{tb6-3-2} that higher accuracy is obtained
with correction terms ($m>0$) than that without correction terms ($m=0$).
For $\alpha=0.5$, compared with $m=0$, we have gained one order of  magnitude in the maximum error when $m=1$ and  two orders of  magnitude when $m=2,3$, see Table  \ref{tb6-1-1}.  For
$\alpha=0.1$, we observe similar improvement in accuracy, see Table \ref{tb6-3-1}.
For the error at final time $t=1$, we also have similar effects  for
$\alpha=0.1,0.5$ and have even more significant improvement in accuracy, see Tables  \ref{tb6-1-2} and \ref{tb6-3-2}. The convergence order  for $\alpha=0.5$ in the maximum sense is  consistent with the theoretical prediction in Theorem  \ref{thm3-1}, which is $\min\{2,{(m+2)\alpha}\}$, while a lower convergence rate is observed  for $\alpha=0.1$.

\begin{table}[!h]
\caption{{The maximum error $\|e\|_{\infty}$ of the method \eqref{fode3}, $\sigma_k=(k+1)\alpha$,  $\alpha=0.5$,  ${T=1}$.}}\label{tb6-1-1}
\centering\footnotesize
\begin{tabular}{|c|c|c|c|c|c|c|c|c|c|c|c|c|}
\hline
 $\tau$ & $m=0$ & Order& $m=1$ & Order& $m=2$ & Order & $m=3$   & Order\\
 \hline
$2^{-8}$ &8.1812e-4&    &6.5427e-5&    &3.2368e-6&    &1.0496e-6&    \\
$2^{-9}$ &4.2685e-4&0.93&2.4571e-5&1.41&8.6440e-7&1.90&2.8393e-7&1.88\\
$2^{-10}$&2.2033e-4&0.94&9.0197e-6&1.43&2.2482e-7&1.93&7.4557e-8&1.91\\
$2^{-11}$&1.1340e-4&0.96&3.3073e-6&1.45&5.8568e-8&1.95&1.9559e-8&1.94\\
$2^{-12}$&5.7783e-5&0.97&1.1952e-6&1.46&1.4996e-8&1.96&5.0336e-9&1.95\\
\hline
\end{tabular}
\end{table}

\begin{table}[!h]
\caption{{The maximum error $\|e\|_{\infty}$ of the method \eqref{fode3},  $\sigma_k=(k+1)\alpha$,  $\alpha=0.1$, ${T=1}$.}}\label{tb6-3-1}
\centering\footnotesize
\begin{tabular}{|c|c|c|c|c|c|c|c|c|c|c|c|c|}
\hline
 $\tau$ & $m=0$ & Order& $m=1$ & Order& $m=3$ & Order & $m=5$   & Order\\
 \hline
$2^{-8}$ &1.1149e-2&    &1.0250e-3&    &1.0556e-5&    &2.3121e-7&     \\
$2^{-9}$ &1.0262e-2&0.11&9.0252e-4&0.18&8.5194e-6&0.30&1.7132e-7&0.43 \\
$2^{-10}$&9.4163e-3&0.12&7.9112e-4&0.18&6.8266e-6&0.31&1.2569e-7&0.44 \\
$2^{-11}$&8.6257e-3&0.12&6.9196e-4&0.19&5.4520e-6&0.32&9.1801e-8&0.45 \\
$2^{-12}$&7.8776e-3&0.13&6.0262e-4&0.19&4.3243e-6&0.33&6.6411e-8&0.47 \\
\hline
\end{tabular}
\end{table}

\begin{table}[!h]
\caption{{The  absolute  error   $|e^n|$ at $T=1$
of the method \eqref{fode3}, $\sigma_k=(k+1)\alpha$,  $\alpha=0.5$.}}\label{tb6-1-2}
\centering\footnotesize
\begin{tabular}{|c|c|c|c|c|c|c|c|c|c|c|}
\hline
 $\tau$ & $m=0$ & Order& $m=1$ & Order& $m=2$ & Order & $m=3$ & Order\\
 \hline
$2^{-8}$ &2.3477e-4&    &1.3374e-5&    &1.8122e-7&    &1.0496e-6&    \\
$2^{-9}$ &1.1716e-4&1.00&4.8291e-6&1.46&4.7282e-8&1.93&2.8390e-7&1.88\\
$2^{-10}$&5.8294e-5&1.00&1.7244e-6&1.47&1.2107e-8&1.95&7.4489e-8&1.91\\
$2^{-11}$&2.9247e-5&1.00&6.2033e-7&1.48&3.1214e-9&1.96&1.9521e-8&1.94\\
$2^{-12}$&1.4620e-5&1.00&2.2116e-7&1.48&7.9342e-10&1.97&5.0190e-9&1.95\\
\hline
\end{tabular}
\end{table}

\begin{table}[!h]
\caption{{The  absolute  error   $|e^n|$ at $t=1$
of the method \eqref{fode3}, $\sigma_k=(k+1)\alpha$,  $\alpha=0.1$.}}\label{tb6-3-2}
\centering\footnotesize
\begin{tabular}{|c|c|c|c|c|c|c|c|c|c|c|}
\hline
 $\tau$ & $m=0$ & Order& $m=1$ & Order& $m=3$ & Order & $m=5$ & Order\\
 \hline
$2^{-8}$ &3.9852e-5 &    &3.8881e-6 &    &5.6808e-8 &    &1.4530e-9 &     \\
$2^{-9}$ &1.9883e-5 &1.00&1.8543e-6 &1.07&2.4782e-8 &1.20&5.7487e-10&1.34 \\
$2^{-10}$&9.8918e-6 &1.00&8.8032e-7 &1.07&1.0746e-8 &1.20&2.2518e-10&1.34 \\
$2^{-11}$&4.9626e-6 &1.00&4.2112e-7 &1.07&4.6920e-9 &1.20&8.8819e-11&1.35 \\
$2^{-12}$&2.4806e-6 &1.00&2.0045e-7 &1.07&2.0333e-9 &1.21&3.4752e-11&1.35 \\
\hline
\end{tabular}
\end{table}

From Tables \ref{tb6-1-2} and \ref{tb6-3-2}, we observe that much better accuracy is obtained far from $t=0$. This phenomenon occurs in many time stepping methods for FDEs in literature, which
can be also explained from the average error estimate
$(\tau\sum_{k=0}^n|e^n|^2)^{1/2}\leq C\tau^{\min\{2,m\alpha+0.5\}}$, see Eq. \eqref{apx-fode4-8}.  Clearly, the average error has smaller upper bound than the maximum error estimate $|e^n|\leq C\tau^{\min\{2,\sigma_{m+1}\}},\sigma_m=(m+1)\alpha$, which implies   much better numerical solutions far from $t=0$, see the average errors in Table \ref{tb6-3-3}.

Second,  we find that  a small number of  corrections terms  suffices to have high accuracy in both maximum error and the error at final time.
According to Theorem \ref{thm3-1}, we can get the global second-order accuracy when  $(m+1)\alpha=\sigma_{m+1} \geq2$, i.e.,   $m\geq 3$ for $\alpha=0.5$. Indeed, we observed second-order accuracy for $\alpha=0.5$ and $m\geq 3$ in Table \ref{tb6-1-1}, and the accuracy is also much smaller than $\tau^2$. When $\alpha$ is small, i.e., $\alpha=0.1$, we need at least 19 correction terms to get the global second-order accuracy theoretically. Yet we obtain highly accurate numerical solutions using a small number of correction terms, see the case $m=3,5$ in Tables \ref{tb6-3-1} and \ref{tb6-3-2}, and the accuracy is smaller than $\tau^2$ in Table \ref{tb6-3-2} for $m=3,5$.  Though   accuracy is improved when the number of correction terms increases,  we did not use more than $10$ terms since the starting weights in \eqref{s31-8} may suffer from round-off error when $m>10$ when computed with double precision.
Though we did not observe a sharp second-order
convergence in the presented tables, second-order accuracy can be observed if we use more than 19 correction terms with  quadruple-precision in the computation (results not presented here).

Lastly, we show  the case that $\sigma_k$ is not taken as $k\alpha$. In Tables \ref{tb6-5-1}--\ref{tb6-5-2}, we present the maximum error and the error at $t=1$ of the method \eqref{fode3} for  $\alpha=0.1$  and  $\sigma_k=k\alpha+0.05$ in \eqref{s31-8}.
We  do not exactly match the  singularity of the solution but we still obtain satisfactory accuracy as the number of ``correction terms'' $m$ increases up to $m=7$. The numerical results confirm  the estimate \eqref{Rn2}, see also explanations in Example \ref{eg3-2} in the previous section.  We further illustrate this effect in our next example solving a nonlinear FODE where we don't know the   singularity of the solution.

In summary,  we find that a smaller number of correction  terms can lead to significant improvement in accuracy.  When the regularity of the solution is relatively high (the fractional   order $\alpha$ is large here), we need  only a few correction  terms to achieve a global second-order convergence. When the regularity of the solutions is low (the fractional   order $\alpha$ is small here), we need also a few correction terms as the correction terms bring a small factor $S_{m}^\sigma$ (see Lemma \ref{lem:quad-2}) that leads to high accuracy. Moreover,
the correction terms can be chosen such that the approximation $\mathcal{{A}}_{0,-1}^{\alpha,n,m}$ can be not exact for the singular terms of exact solutions.

\begin{table}[!h]
\caption{{The average  error   $\left(\tau\sum_{n=1}^{n_T}|e^n|^2\right)^{1/2}$
of the method \eqref{fode3}, $\sigma_k=(k+1)\alpha$,  $\alpha=0.1$.}}\label{tb6-3-3}
\centering\footnotesize
\begin{tabular}{|c|c|c|c|c|c|c|c|c|c|c|}
\hline
 $\tau$ & $m=0$ & Order& $m=1$ & Order& $m=3$ & Order & $m=5$ & Order\\
 \hline
$2^{-8}$ &8.0099e-4&    &7.2902e-5&    &8.5073e-7&    &1.9323e-8&     \\
$2^{-9}$ &5.2326e-4&0.61&4.5546e-5&0.67&4.8767e-7&0.80&1.0173e-8&0.92 \\
$2^{-10}$&3.3998e-4&0.61&2.8264e-5&0.68&2.7688e-7&0.81&5.2897e-9&0.93 \\
$2^{-11}$&2.2135e-4&0.62&1.7565e-5&0.69&1.5724e-7&0.82&2.7477e-9&0.95 \\
$2^{-12}$&1.4338e-4&0.62&1.0847e-5&0.69&8.8492e-8&0.82&1.4107e-9&0.96 \\
\hline
\end{tabular}
\end{table}

\begin{table}[!h]
\caption{{The maximum error $\|e\|_{\infty}$ of the method \eqref{fode3}
 with   $\sigma_k=k\alpha+0.05,\alpha=0.1$, ${T=1}$.}}\label{tb6-5-1}
\centering\footnotesize
\begin{tabular}{|c|c|c|c|c|c|c|c|c|c|c|c|c|}
\hline
 $\tau$ & $m=0$ & $m=1$ & $m=2$ &$m=3$ & $m=4$ &$m=5$  & $m=6$   \\
 \hline
$2^{-5}$&1.1149e-2&1.8430e-3&2.9611e-4&6.5830e-5&1.8374e-5&6.0966e-6&2.3116e-6 \\
$2^{-6}$&1.0262e-2&1.6601e-3&2.6347e-4&5.8298e-5&1.6273e-5&5.4119e-6&2.0579e-6 \\
$2^{-7}$&9.4163e-3&1.4904e-3&2.3376e-4&5.1517e-5&1.4390e-5&4.7978e-6&1.8297e-6 \\
$2^{-8}$&8.6257e-3&1.3362e-3&2.0726e-4&4.5530e-5&1.2732e-5&4.2566e-6&1.6280e-6 \\
$2^{-9}$&7.8776e-3&1.1943e-3&1.8331e-4&4.0165e-5&1.1249e-5&3.7718e-6&1.4467e-6 \\
\hline
\end{tabular}
\end{table}

\begin{table}[!h]
\caption{{The  absolute  error   $|e^n|$ at $t=1$
of the method \eqref{fode3}  with   $\sigma_k=k\alpha+0.05,\alpha=0.1$.}}\label{tb6-5-2}
\centering\footnotesize
\begin{tabular}{|c|c|c|c|c|c|c|c|c|c|c|}
\hline
 $\tau$ & $m=0$ & $m=1$ & $m=2$ &$m=3$ & $m=4$ &$m=5$  & $m=6$   \\
 \hline
$2^{-5}$&3.9852e-5&6.9546e-6&1.4421e-6&3.5405e-7&1.0639e-7&3.8423e-8&1.2946e-8  \\
$2^{-6}$&1.9883e-5&3.3948e-6&6.9208e-7&1.6954e-7&5.1231e-8&1.8212e-8&6.2020e-9  \\
$2^{-7}$&9.8918e-6&1.6515e-6&3.3157e-7&8.1111e-8&2.4573e-8&8.6205e-9&2.9998e-9  \\
$2^{-8}$&4.9626e-6&8.1023e-7&1.6045e-7&3.9212e-8&1.1887e-8&4.1314e-9&1.4772e-9  \\
$2^{-9}$&2.4806e-6&3.9600e-7&7.7442e-8&1.8910e-8&5.7307e-9&1.9807e-9&7.2870e-10 \\
\hline
\end{tabular}
\end{table}

\begin{example}\label{s6eg0-2}
Consider the following two-term nonlinear FODE
\begin{equation}\label{sec6:eq0-2}
{}_{C}D^{\alpha_1}_{0,t}Y(t)+{}_{C}D^{\alpha_2}_{0,t}Y(t)=Y(t)(1-Y^2(t))+\cos(t),{\quad}t{\,\in\,}
 {(0,T],T>0}\\
\end{equation}
subject to the initial condition $Y(0)=1/2$,   and $0<\alpha_1,\alpha_2<1$.
\end{example}
Here we consider three cases:
Case I:  $\alpha_1=0.7$, $\alpha_2=0.5$;
Case II:  $\alpha_1=0.2$, $\alpha_2=0.1$;
Case III:  $\alpha_1=0.16$, $\alpha_2=0.09$.


In this example, we compare our method with the two well-known methods. Specifically, we first apply  the L1 method to discretize the Caputo derivative in \eqref{sec6:eq0-2} to derive the corresponding numerical scheme, which is called the L1 method, see  \cite{{JinLaz-etal15},ZhangSunLiao14}. We also transform \eqref{sec6:eq0-2} into its integral form as $Y(t)+D^{-(\alpha_1-\alpha_2)}_{0,t}(Y(t)-Y(0))= Y(0)+D^{-(\alpha_1-\alpha_2)}_{0,t}(Y(t)(1-Y^2(t))+\cos(t))$, and then we apply the trapezoidal rule to the fractional integrals to obtain the corresponding numerical scheme, which is called the trapezoidal rule method, see \cite{DieFF04}.
Since we do not have the exact solution,  we obtain a  reference solution using the trapezoidal
rule method with  time step size $\tau= {T}/{2^{17}}$.  We also
apply the L1 method with  the step size $\tau=T/2^{17}$ to get  a reference solution and obtain almost the same results (not presented here).
In all computations, we use time step size  $\tau=T/2^{8}$  unless otherwise stated.

Obviously we do not know the regularity of   $Y(t)$ of Eq. \eqref{sec6:eq0-2} but we can   estimate the regularity from  linear  equations related to Eq. \eqref{sec6:eq0-2}. Here we first investigate  the regularity of the  following equation
\begin{equation}\label{sec6:eq0-2=linearized}
{}_{C}D^{\alpha_1}_{0,t}X(t)+{}_{C}D^{\alpha_2}_{0,t}X(t)= X(t) +\cos(t),{\quad}t{\,\in\,}
 (0,T].
\end{equation}
The solution  $X(t)$ can be represented by a generalized Mittag--Leffler function
\cite[Chapter 5.4]{Pod-B99} and   is of the form
 $\sum_{j,k=0}^\infty a_{j,k} t^{(\alpha_1-\alpha_2)j+(k+1)\alpha_1-1}$ if  $\alpha_1\geq \alpha_2$.
Meanwhile, we choose
the  correction terms  that make $S_m^{\sigma} = \Pi_{i=1}^m\abs{\sigma-\sigma_k}$ in \eqref{Rn2}   as small as possible if $Y(t)$ contains $t^{\sigma}$ when $\sigma$ is relatively small.
Empirically, we  choose   $\sigma_k<1$ that guarantees the decrease of     $S_m^{\sigma}$  with $m$ for $\sigma<1$.  Consequently, for all cases, we choose $\sigma_k$ as $\sigma_{k}=\alpha_1+(\alpha_1-\alpha_2)(k-1),\,k\geq 1$.

We show in Fig. \ref{fig5-2} numerical solutions and pointwise errors of both our method \eqref{fode3}   with/without  correction terms and the L1 method for Cases I and II. Adding correction terms  greatly improves accuracy, although these ``correction terms'' may not match the singularity of the analytical solutions.

We further test  choices of  $\sigma_k$ in Case III where fractional orders are small that  usually  lead to low regularity of   $Y(t)$. In addition to the choice of  $\sigma_{k}=\alpha_1+(\alpha_1-\alpha_2)(k-1)$,  we also choose $\sigma_k=0.1k,0.15k,0.2k$ in the computation.  All these choices of $\sigma_k$ yield smaller $S_m^{\sigma} = \Pi_{i=1}^m\abs{\sigma-\sigma_k}$ as $m$ increases when $\sigma_m\leq 1$ and $\sigma\leq 1$. In this case,  we choose $m\leq 5$ in the computation that leads to $\sigma_m\leq 1$.
The pointwise errors are shown in Fig. \ref{fig5-3}. We observe that numerical solutions  with higher accuracy are obtained by properly choosing  correction terms, which  confirms Lemma \ref{lem:quad-2}.


In conclusion,  we again observed that a smaller number of correction terms leads to  more accurate numerical solutions than  those from formulas without correction terms.   We also discussed how to choose $\sigma_k$ based on some preliminary singularity analysis and
error estimates in Lemma  \ref{lem:quad-2}.

\begin{figure}[!t]
\begin{center}
\begin{minipage}{0.4\textwidth}\centering
\epsfig{figure=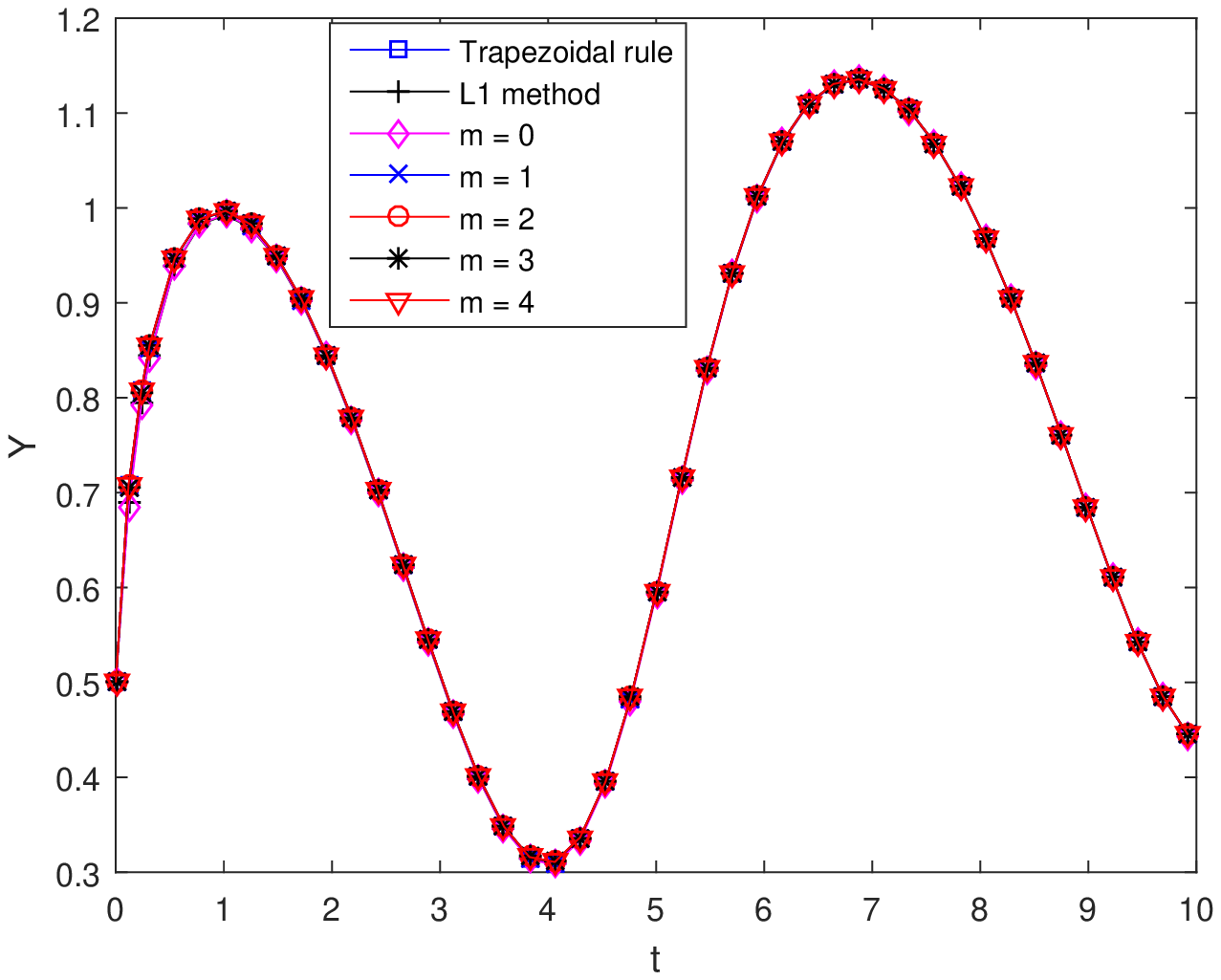,width=5cm}  \par{(a1)   Case I.}
\end{minipage}
\begin{minipage}{0.4\textwidth}\centering
\epsfig{figure=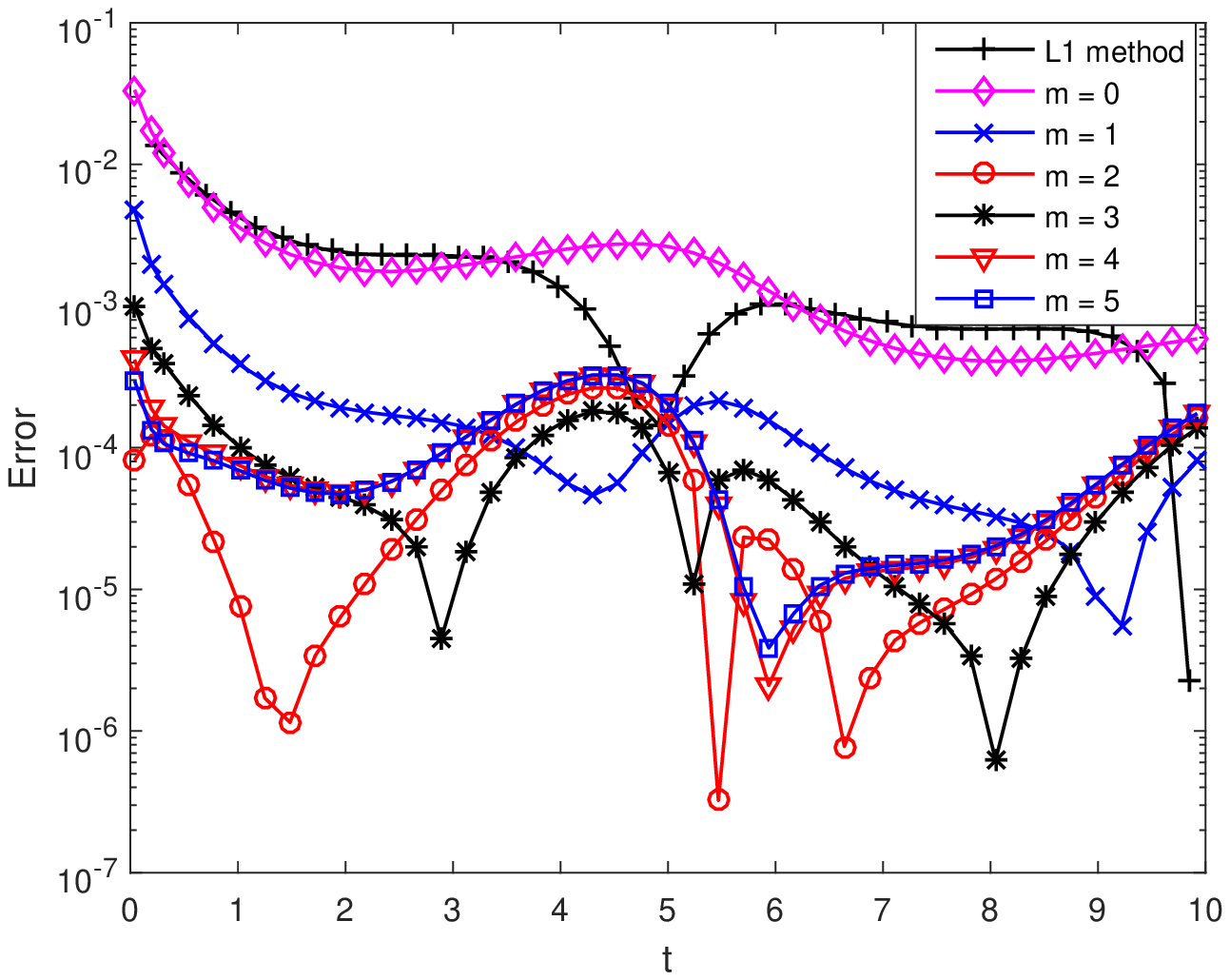,width=5cm}   \par{(a2) Case I.}
\end{minipage}
\begin{minipage}{0.4\textwidth}\centering
\epsfig{figure=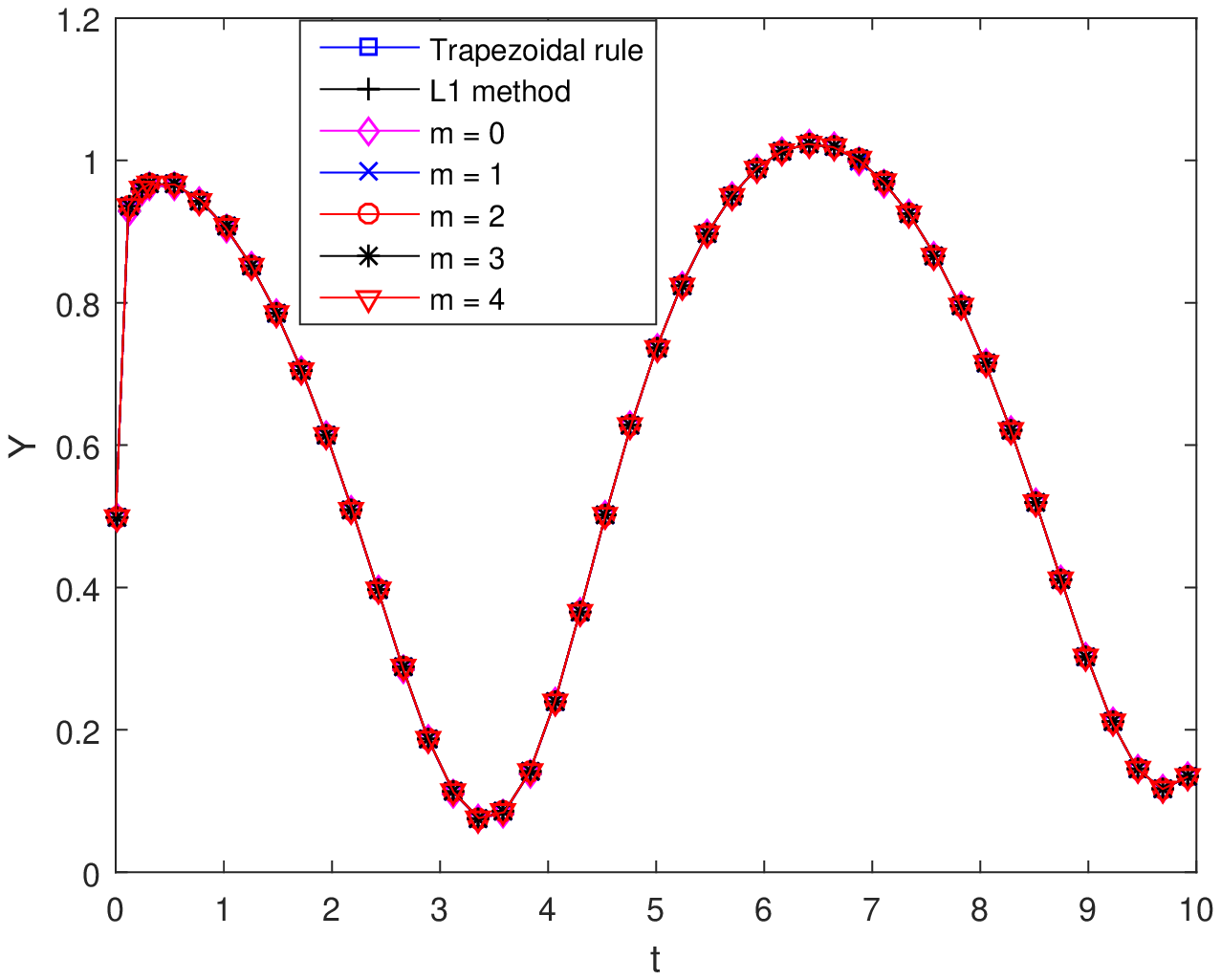,width=5cm}  \par{(b1)   Case II.}
\end{minipage}
\begin{minipage}{0.4\textwidth}\centering
\epsfig{figure=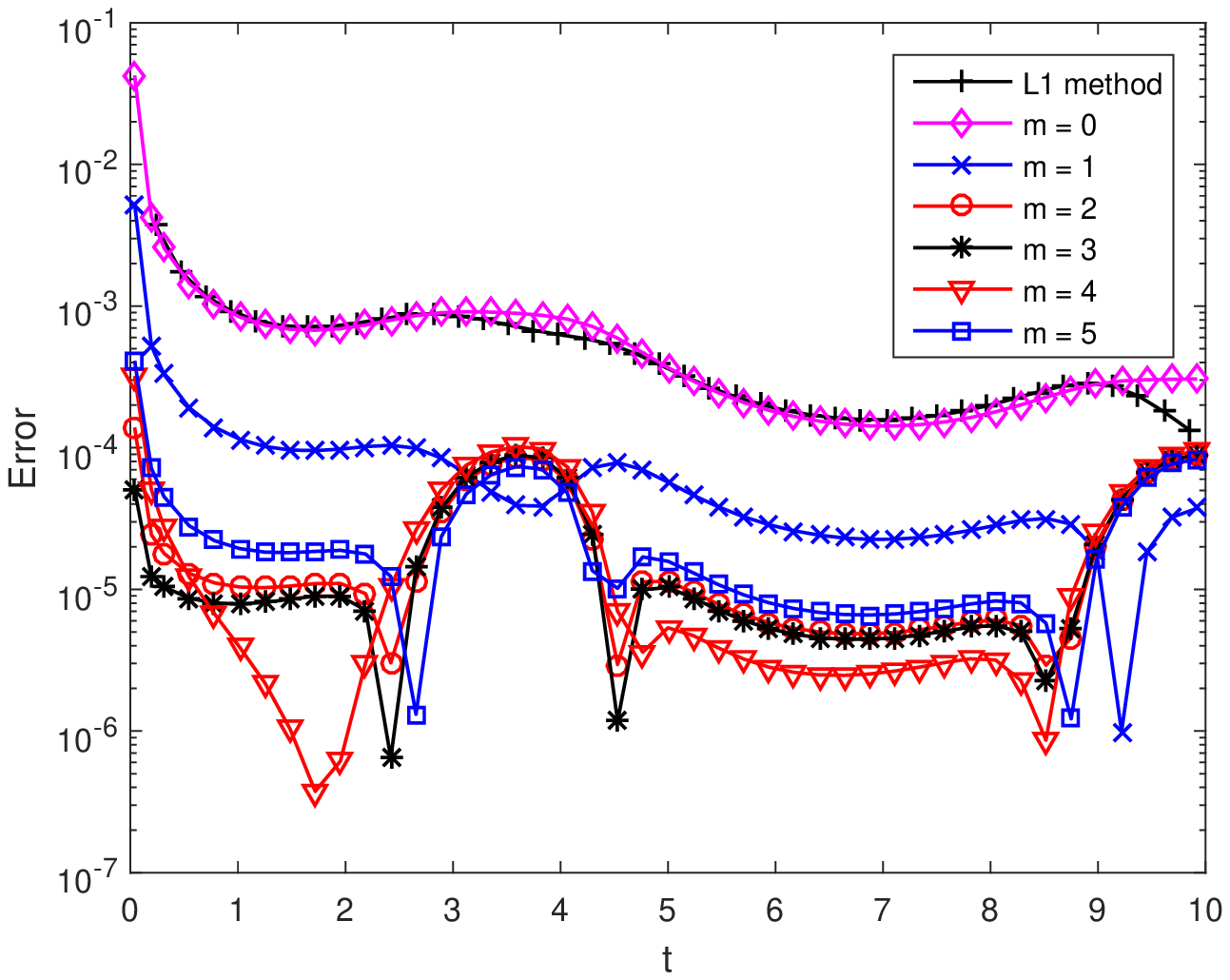,width=5cm}   \par{(b2) Case II.}
\end{minipage}
\end{center}
\caption{Numerical solutions and pointwise errors for Cases I and II, $\tau=T/2^{8}$, and $T=10$.\label{fig5-2}}
\end{figure}

\begin{figure}[!t]
\begin{center}
\begin{minipage}{0.4\textwidth}\centering
\epsfig{figure=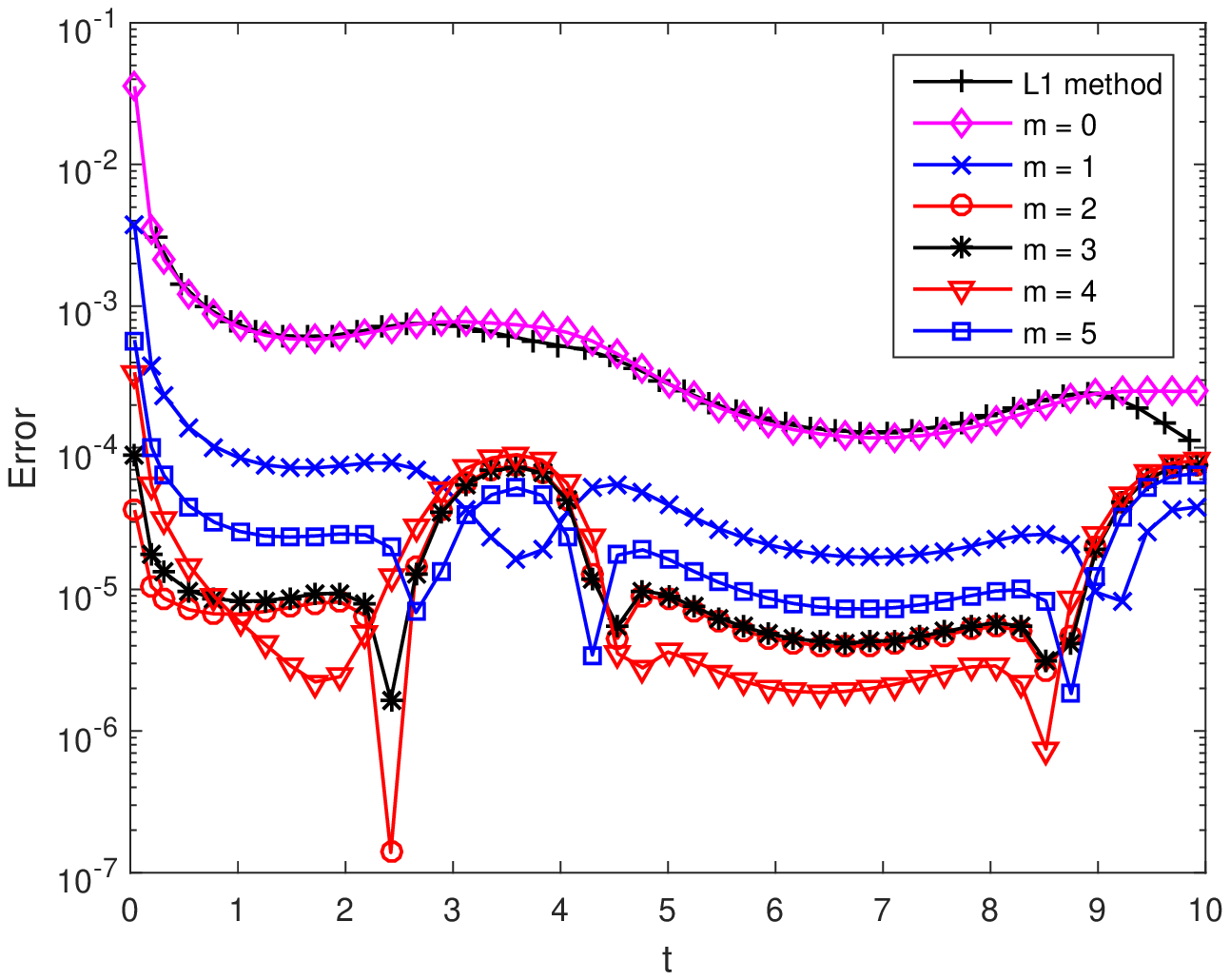,width=5cm}   \par{(a)  $\sigma_{k}=\alpha_1+\alpha_2(k-1)$.}
\end{minipage}
\begin{minipage}{0.4\textwidth}\centering
\epsfig{figure=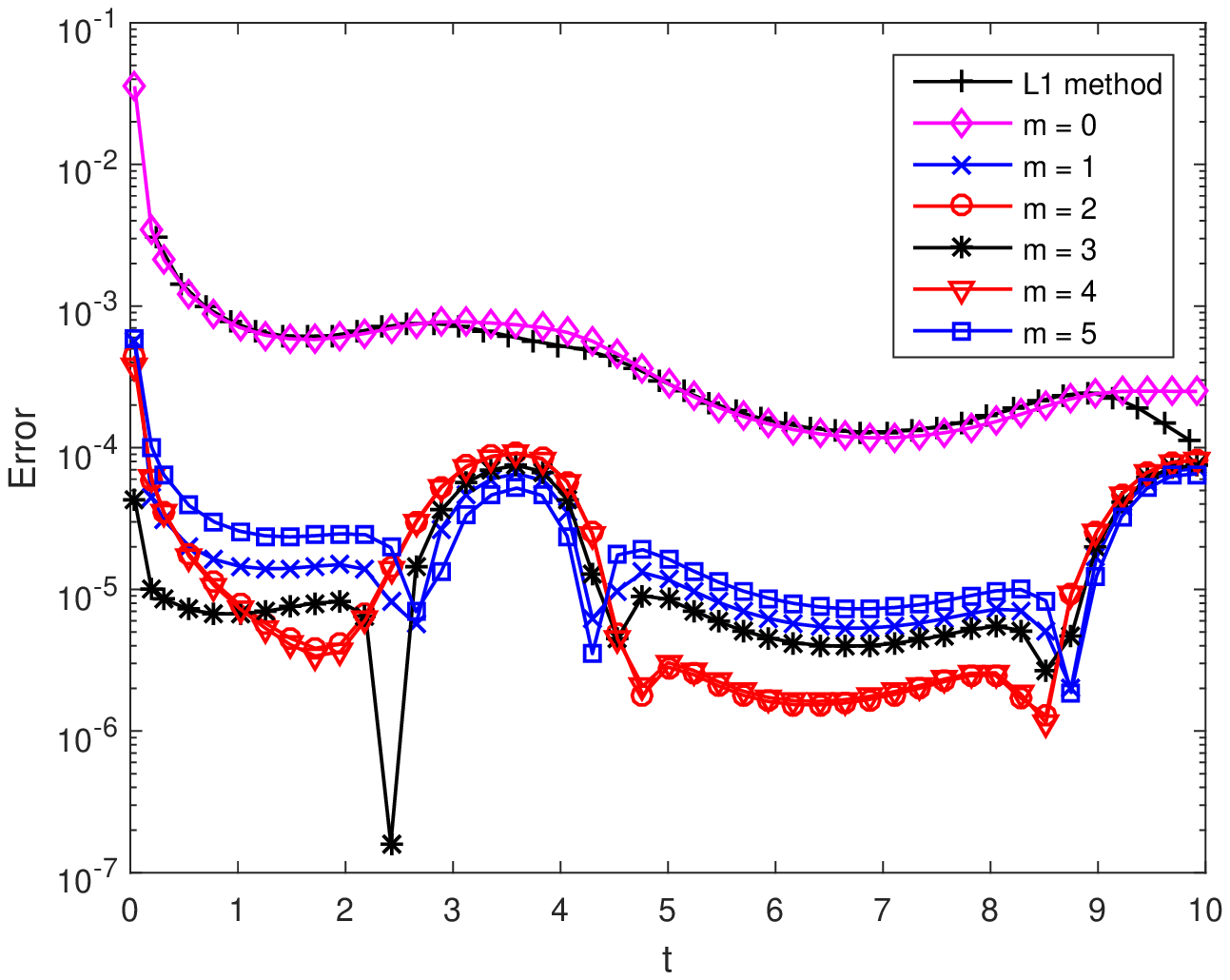,width=5cm}  \par{(b)   $\sigma_{k}=0.1k$.}
\end{minipage}
\begin{minipage}{0.4\textwidth}\centering
\epsfig{figure=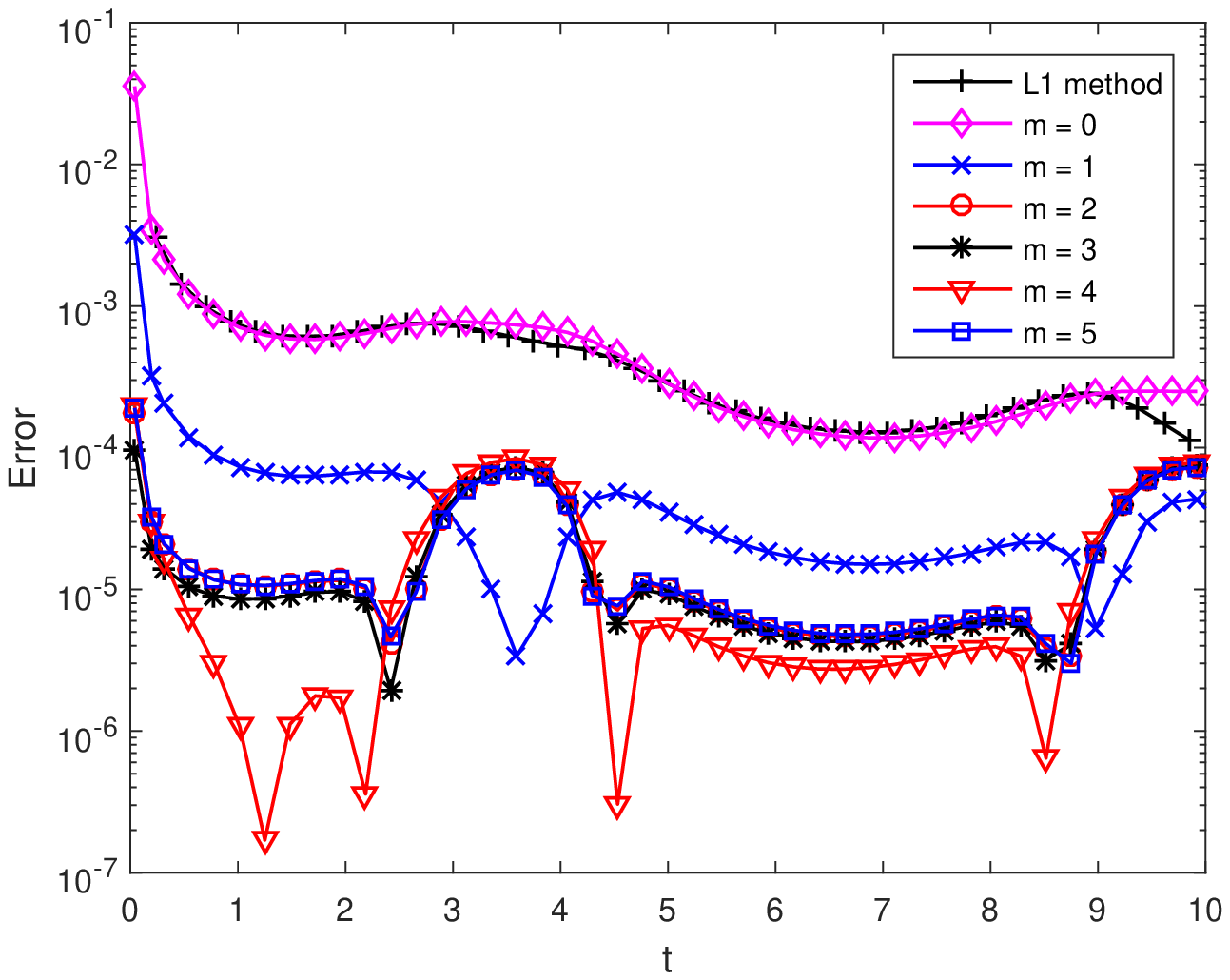,width=5cm}   \par{(c) $\sigma_{k}=0.15k$.}
\end{minipage}
\begin{minipage}{0.4\textwidth}\centering
\epsfig{figure=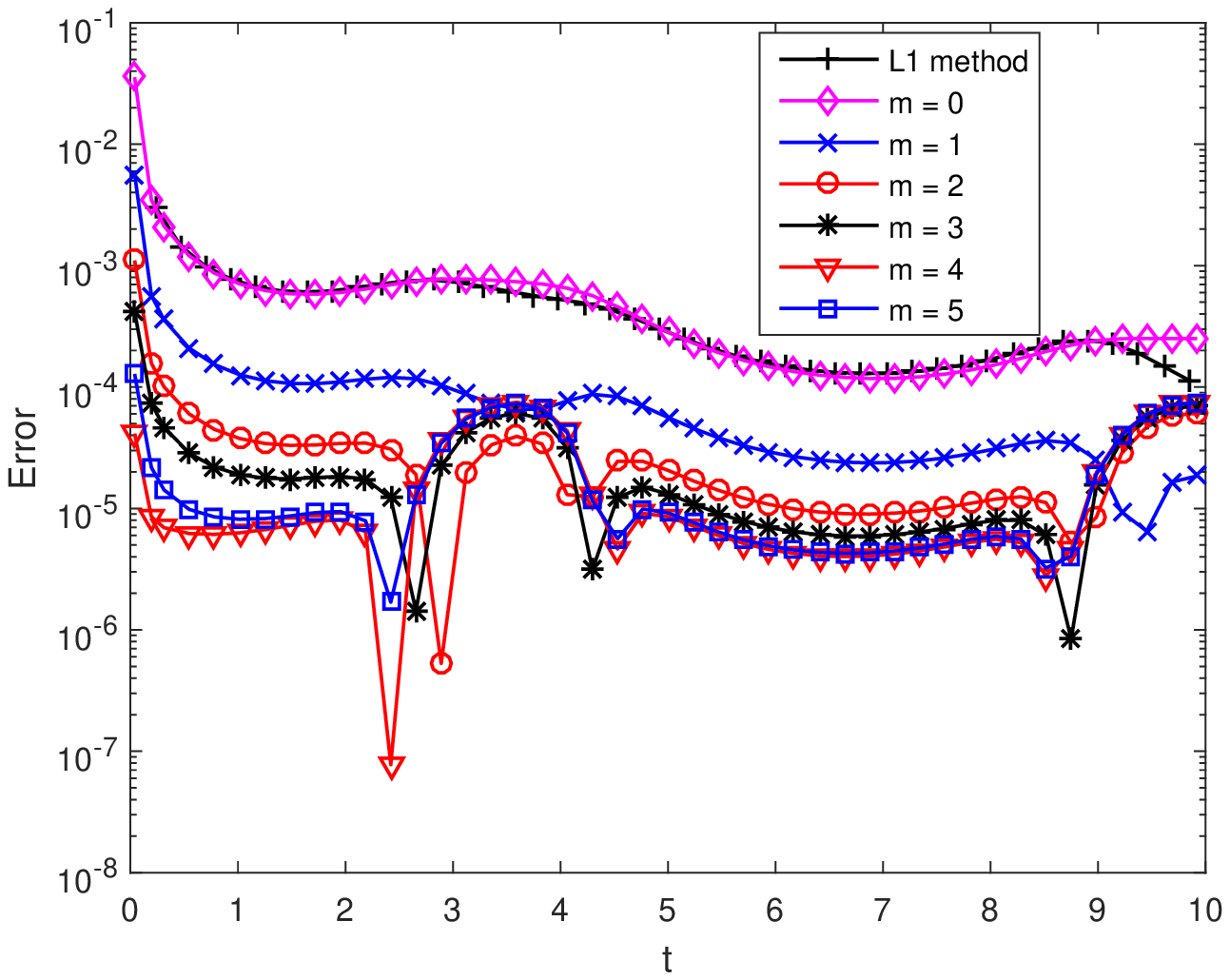,width=5cm}   \par{(d) $\sigma_{k}=0.2k$.}
\end{minipage}
\end{center}
\caption{Numerical solutions Pointwise errors  for Case III,  $\tau=T/2^{8}$, and $T=10$.\label{fig5-3}}
\end{figure}

\begin{example}\label{s6eg1}
Consider the following fractional diffusion-wave equation
\begin{equation}\label{sec6:eq1}
\left\{\begin{aligned}
&\px[t]^2U+{}_{C}D^{1+\alpha}_{0,t}U=\px^2U+f(x,t),{\quad}(x,t){\,\in\,}\Omega{\times}(0,1],\\
&U(x,0)=\sin(2\pi x),\quad\px[t]U(x,0)=\sin(2\pi x),{\quad}x{\,\in\,}\bar{\Omega},\\
&U(x,t)=0,{\quad}(x,t)\in\partial\Omega\times(0,1],
\end{aligned}\right.
\end{equation}
where $\Omega=(-1,1)$ and $0<\alpha<1$.
\end{example}
In this example, we consider two cases:
\begin{itemize}
  \item Case I (Smooth solutions): Choose a suitable source term $f(x,t)$
  such that  the exact solution  to  \eqref{sec6:eq1} is
$U(x,t)=(t^4+t^3+t^{2}+t+1)\sin(2\pi x).$
  \item Case II (Smooth inputs): The source term is $f=\exp(-t)\sin(\pi x)$,
  the initial conditions in
  \eqref{sec6:eq1} are  replaced by $U(x,0)=\px[t]U(x,0)=0$, and $\alpha=1/2$.
\end{itemize}

We use the LGSEM \eqref{scheme2-1}--\eqref{scheme2-3}  to solve \eqref{sec6:eq1}, where
the domain $\Omega$   is divided into three subdomains, i.e., $\Omega=(-1,-1/2]\cup[-1/2,1/2]\cup[1/2,1)$,
and $N=(24,32,24)$.

In Table \ref{tb1-1}, we present the  $L^2$ errors at $t=1$ for Case I (smooth solution
with $\sigma_k=k$ in \eqref{assu:fpde}). We choose $m_1=m_2=0$ and $m_3=0,1,2$
in the computation and observe that   second-order accuracy
is observed for $m_3=1,2$ and $\alpha=0.2,0.5,0.8,0.9$,
which is in line with or even better than the theoretical analysis,
see also Remark \ref{remark6}.
\textcolor[rgb]{0.00,0.00,0.00}{However,} for $m_3=0$,   we do not obtain
second-order accuracy, especially when $\alpha$ is close to 1; see  also the related
numerical results in \cite{WangVong14b,Zeng14},  {where second-order accuracy
was lost without correction terms}.

For Case II,
we do not have the explicit analytical solutions.
It is known, see e.g.   \cite[p. 183]{Diethelm-B10} and
\cite{JiangLiu-etal12b,JiangLiu-etal12,Luchko11},
that the analytical solution of \eqref{sec6:eq1} satisfies $U(x,t)=c_1(x)t^{\sigma_1}+c_2(x)t^{\sigma_2}+c_3(x)t^{\sigma_3}+....$,
where $\sigma_k=(3+k)\alpha,\,k=1,2,\ldots$ when $\alpha=1/2$.
For simplicity, we choose $m_1=m_2=m_3=m$ in the computation.
The benchmark solutions are obtained with smaller time step size $\tau=1/2048$.
In Table \ref{tb1-2}, we present the   $L^2$ errors at $t=1$ and observe second-order accuracy. Moreover,
we find that  more accurate numerical solutions are obtained when $m$ increases, which is in line with our theoretical predictions.

\begin{table}[!h]
\caption{{The   $L^{2}$ errors at $t=1$ for Case I, $N=(24,32,24)$, $m_1=m_2=0$.}}\label{tb1-1}
\centering\footnotesize
\begin{tabular}{|c|c|c|c|c|c|c|c|c|c|}
\hline
$m_3$& $\tau$ & $\alpha=0.2$ & Order& $\alpha=0.5$ & Order& $\alpha=0.8$ & Order & $\alpha=0.9$ & Order\\
 \hline
 &$2^{-5}$&2.6657e-4&    &6.1569e-4&    &1.7772e-3&    &2.4260e-3&    \\
 &$2^{-6}$&6.7681e-5&1.97&1.8086e-4&1.76&6.9840e-4&1.34&1.0494e-3&1.20\\
0&$2^{-7}$&1.7203e-5&1.97&5.4727e-5&1.72&2.8484e-4&1.29&4.6837e-4&1.16\\
 &$2^{-8}$&4.3805e-6&1.97&1.7044e-5&1.68&1.1915e-4&1.25&2.1310e-4&1.13\\
 &$2^{-9}$&1.1178e-6&1.97&5.4503e-6&1.64&5.0648e-5&1.23&9.8050e-5&1.11\\
 \hline
 &$2^{-5}$&1.6225e-4&    &2.5548e-4&    &4.2711e-4&    &5.1769e-4&    \\
 &$2^{-6}$&4.2531e-5&1.93&6.4545e-5&1.98&1.0434e-4&2.03&1.2298e-4&2.07\\
1&$2^{-7}$&1.0862e-5&1.96&1.6273e-5&1.98&2.6149e-5&1.99&3.0352e-5&2.01\\
 &$2^{-8}$&2.7426e-6&1.98&4.0901e-6&1.99&6.6081e-6&1.98&7.6092e-6&1.99\\
 &$2^{-9}$&6.8889e-7&1.99&1.0258e-6&1.99&1.6734e-6&1.98&1.9197e-6&1.98\\
 \hline
 &$2^{-5}$&2.5474e-4&    &4.6474e-4&    &8.2000e-4&    &9.9172e-4&    \\
 &$2^{-6}$&6.2411e-5&2.02&1.1320e-4&2.03&1.9873e-4&2.04&2.4016e-4&2.04\\
2&$2^{-7}$&1.5464e-5&2.01&2.7927e-5&2.01&4.8842e-5&2.02&5.8955e-5&2.02\\
 &$2^{-8}$&3.8497e-6&2.00&6.9349e-6&2.00&1.2104e-5&2.01&1.4598e-5&2.01\\
 &$2^{-9}$&9.6047e-7&2.00&1.7279e-6&2.00&3.0125e-6&2.00&3.6316e-6&2.00\\
\hline
\end{tabular}
\end{table}

\begin{table}[!h]
\caption{{The   $L^{2}$ errors at $t=1$ for Case II, $N=(24,32,24)$, $\alpha=1/2$.}}\label{tb1-2}
\centering\footnotesize
\begin{tabular}{|c|c|c|c|c|c|c|c|c|}
\hline
 $\tau$ & $m=0$ & Order& $m=1$ & Order& $m=2$ & Order & $m=3$ & Order\\
 \hline
$2^{-5}$&2.6290e-4&    &1.7419e-4&    &3.0941e-5&    &5.0036e-5&    \\
$2^{-6}$&8.8199e-5&1.57&5.6954e-5&1.61&6.1603e-6&2.32&9.3685e-6&2.41\\
$2^{-7}$&3.0182e-5&1.54&1.9353e-5&1.55&1.3292e-6&2.21&1.8037e-6&2.37\\
$2^{-8}$&1.0260e-5&1.55&6.5824e-6&1.55&3.0150e-7&2.14&3.6715e-7&2.29\\
$2^{-9}$&3.3070e-6&1.63&2.1279e-6&1.62&6.8463e-8&2.13&7.7103e-8&2.25\\
\hline
\end{tabular}
\end{table}

\section{Proofs}\label{sec:proof}


\subsection{Proof of Lemma \ref{lem:quad-accuracy} }A special case of Lemma \ref{lem:quad-accuracy} with $q=0,1$ and  $\sigma$
being an integer can be found in \cite{Sousa12}. We present the proof here for completeness and for all $q$.
\begin{proof}
For $q=0$ and $U(t)=t^{\sigma}$, we have
\begin{equation}\label{SGL-2}\begin{aligned}
\tau^{\alpha-\sigma}\mathcal{B}^{\alpha,n}_{0}U
=&\sum_{k=0}^{n}\omega^{(\alpha)}_{k}(n-k)^{\sigma}
=\sum_{k=0}^{n}\omega^{(\alpha)}_{n-k}k^{\sigma}.
\end{aligned}\end{equation}

By  analyzing the proof of Lemma 3.5 in \cite{Lub86} again, we can obtain
\begin{equation}\label{SGL-3}\begin{aligned}
\sum_{k=0}^{n}\omega^{(\alpha)}_{n-k}k^{\sigma}=\frac{\Gamma(\sigma+1)}{\Gamma(\sigma+1-\alpha)}n^{\sigma-\alpha}
-\frac{\alpha}{2}\frac{\Gamma(\sigma+1)}{\Gamma(\sigma-\alpha)}n^{\sigma-\alpha-1}
+O(n^{\sigma-\alpha-2}),
\end{aligned}\end{equation}
which yields  \eqref{SGL} for $q=0$.
Now we prove \eqref{SGL} for $q\neq 0$.  From \eqref{SGL-3}, we have
\begin{equation}\label{SGL-4}\begin{aligned}
&\tau^{\alpha-\sigma}\mathcal{B}^{\alpha,n}_{q}U
=\sum_{k=0}^{n+q}\omega^{(\alpha)}_{k}U(t_{n-k+q})
=\sum_{k=0}^{n+q}\omega^{(\alpha)}_{k}(n-k+q)^{\sigma}\\
=&\frac{\Gamma(\sigma+1)}{\Gamma(\sigma+1-\alpha)}(n+q)^{\sigma-\alpha}
-\frac{\alpha}{2}\frac{\Gamma(\sigma+1)}{\Gamma(\sigma-\alpha)}(n+q)^{\sigma-\alpha-1}
+O((n+q)^{\sigma-\alpha-2}).
\end{aligned}\end{equation}
By the fact that $\left(1+\sfrac{q}{n}\right)^{\sigma}=1+{\sigma}qn^{-1}+O(n^{-2}),n\geq|q|$,
we have
\begin{eqnarray}\label{SGL-5}
\sum_{k=0}^{n+q}\omega^{(\alpha)}_{k}(n-k+q)^{\sigma}
&=&\frac{\Gamma(\sigma+1)}{\Gamma(\sigma+1-\alpha)}
\left(n^{\sigma-\alpha}+{(\sigma-\alpha)}qn^{\sigma-\alpha-1}\right) \notag \\
&& -\frac{\alpha}{2}\frac{\Gamma(\sigma+1)}{\Gamma(\sigma-\alpha)}n^{\sigma-\alpha-1}+O(n^{\sigma-\alpha-2}) \\
&=&\frac{\Gamma(\sigma+1)}{\Gamma(\sigma+1-\alpha)}n^{\sigma-\alpha}
+\left(q-\frac{\alpha}{2}\right)\frac{\Gamma(\sigma+1)}{\Gamma(\sigma-\alpha)}n^{\sigma-\alpha-1}
+O(n^{\sigma-\alpha-2}). \notag
 \end{eqnarray}
From \eqref{SGL-3} and \eqref{SGL-5}, we reach \eqref{SGL}, which completes the proof.
\end{proof}

\subsection{Proof of Lemma \ref{lem:quad-2}}
To prove Lemma \ref{lem:quad-2}, we need the following lemma.
\begin{lemma}\label{lm4.2}
Let $m$ be a positive integer and  $w_{n,r}^{(\alpha)}\,(r=1,2,...,m)$  be defined by \eqref{s31-8}.
Suppose that $\{\sigma_r\}$ are a sequence of strictly increasing  positive numbers.
Then we have
\begin{equation}\label{s4:eq-1}\begin{aligned}
w_{n,r}^{(\alpha)}
= O(n^{\sigma_1-2-\alpha})+O(n^{\sigma_2-2-\alpha})+\cdots+O(n^{\sigma_m-2-\alpha}).
\end{aligned}\end{equation}
\end{lemma}
\begin{proof}
Letting $U(t)=t^{\sigma_r}$ and $(p,q)=(0,-1)$ in \eqref{s31-1}, we derive
\begin{equation}\label{s4:eq-2}\begin{aligned}
\frac{\Gamma(\sigma_r+1)}{\Gamma(\sigma_r+1-\alpha)}n^{\sigma_r-\alpha}
=\sum_{k=0}^{n}g^{(\alpha)}_{k}(n-k)^{\sigma_r}
+O({n}^{\sigma_r-2-\alpha}).
\end{aligned}\end{equation}
Combining \eqref{s4:eq-2} and \eqref{s31-8} yields a linear system
\begin{equation}\label{s4:eq-4}\begin{aligned}
w_{n,1}^{(\alpha)}&+2^{\sigma_r}w_{n,2}^{(\alpha)}+...+m^{\sigma_r}w_{n,m}^{(\alpha)}
=O(n^{\sigma_r-2-\alpha}),{\quad}r=1,2,...,m,
\end{aligned}\end{equation}
which leads to \eqref{s4:eq-1}. The proof is completed.
\end{proof}

Proof of Lemma \ref{lem:quad-2}.
\begin{proof}
We have
$\left[{}_{RL}D_{0,t}^{\alpha}U(t)\right]_{t=t_n}=\mathcal{A}_{p,q}^{\alpha,n,m}U
+\tau^{\sigma-\alpha}H^{n,m,\alpha}(\sigma)$
from \eqref{SGL} and \eqref{s31-1},
%
and $H^{n,m,\alpha}(\sigma)$  satisfies (see Lemma \ref{lm4.2})
\begin{equation*} \begin{aligned}
H^{n,m,\alpha}(\sigma)=&\frac{\Gamma(\sigma+1)}{\Gamma(\sigma+1-\alpha)}n^{\sigma-\alpha}
-\sum_{k=0}^{n}g^{(\alpha)}_{k}(n-k)^{\sigma} -\sum_{k=1}^mw_{n,k}^{(\alpha)}k^{\sigma}\\
=&Gn^{\sigma-\alpha-2}-\sum_{k=1}^mG_kn^{\sigma_k-\alpha-2}
+{H}_2^{n,m,\alpha}(\sigma)={H}_1^{n,m,\alpha}(\sigma)+{H}_2^{n,m,\alpha}(\sigma),
\end{aligned}\end{equation*}
where $G$ is independent of $n,\tau$, and $\sigma_k(1\leq k \leq m)$,  $G_k$ is independent of $n$ and $\tau$,   and ${H}_2^{n,m,\alpha}(\sigma)=O({n}^{\sigma-3-\alpha})
+O({n}^{\sigma_1-3-\alpha})+\cdots + O({n}^{\sigma_m-3-\alpha})$. From  Eqs. (3.1) and (3.13) in \cite{Lub86}, we can readily obtain both $G$ and $G_k$ are analytical functions with respect to $\sigma$.
Hence, $|\px[\sigma]^rG_k|$ and $|\px[\sigma]^rG|$ are bounded for a given $r$.
Next, we derive a bound for
${H}_1^{n,m,\alpha}(\sigma)$. Since $H^{n,m,\alpha}(\sigma_k)=0$, it implies
${H}_1^{n,m,\alpha}(\sigma_k)=0$, $1\leq k \leq m$. It is known that ${H}_1^{n,m,\alpha}(\sigma)$ is infinitely smooth with respect to $\sigma$ and ${H}_1^{n,m,\alpha}(\sigma_k)=0$, there exists an $\xi(\sigma)\in(\min\{\sigma,\sigma_1,...,\sigma_m\},\max\{\sigma,\sigma_1,...,\sigma_m\})$,
such that
$${H}_1^{n,m,\alpha}(\sigma)=\frac{\px[\sigma]^{m}{H}_1^{n,m,\alpha}(\xi(\sigma))}{m!}
\prod_{k=1}^m(\sigma-\sigma_k).$$
From the boundedness of $|\px[\sigma]^r G_k|$, $|\px[\sigma]^r G|$ , and the following relation
\begin{equation*} \begin{aligned}
|\px[\sigma]^{m} (Gn^{\sigma-\alpha-2})|
=&\bigg|\sum_{k=0}^m\binom{k}{m}(\px[\sigma]^{m-k}G)\, n^{\sigma-\alpha-2}\log^{k}(n)\bigg|
\leq Cn^{\sigma-\alpha-2}\log^{m}(n),
\end{aligned}\end{equation*}
we have
$|\px[\sigma]^{m}{H}_1^{n,m,\alpha}(\sigma)|
\leq C\left(n^{\sigma-\alpha-2}\log^{m}(n)
+n^{\sigma_{\max}-\alpha-2}\right)$, which leads to
$$|{H}_1^{n,m,\alpha}(\sigma)|\leq CS_m^{\sigma}\left(n^{\sigma-\alpha-2}\log^{m}(n)
+n^{\sigma_{\max}-\alpha-2}\right).$$

With the same reasoning, we can derive that ${H}_2^{n,m,\alpha,\sigma}$ also satisfies
\begin{equation*} \begin{aligned}
|{H}_2^{n,m,\alpha}(\sigma)|\leq& \widetilde{C}\left[S_m^{\sigma}\left(n^{\sigma-\alpha-3}\log^{m}(n)
+n^{\sigma_{\max}-\alpha-3}\right)+ {n}^{\sigma_{\max}-4-\alpha}\right]\\
\leq& C\left[S_m^{\sigma}\left(n^{\sigma-\alpha-3}\log^{m}(n)
+n^{\sigma_{\max}-\alpha-3}\right)+ {n}^{\sigma_{\max}-3-d-\alpha}\right],
\end{aligned}\end{equation*}
where $d$ is a positive integer and $C$ is independent of $n$.
From the estimates of ${H}_1^{n,m,\alpha}(\sigma)$ and ${H}_2^{n,m,\alpha}(\sigma)$, we obtain the desired result. The proof is completed.
\end{proof}

\subsection{Proofs of Theorems \ref{thm3-2}--\ref{theorem3-3}}
We first introduce a lemma.
\begin{lemma}\label{lm5-2}
Suppose that $a(z)=\sum_{k=0}^{\infty}a_kz^k$, $b(z)=\sum_{k=0}^{\infty}b_kz^k$, $\alpha$ is real,
 and $0\leq \beta\leq 1$.
If $\abs{a_k}\leq Ck^{\alpha}$ and $a(z)\left(1+\frac{\beta}{2}-\frac{\beta}{2}z\right)^{-1}=b(z)$,
then there exists a positive constant  $C$ independent of $k$, such that $|b_k|\leq C k^{\alpha}$.
\end{lemma}
\begin{proof}
From $a(z)\left(1+\frac{\beta}{2}-\frac{\beta}{2}z\right)^{-1}=b(z)$, we have
$b_0=a_0\left(1+\frac{\beta}{2}\right)^{-1}$ and
$a_k=\left(1+\frac{\beta}{2}\right)b_{k}-\frac{\beta}{2}b_{k-1},k>0$, which leads to
$b_k=\frac{\beta}{2+\beta}b_{k-1}+\frac{2}{2+\beta}a_k$.
Since $a_k=O(k^{\alpha})$, there exist a positive constant $C_1$ independent of $k$ such that
$|a_k|\leq C_1k^{\alpha}$. With the mathematical induction method
and $C_1\leq C/2$, we   complete  the proof.
\end{proof}

From Lemma \ref{lm5-2} and Lemma 3.3 in \cite{ZengLLT15}, one can   derive the following corollary.
\begin{corollary}\label{crlr5-3}
Assume that $0\leq \beta_2\leq 1$, $\alpha\leq 1$, and $\beta_1$ is real.
If $a_n=O(n^{\alpha})$ and
 $(1-z)^{\beta_1}\left(1+\frac{\beta_2}{2}-\frac{\beta_2}{2}z\right)^{-1}
=\sum_{k=0}^{\infty}\tilde{g}_kz^k$,
then $\abs{\sum_{k=0}^n\tilde{g}_ka_{n-k}}\leq C n^{\alpha-\beta_1}$.
\end{corollary}

\subsubsection{Proof of Theorem \ref{thm3-2}}
\begin{proof}
Following the similar idea in \cite{Lub86b} (see also \cite{Zeng14}),
we derive that the  method \eqref{fode3}  is stable if $-\lambda\in\mathbb{S}$,
where $\mathbb{S}$ is defined by
$\mathbb{S}=\mathbb{C}\setminus
 \big\{\sum_{j=1}^Q\nu_j\tau^{-\alpha_j}g^{(\alpha_j)}(z),|z|\leq1,\tau>0\big\}$,
in which
$g^{(\alpha_j)}(z)=(1-z)^{\alpha_j}\left(1+\frac{\alpha_j}{2}-\frac{\alpha_j}{2}z\right)$.
Since $\text{Re}(g^{(\alpha_j)}(z))\geq 0$ for all $|z|\leq 1,\tau>0$ and $0< \alpha_j\leq 1$,
the  region $\mathbb{S}$ contains the whole negative real line, i.e., the scheme
\eqref{fode3}  is stable for any $\tau>0$. This completes the proof.
\end{proof}

\subsubsection{Proof of Theorem \ref{thm3-1}}
\begin{proof}
Let $e^n=Y(t_n)-y^n$.
Then we derive the error equation of \eqref{fode3}  as
\begin{eqnarray}
&&\sum_{j=1}^Q\frac{\nu_j}{\tau^{\alpha_j}}\left[\sum_{k=0}^ng^{(\alpha_j)}_{n-k}e^k
+\sum_{r=1}^{m_j}w^{(\alpha_j)}_{n,r}e^r\right] = f(t_n,Y(t_n))-f(t_n,y^n)+R^n,\label{fode4}
\end{eqnarray}
where $R^n=\sum_{j=1}^QO(\tau^2t_{n}^{\sigma_{m_j+1}-2-\alpha_j})$ is defined in \eqref{fode2}.
By  Lemma 3.4 in \cite{ZengLLT13},
Eq. \eqref{fode4} can be written in the following form
\begin{eqnarray}
&&\nu_1e^n+\sum_{j=2}^Q{\nu_j}{\tau^{\alpha_1-\alpha_j}}\left[\sum_{k=0}^nG^{(j)}_{n-k}e^k +\sum_{r=1}^{m_j}\sum_{k=0}^n\tilde{g}_{n-k}w^{(\alpha_j)}_{k,r}e^r \right]\nonumber \\
&=& {\tau^{\alpha_1}}\sum_{k=0}^n\tilde{g}_{n-k}(f(t_k,Y(t_k))-f(t_k,y^k))
-\nu_1\sum_{r=1}^{m_1}\sum_{k=0}^n\tilde{g}_{n-k}w^{(\alpha_1)}_{k,r}e^r+\tilde{R}^n,\label{fode5}
\end{eqnarray}
where $\tilde{R}^n={\tau^{\alpha_1}}\sum_{k=0}^n\tilde{g}_{n-k}R^k$, and
$\{\tilde{g}_k\}$ and  $\{G^{(j)}_{k}\}$  are, respectively, the coefficients of the
Taylor expansions of the  functions
$\tilde{g}(z)=(1-z)^{-\alpha_1}
\left(1+\frac{\alpha_1}{2}-\frac{\alpha_1}{2}z\right)^{-1}$  and
$$G^{(j)}(z)=(1-z)^{-(\alpha_1-\alpha_j)}
\left(1+\frac{\alpha_j}{2}-\frac{\alpha_j}{2}z\right)
\left(1+\frac{\alpha_1}{2}-\frac{\alpha_1}{2}z\right)^{-1}.$$

Assume that $f(t,Y)$ satisfies the uniform  Lipschitz condition, i.e., there exists a positive
constant $L$ such that
$|f(t_k,Y(t_k))-f(t_k,y^k)|\leq L|Y(t_k)-y^k|=L|e^k|$. Hence, for small enough $\tau$,
we have
\begin{eqnarray}
&&|e^n|\leq C\sum_{j=2}^Q{\tau^{\alpha_1-\alpha_j}}\sum_{k=0}^{n-1}G^{(j)}_{n-k}|e^k|
+C{\tau^{\alpha_1}}\sum_{k=0}^{n-1}\tilde{g}_{n-k}|e^k| +\rho^n,\label{fode7}
\end{eqnarray}
where
$\rho^n=C\sum_{j=1}^Q\sum_{r=1}^{m_j}\tau^{\alpha_1-\alpha_j}
\Big|\sum_{k=0}^n\tilde{g}_{n-k}w^{(\alpha_j)}_{k,r}\Big||e^r|+C|\tilde{R}^n|.$


It is known that $(1-z)^{\alpha}=\sum_{k=0}^{\infty}\omega_k^{(\alpha)}z^k$ with $\omega_k^{(\alpha)}=O(k^{-\alpha-1})$.
It then follows from  Corollary \ref{crlr5-3} that $G^{(j)}_{k}=O(k^{\alpha_1-\alpha_j-1})$ and $\tilde{g}_{k}=O(k^{\alpha_1-1})$.
We can similarly obtain
$|\sum_{k=0}^n\tilde{g}_{n-k}w^{(\alpha_j)}_{k,r}|\leq Cn^{\sigma_{m_j}-2-\alpha_j+\alpha_1}$
and
$|\tilde{R}^n|\leq C \tau^2\sum_{j=1}^{Q}t_n^{\sigma_{m_j+1}-2-\alpha_j+\alpha_1}$.
Hence,
$\rho^n\leq  C\left(\sum_{k=0}^{\max_{1\leq j\leq Q}\{m_j\}}|e^k|
+\sum_{j=1}^{Q}n^{\sigma_{m_j+1}-2-\alpha_j+\alpha_1}\tau^{\sigma_{m_j+1}-\alpha_j+\alpha_1}\right)$.
Similar to the proof in   \cite[Lemma 3.4]{LiZeng13},  we derive \eqref{fode6}.

If $f(t,Y)$ satisfies the local Lipschitz condition with respect to $Y$, then we can
also derive \eqref{fode6} by the mathematical induction method, which is omitted here,
see \cite{LiZeng13}. This ends the proof.
\end{proof}

\subsubsection{Proof of Theorem \ref{theorem3-3}}
We first introduce a lemma.
\begin{lemma}[\cite{WangVong14b}]\label{lm4.1}
Suppose that $-1<\alpha \leq 1$. Then we have
$$\sum_{n=0}^K\left(\sum_{k=0}^ng^{(\alpha)}_{k}v_{n-k}\right)v_{n}
\geq 0,$$
where $\{v_k\}$ are real numbers and $g^{(\alpha)}_{k}$ is defined by \eqref{g-k}.
\end{lemma}
\begin{proof}
From \eqref{fode4}, we have
\begin{eqnarray}
&&\sum_{j=1}^Q\frac{\nu_j}{\tau^{\alpha_j}}\left[\sum_{k=0}^ng^{(\alpha_j)}_{n-k}e^k
+\sum_{r=1}^{m_j}w^{(\alpha_j)}_{n,r}e^r\right] + e^n=R^n,\label{apx-fode4}
\end{eqnarray}
where $|R^n|\leq C\tau^2\sum_{j=1}^Qt_{n}^{\sigma_{m_j+1}-2-\alpha_j}$ is defined in \eqref{fode2}.

Multiplying $e^n$ on both sides of \eqref{apx-fode4} and summing up $n$ from $1$ to $K$, one has
\begin{eqnarray}
&&\sum_{j=1}^Q\frac{\nu_j}{\tau^{\alpha_j}}\left[\sum_{n=1}^K\sum_{k=1}^ng^{(\alpha_j)}_{n-k}e^ke^n
+\sum_{r=1}^{m_j}\sum_{n=1}^Kw^{(\alpha_j)}_{n,r}e^re^n\right] + \sum_{n=1}^K|e^n|^2
=\sum_{n=1}^KR^ne^n.\label{apx-fode4-2}
\end{eqnarray}
Applying Lemma \ref{lm4.1} and  the Cauchy--Schwarz inequality, we have
\begin{eqnarray}
&&\sum_{n=1}^K|e^n|^2
\leq\sum_{j=1}^Q\frac{\nu_j}{\tau^{\alpha_j}}\sum_{n=1}^K\sum_{k=1}^ng^{(\alpha_j)}_{n-k}e^ke^n
 + \sum_{n=1}^K|e^n|^2\nonumber\\
&\leq&\sum_{n=1}^KR^ne^n-\sum_{j=1}^Q\frac{\nu_j}{\tau^{\alpha_j}} \sum_{r=1}^{m_j}\sum_{n=1}^Kw^{(\alpha_j)}_{n,r}e^re^n\nonumber\\
&\leq&\frac{1}{2}\sum_{n=1}^K(|R^n|^2+|e^n|^2)+
\sum_{j=1}^Q \sum_{r=1}^{m_j}\sum_{n=1}^K
\bigg(\frac{|w^{(\alpha_j)}_{n,r}|^2\nu_j^2\tau^{-2\alpha_j}}{4\epsilon}
|e^r|^2+\epsilon|e^n|^2\bigg),\label{apx-fode4-3}
\end{eqnarray}
where $\epsilon>0$ is a positive constant. If we choose $\epsilon=\frac{1}{4Q(m_1+m_2+...+m_Q)}$,
then we have
\begin{eqnarray}
\sum_{n=1}^K|e^n|^2
&\leq&C\sum_{n=1}^K|R^n|^2+C\sum_{j=1}^Q \sum_{r=1}^{m_j}
\sum_{n=1}^K{\tau^{-2\alpha_j}|w^{(\alpha_j)}_{n,r}|^2}|e^r|^2.\label{apx-fode4-4}
\end{eqnarray}
Let $q_j=\sigma_{m_j+1}-\alpha_j$. From \eqref{apx-fode4-4}, \eqref{s4:eq-1}, and
$|R^n|\leq C\tau^2\sum_{j=1}^Qt_{n}^{\sigma_{m_j+1}-2-\alpha_j}$, we have
\begin{eqnarray}
\tau\sum_{n=1}^K|e^n|^2
&\leq& C\tau\sum_{n=1}^K|R^n|^2+C\sum_{j=1}^Q \sum_{r=1}^{m_j}
\sum_{n=1}^K{\tau^{1-2\alpha_j}|w^{(\alpha_j)}_{n,r}|^2}|e^r|^2\nonumber\\
&\leq&C\sum_{j=1}^Q\tau^{2q_j+1}\sum_{n=1}^Kn^{2(q_j-2)}
+C\max_{1\leq r \leq m}|e^r|^2
\bigg(\sum_{j=1}^Q\tau^{1-2\alpha_j}\sum_{n=1}^Kn^{{2(\sigma_m-\alpha_j-2)}}\bigg)\nonumber\\
&\leq&C\sum_{j=1}^Q\bigg(\tau^{2q_j+1}K^{\max\{0,2q_j-3\}}
+\max_{1\leq r \leq m}|e^r|^2\tau^{1-2\alpha_j}K^{\max\{0,2\sigma_m-2\alpha_j-3\}}\bigg).\nonumber
\label{apx-fode4-5}
\end{eqnarray}
Let $q=\min_{1\leq j\leq Q}\{q_j\}$.  Then we have from the above equation
\begin{eqnarray}
\tau\sum_{n=1}^K|e^n|^2
&\leq&C\tau^{\min\{4,2q+1\}}
+C\max_{1\leq r \leq m}|e^r|^2
\sum_{j=1}^Q\tau^{1-2\alpha_j}K^{\max\{0,2\sigma_m-2\alpha_j-3\}},
\nonumber
\end{eqnarray}
which leads to \eqref{fode6-2}. The proof is completed.
\end{proof}

\section{Summary}\label{concl}

In this work, we obtained the asymptotic expansion of the error equation of the
WSGL formula \eqref{s31-1-2} proposed in \cite{TianZD14}. The WSGL formula is second-order accurate far from $t=0$  but is not second-order accurate  near $t=0$. Hence  second-order numerical scheme is not expected when the formula is
applied to   time-fractional differential equations.
We then followed Lubich's approach  by adding the correction terms
to the  WSGL formula  and obtained a modified formula
with global second-order accuracy  both around and far from $t=0$.
We  applied our modified formula   to solve  two-term
FODEs and   two-term time-fractional anomalous diffusion equations
and proved the stability and second-order convergence in time.

We found that only a small number of correction terms  is  needed to improve
convergence order and accuracy regardless of regularity of the analytical solutions. We showed both theoretically and numerically that
a few correction terms are sufficient to obtain  relatively high accuracy at $t=0$ and
thus improve the convergence order and accuracy far from $t=0$.  With a few correction terms, we avoid solving the   linear system  with an exponential Vandermonde matrix of large size to obtain starting weights. We observed that with no more than $10$ terms, the linear system can be accurately solved with double precision without harming the accuracy of the second-order formula. Moreover, the correction terms do not have to  exactly match
the singularity indexes of solutions to FDEs. Even when we do not know the precise singularity indices of solutions to FDEs, we provided some empirical guidelines to choose correction terms.

Although, we only focus on the  WSGL formulas,
the methodology proposed here can be also applied to recover  globally high accuracy for some other  high-order formulas, see e.g.
\cite{CelikDuman12,ChenDeng14,DingLC14a,DingLC14b,GaoSS15,ZengLLT13,ZhaoSH14,ZhouTD13}, where the high-order accuracy requires
vanishing initial/boundary values of the corresponding function, even   vanishing values of  higher-order derivatives at boundaries.


\newpage
\pagenumbering{roman}

\setcounter{section}{0}
\setcounter{thm}{0}
\def\thesection{\alph{section}} 
\section{Supplementary Material}
We provide more numerical results to support the theoretical analysis and the proofs of Theorems \ref{thm:stability}--\ref{thm:convergence} in Subsection \ref{sec4-3}.

\subsection{Further investigation of the upper bound of \eqref{Rn2}} Numerically, we find  a much better upper bound of \eqref{Rn2}, which is presented below
\begin{equation}\label{Rn2-2}
|{R_n}|\leq C_mS_m^{\sigma} \tau^{\sigma-\alpha-2} n^{\max\{\sigma_{\max},\sigma\}-\alpha-2}.
\end{equation}

We  plot a bound of $C_m$ in \eqref{Rn2-2} for $U(t)=t^{0.85}$, which is shown in Fig. \ref{fig2-3}. We  see that $C_m$ does not increase with respect to $m$  and $\alpha$.

\begin{figure}[!ht]
\begin{center}
\begin{minipage}{0.45\textwidth}\centering
\epsfig{figure=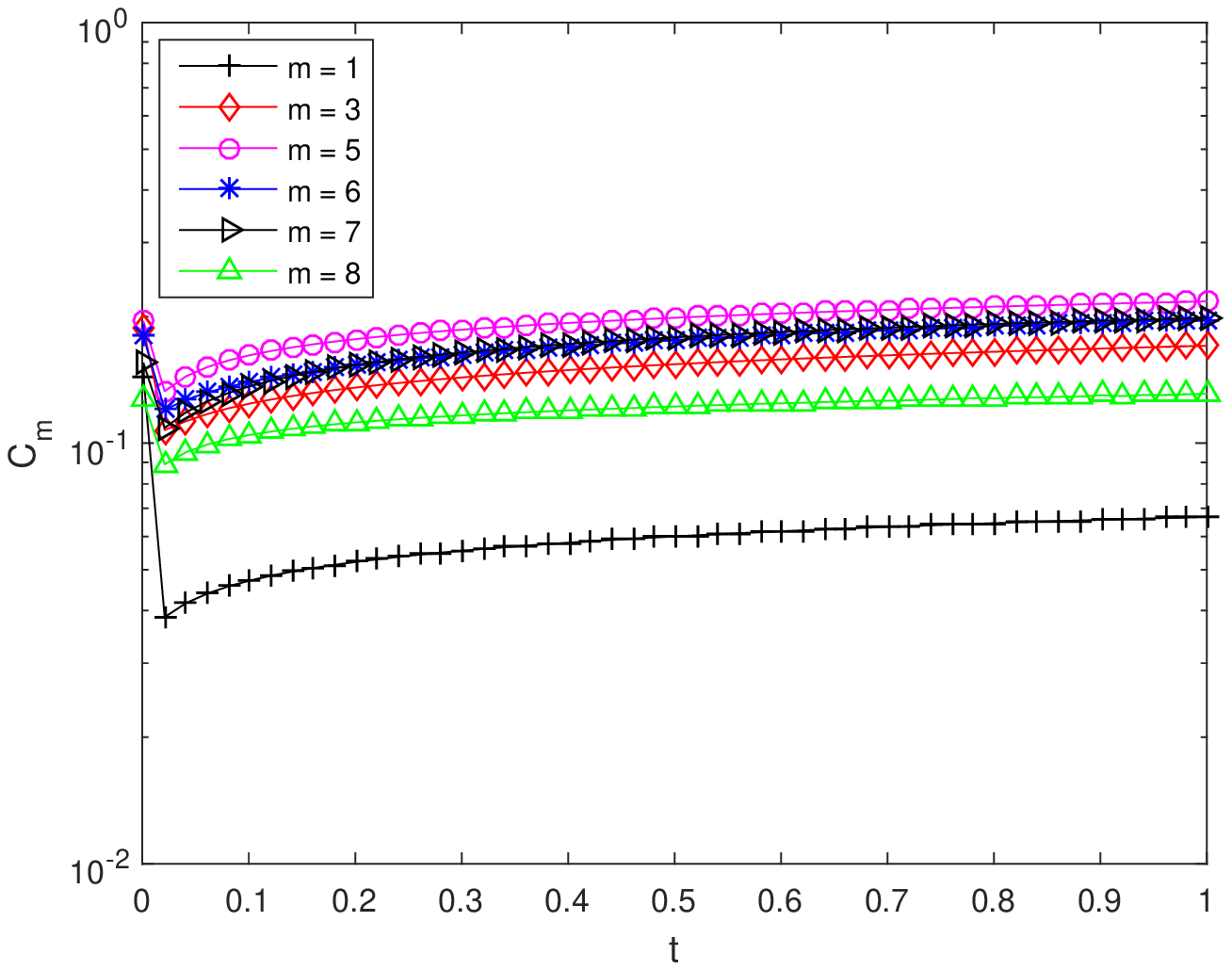,width=5cm}  \par{(a)   $\alpha=0.06$.}
\end{minipage}
\begin{minipage}{0.45\textwidth}\centering
\epsfig{figure=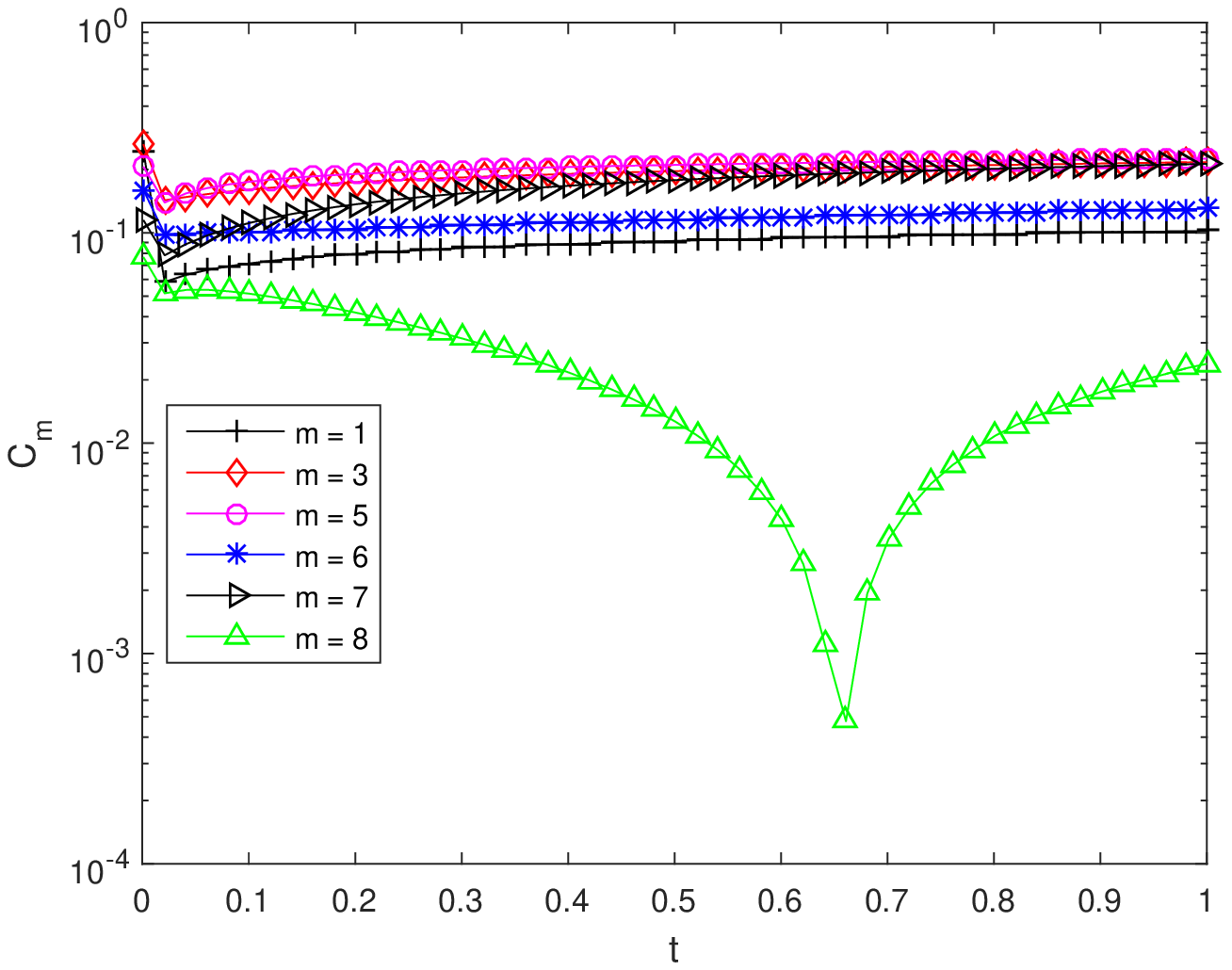,width=5cm}   \par{(b)  $\alpha=0.1$.}
\end{minipage}
\begin{minipage}{0.45\textwidth}\centering
\epsfig{figure=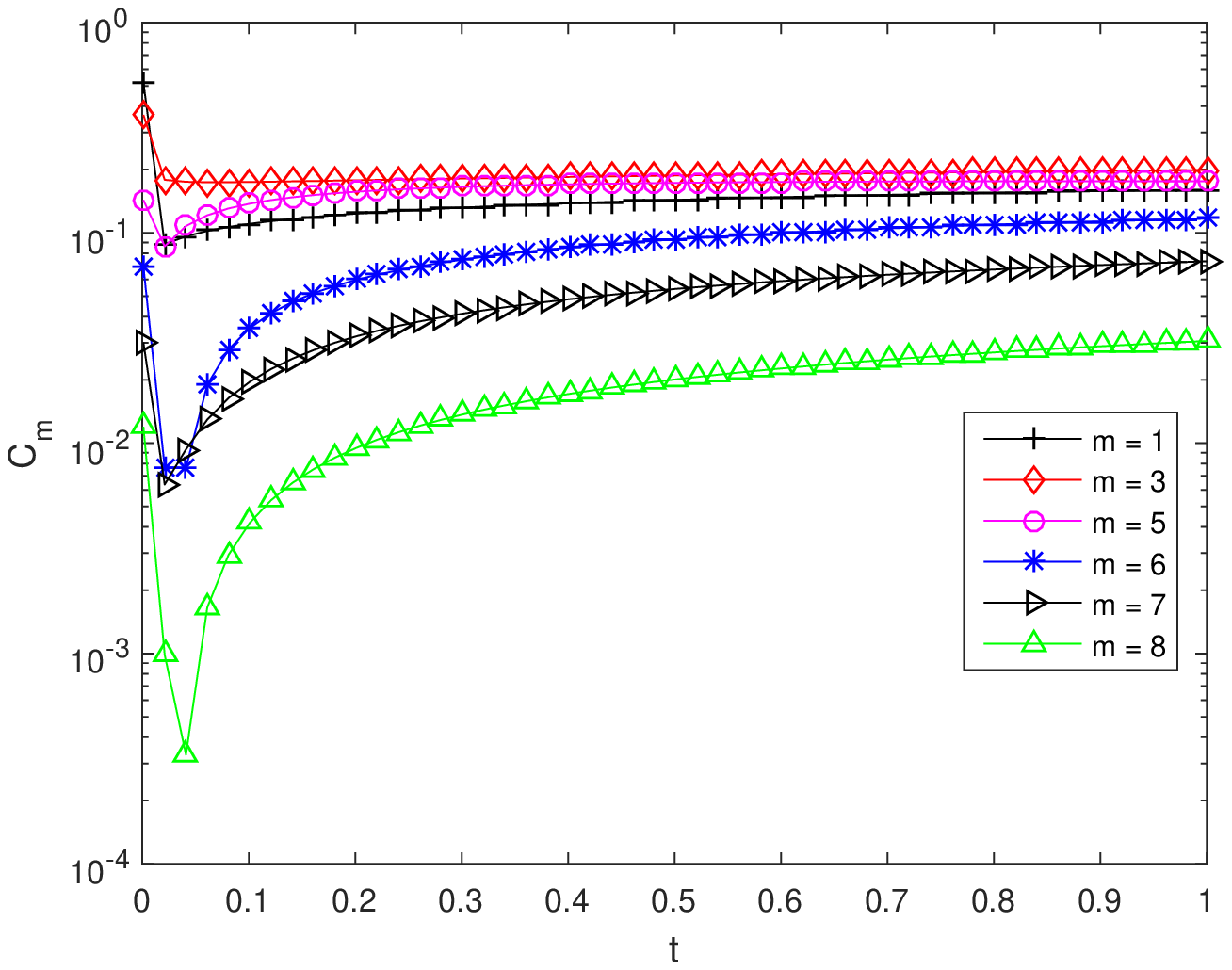,width=5cm}  \par{(c) $\alpha=0.2$.  }
\end{minipage}
\begin{minipage}{0.45\textwidth}\centering
\epsfig{figure=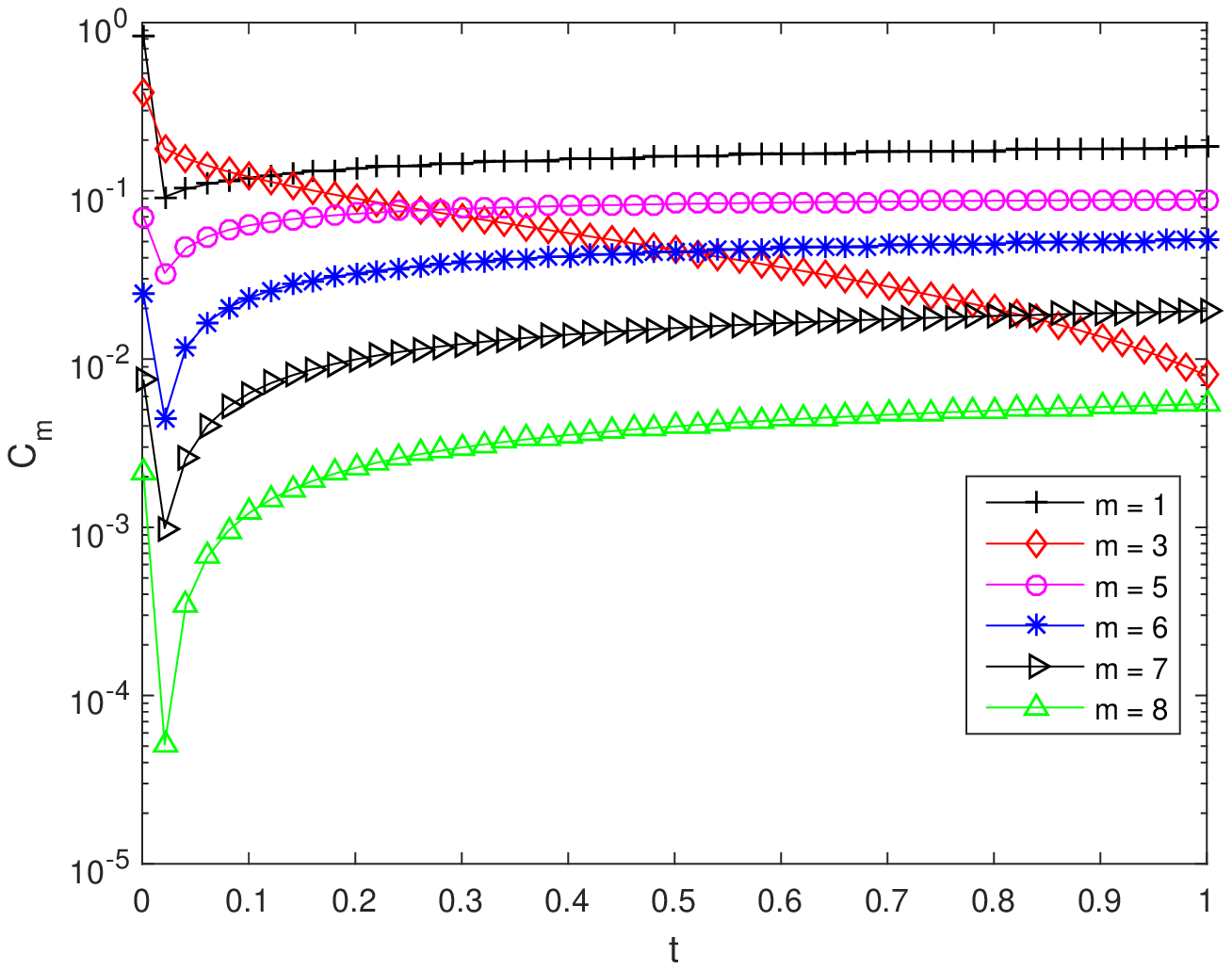,width=5cm}   \par{(d)  $\alpha=0.3$.  }
\end{minipage}
\end{center}
\caption{The  bound of constant $C_m$ in \eqref{Rn2-2},
$U(t)=t^{0.85}$, $\tau=10^{-3}$.\label{fig2-3}}
\end{figure}

\subsection{More numerical results for multi-term FODEs}
We present an example using more than ten correction terms to solve the following FODE.
\begin{example}\label{appdx-eg-1}
Consider the following two-term FODE
\begin{equation}\label{sec-appdx:eq1}
{}_{C}D^{2\alpha}_{0,t}Y(t)+{}_{C}D^{\alpha}_{0,t}Y(t)=-Y(t)+f(t),{\quad}t{\,\in\,}
 {(0,T],T>0}\\
\end{equation}
subject to the initial condition $Y(0)=1$,   and $0<\alpha\leq1/2$.
\end{example}
Choose a suitable right-hand side function $f(t)$ such that
the analytical solution of \eqref{sec-appdx:eq1} is
$$Y(t)=1+\sum_{k=1}^{16}\frac{t^{(k+1)\alpha}}{k}.$$

Here, we use the multi-precision toolbox  with 48 significant digits in the computation in order to avoid round-off errors (see http://www.mathworks.com/matlabcentral
/fileexchange/6446-multiple-precision-toolbox-for-matlab).

We consider only the fractional order $\alpha=0.05,0.1$ with $\sigma_k=(k+1)\alpha$ in \eqref{fode3}, and the pointwise errors are shown in Figs. \ref{appdx-fig2}--\ref{appdx-fig3}. We can see that very accurate numerical solutions are obtained as the number of correction terms increases. Although we did not observe the global second-order accuracy, the small factor $S_m^{\sigma}$ in the error equation makes numerical solutions very accurate as correction terms increases.

\begin{figure}[!t]
\begin{center}
\begin{minipage}{0.45\textwidth}\centering
\epsfig{figure=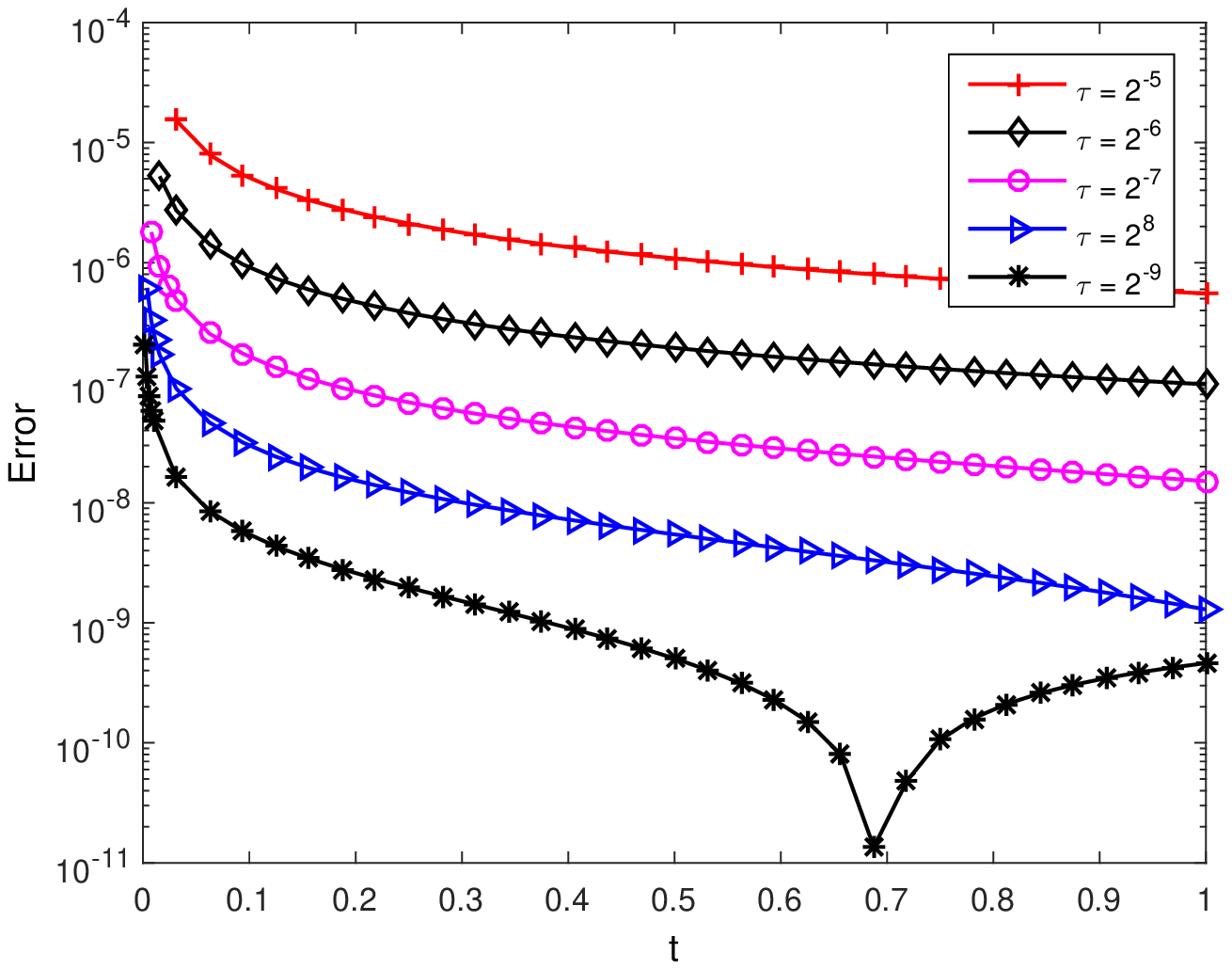,width=5.5cm}   \par{(a)  $m=10$.}
\end{minipage}
\begin{minipage}{0.45\textwidth}\centering
\epsfig{figure=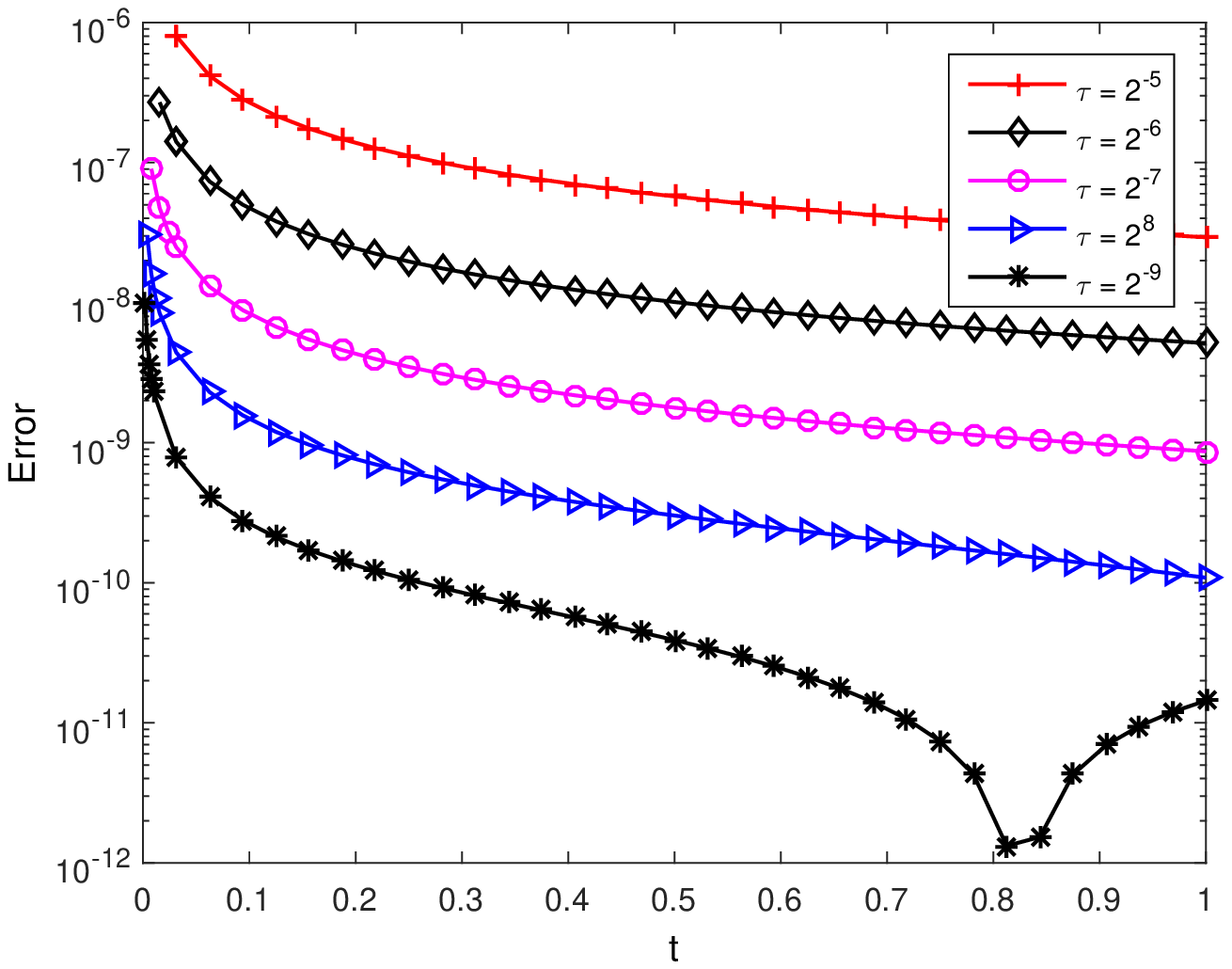,width=5.5cm}  \par{(b)   $m=12$.}
\end{minipage}
\begin{minipage}{0.45\textwidth}\centering
\epsfig{figure=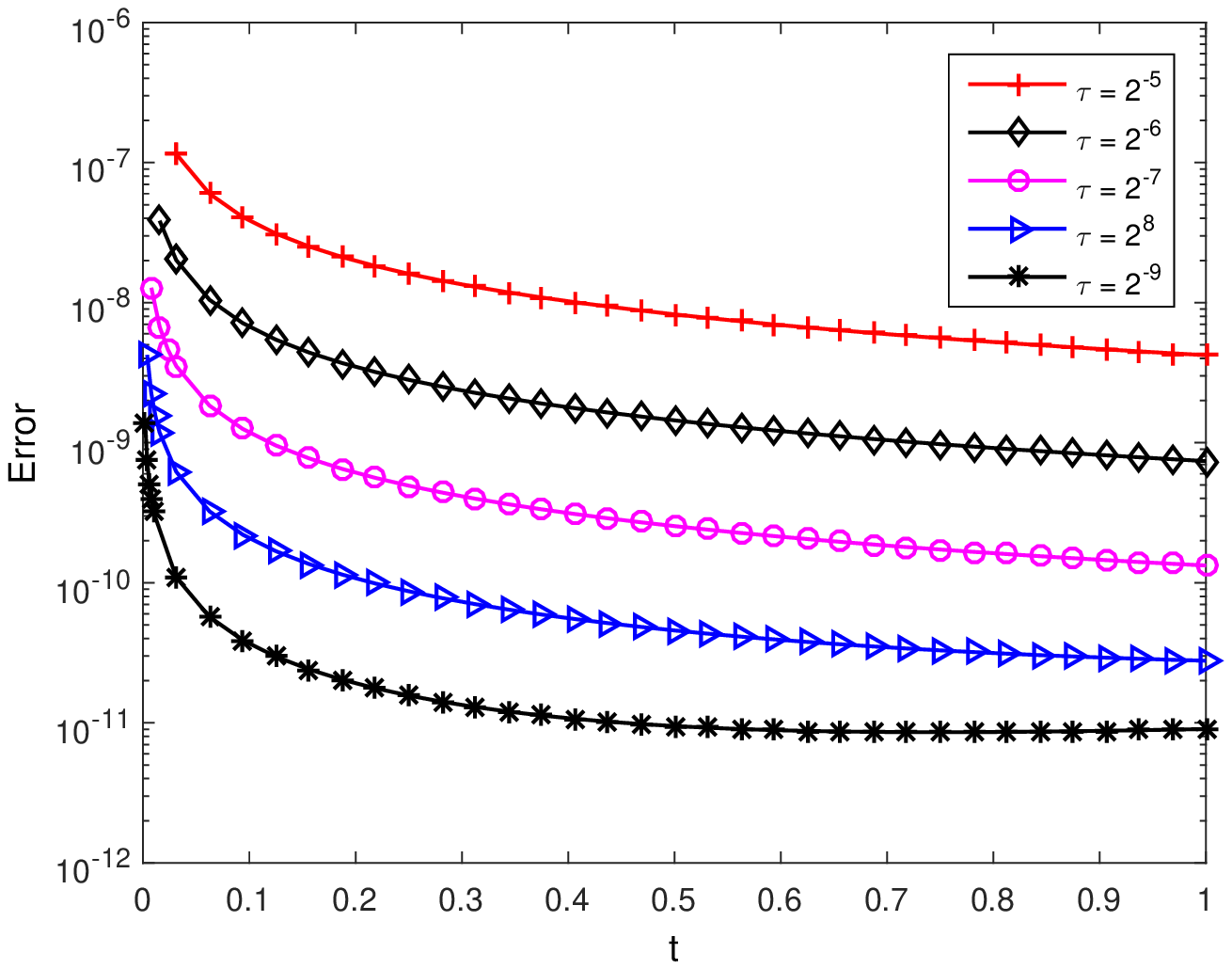,width=5.5cm}   \par{(c) $m=14$.}
\end{minipage}
\begin{minipage}{0.45\textwidth}\centering
\epsfig{figure=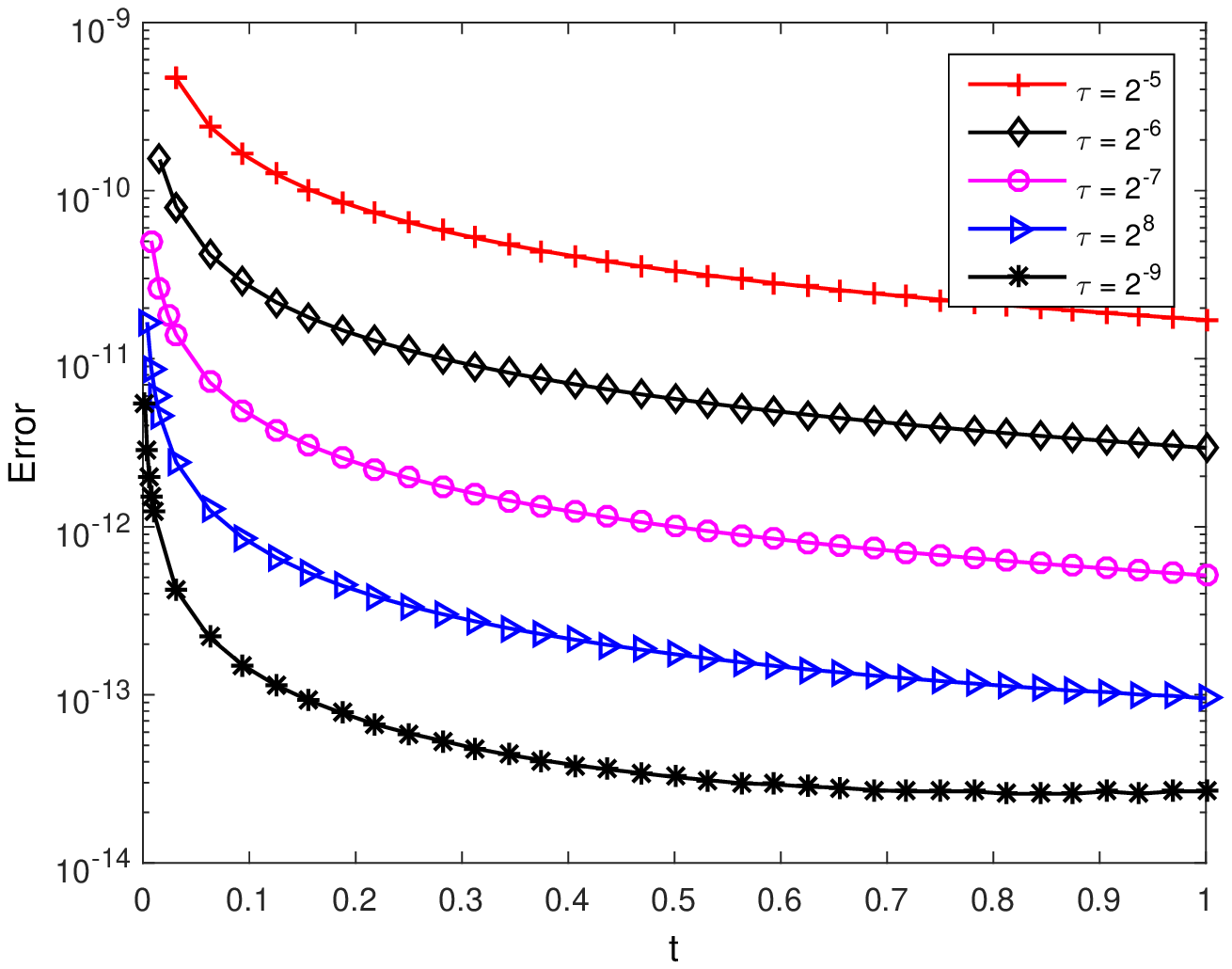,width=5.5cm}   \par{(d) $m=15$.}
\end{minipage}
\end{center}
\caption{Pointwise errors of Example \ref{appdx-eg-1},  $\alpha=0.1$.\label{appdx-fig2}}
\end{figure}

\begin{figure}[!t]
\begin{center}
\begin{minipage}{0.45\textwidth}\centering
\epsfig{figure=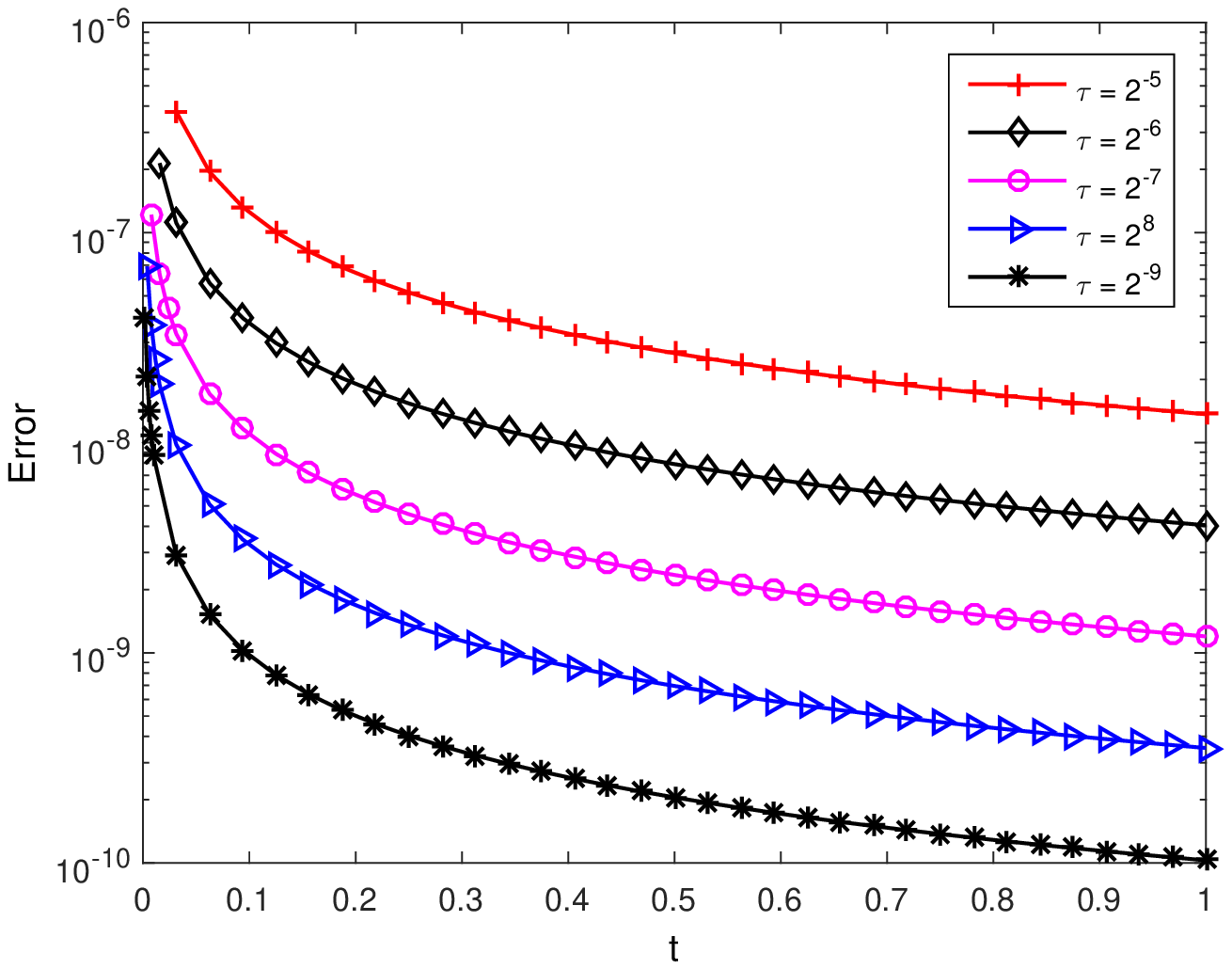,width=5.5cm}   \par{(a)  $m=10$.}
\end{minipage}
\begin{minipage}{0.45\textwidth}\centering
\epsfig{figure=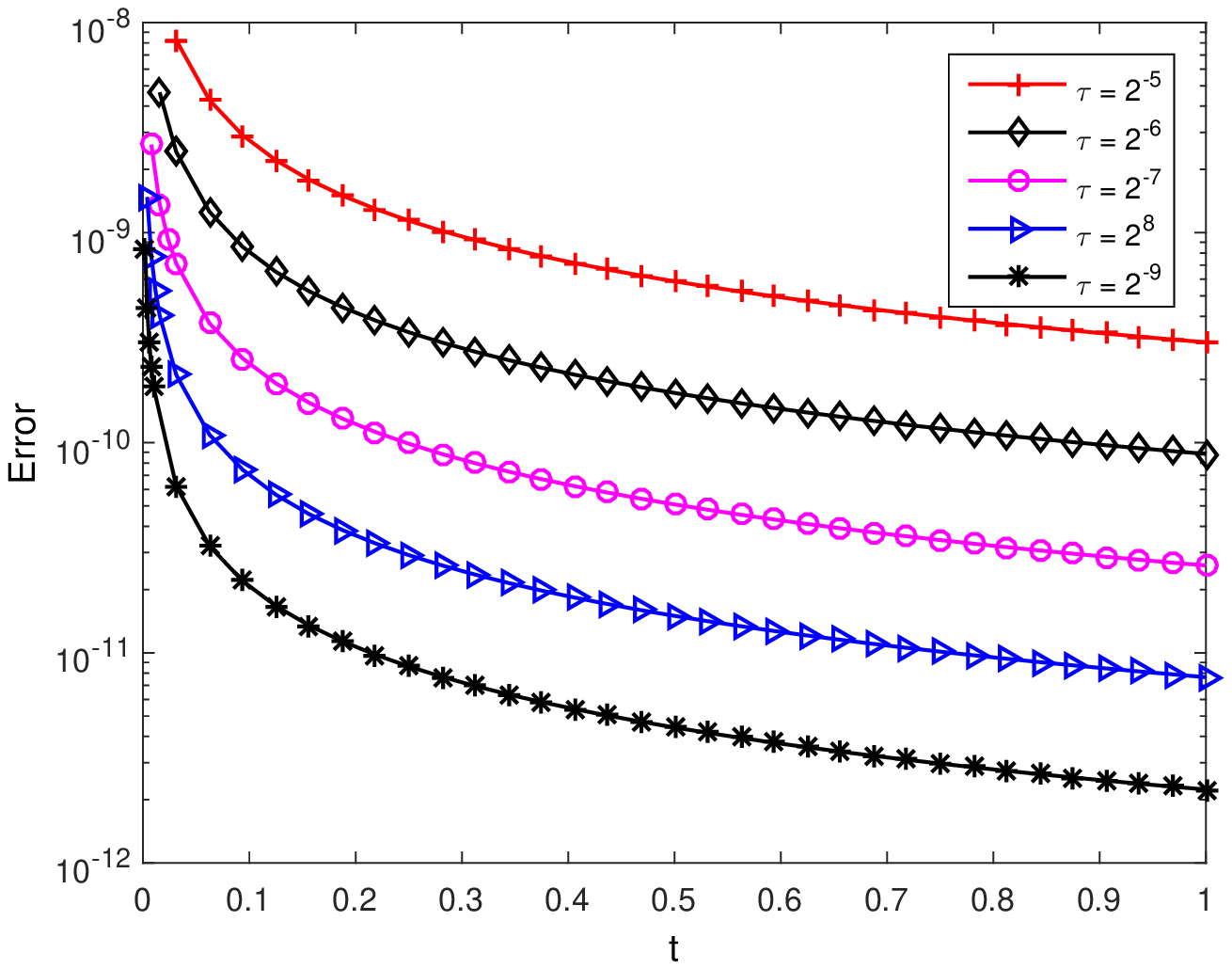,width=5.5cm}  \par{(b)   $m=12$.}
\end{minipage}
\begin{minipage}{0.45\textwidth}\centering
\epsfig{figure=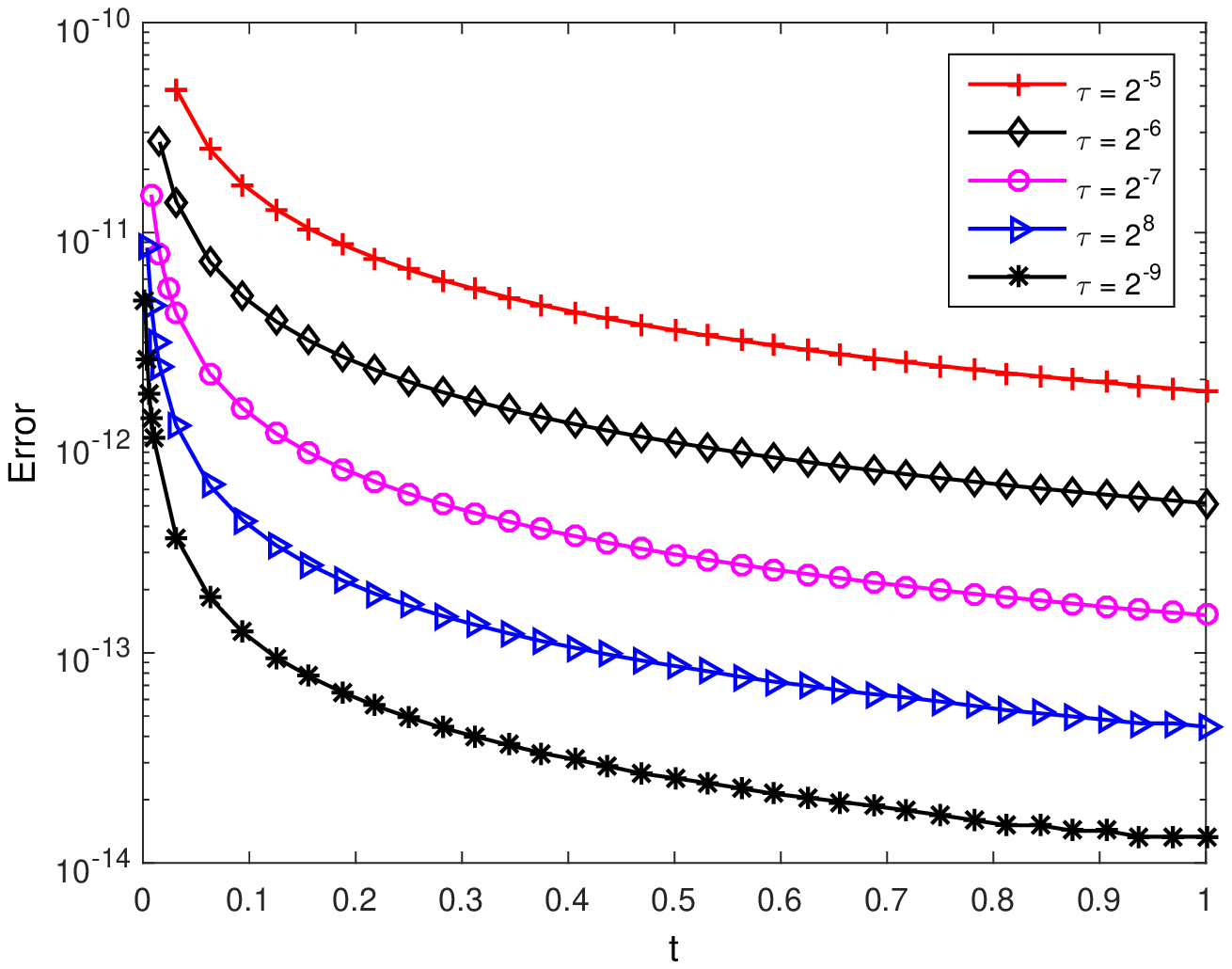,width=5.5cm}   \par{(c) $m=14$.}
\end{minipage}
\begin{minipage}{0.45\textwidth}\centering
\epsfig{figure=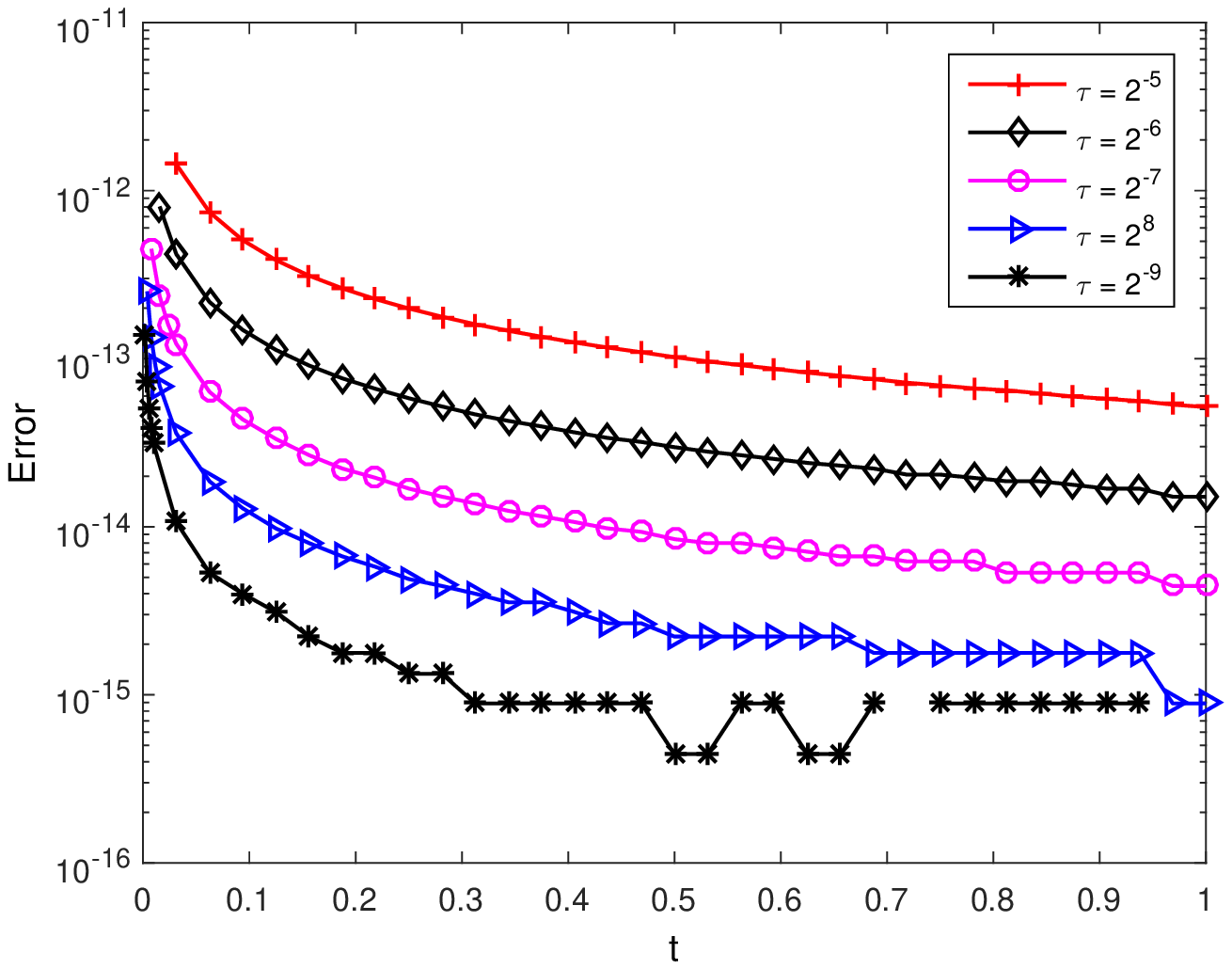,width=5.5cm}   \par{(d) $m=15$.}
\end{minipage}
\end{center}
\caption{Pointwise errors of Example \ref{appdx-eg-1},  $\alpha=0.05$.\label{appdx-fig3}}
\end{figure}

\subsection{Proofs of Theorems \ref{thm:stability} and \ref{thm:convergence}}
We prove the stability and convergence analysis of
LGSEM  \eqref{scheme2-1}--\eqref{scheme2-3}.
We first introduce a lemma.
\begin{lemma}[Gronwall's inequality \cite{QuaVal94}]\label{lm4.3}
Suppose that $\{k_n\}$, $\{\phi_n\}$ and $\{p_n\}$ are nonnegative sequence. Let $A\geq0$ and
$\phi_n$ satisfies
$$\phi_n\leq A + \sum_{j=0}^{n-1}p_j + \sum_{j=0}^{n-1}k_j\phi_j,{\quad}n\geq0.$$
Then we have
$\phi_n\leq \left(A+\sum_{j=0}^{n-1}p_j\right)\exp\left(\sum_{j=0}^{n-1}k_j\right).$
\end{lemma}

\subsubsection{Proof of  Theorem \ref{thm:stability}}
 \begin{proof}
 Letting $v=2\tau v_N^{n+\mfrac}$ in \eqref{scheme2-1} and $u=2\mu\tau u_N^{n+\mfrac}$ in \eqref{scheme2-2},
and eliminating the intermediate term $2\mu\tau\,(\px u_N^{n+\mfrac},\px v_N^{n+\mfrac})$,
we obtain
\begin{eqnarray}
&&\|v_N^{n+1}\|^2-\|v_N^{n}\|^2+2{\nu}\tau^{1-\alpha}
\sum_{k=0}^ng^{(\alpha)}_{k}(v_N^{n-k+\mfrac},v_N^{n+\mfrac})
+\mu(\|\px[x]u_N^{n+1}\|^2-\|\px[x]u_N^{n}\|^2)\nonumber\\
&=&2{\nu}\tau^{1-\alpha}\bigg[C_n(v_N^{0},v_N^{n+\mfrac})-\sum_{r=1}^{m_3}W_{n,r}(v_N^r-v_N^{0},v_N^{n+\mfrac})\bigg]
-\sum_{r=1}^{m_2}v_{n,r}(v_N^r-v_N^{0},v_N^{n+\mfrac})\nonumber\\
&&-2\mu\sum_{r=1}^{m_1}u_{n,r}(\px[x](u_N^r-u_N^{0}-t_rv_N^{0}),\px[x]u_N^{n+\mfrac})
+2\tau(I_Nf^{n+\mfrac},v_N^{n+\mfrac}),\label{scheme2-1-3}
\end{eqnarray}
where $C_n$ and $W_{n,r}(r=1,2)$ are    defined by
\begin{eqnarray}
C_n&=&\sum_{k=0}^{n}g^{(\alpha)}_{k}=O(n^{-\alpha}),\label{c-n} \\
W_{n,r}&=&\frac{1}{2}(w_{n,r}^{(\alpha)}+w_{n+1,r}^{(\alpha)}),\label{w-n}
\end{eqnarray}
in which $\{w_{n,r}^{(\alpha)}\}$ are defined by \eqref{s31-8}
{with $\sigma_r$ replaced by $\sigma_r-1$}.
Here $C_n=O(n^{-\alpha})$ can be derived
from the  fact $\sum_{k=0}^n\omega^{(\alpha)}_k=O(n^{-\alpha})$
(see Lemma 3.4 in \cite{ZengLLT15})
and the definition of $g^{(\alpha)}_{k}$  {(see \eqref{g-k})}.
Also, $|W_{n,r}|\leq Cn^{\sigma_{m_3}-3-\alpha}$ can be derived from Lemma \ref{lm4.2}
with $\sigma_m=\sigma_{m_3}-1$  {in \eqref{s4:eq-1}}.

Summing up $n$ from 0 to $K$ and applying Lemma \ref{lm4.1}, we obtain
\begin{eqnarray}
\|v_N^{K+1}\|^2 &+&\mu\|\px[x]u_N^{K+1}\|^2\leq \|v_N^{0}\|^2 +\mu\|\px[x]u_N^{0}\|^2
+2\tau\sum_{n=0}^K(I_Nf^{n+\mfrac},v_N^{n+\mfrac})  \nonumber\\
&&+2{\nu}\tau^{1-\alpha}\sum_{n=0}^KC_n(v_N^{0},v_N^{n+\mfrac})
-2{\nu}\tau^{2-\alpha} \sum_{n=0}^K \sum_{r=1}^{m_3}W_{n,r}
\left((v_N^r-v_N^{0})/\tau,v_N^{n+\mfrac}\right)\nonumber \\
&&-\tau\sum_{n=0}^K\sum_{r=1}^{m_2}v_{n,r}\left((v_N^r-v_N^{0})/\tau,v_N^{n+\mfrac}\right)
\nonumber\\
&&-2\mu\tau\sum_{n=0}^K\sum_{r=1}^{m_1}u_{n,r}
\left(\px[x](u_N^r-u_N^{0})/\tau-r\px[x]v_N^{0},\px[x]u_N^{n+\mfrac}\right).\label{scheme2-1-4}
\end{eqnarray}
Applying the Cauchy--Schwarz inequality yields
\begin{eqnarray}
&&\|v_N^{K+1}\|^2 +\mu\|\px[x]u_N^{K+1}\|^2\leq \|v_N^{0}\|^2 +\mu\|\px[x]u_N^{0}\|^2
+C\tau\sum_{n=0}^K\left(\|f^{n+\mfrac}\|^2+\|v_N^{n+\mfrac}\|^2\right)\nonumber\\
&&+C\tau^{1-\alpha}\sum_{n=0}^K\left[|C_n|\left(\|v_N^{0}\|^2+\|v_N^{n+\mfrac}\|^2\right)
+ \tau  \sum_{r=1}^{m_3}|W_{n,r}|
\left(\|(v_N^r-v_N^{0})/\tau\|^2+\|v_N^{n+\mfrac}\|^2\right)\right]\nonumber\\
&&+C\tau\sum_{n=0}^K  \sum_{r=1}^{m_2}|v_{n,r}|
\left(\|(v_N^r-v_N^{0})/\tau\|^2+\|v_N^{n+\mfrac}\|^2\right)\nonumber\\
&&+C\tau\sum_{n=0}^K\sum_{r=1}^{m_1}|u_{n,r}|
\left(\|\px[x](u_N^r-u_N^{0})/\tau\|^2+\|\px[x]v_N^{0}\|^2+\|\px[x]u_N^{n+\mfrac}\|^2\right),\label{scheme2-1-6}
\end{eqnarray}
where $C$ is a positive constant independent of $\tau,h,n,k$ and $K$.
For simplicity, we denote
\begin{equation}\label{scheme2-1-6-2}\begin{aligned}
S_1^K=\sum_{n=1}^K n^{\sigma_{m_1}-3},{\quad}S_2^K = \sum_{n=1}^Kn^{\sigma_{m_2}-4},
{\quad} S_3^K= \sum_{n=1}^Kn^{\sigma_{m_3}-3-\alpha}.
\end{aligned}\end{equation}
Then we   obtain
\begin{equation}\label{scheme2-1-6-3}\begin{aligned}
S_1^K=&\sum_{n=1}^K n^{\sigma_{m_1}-3}\leq C\max\{1,K^{\sigma_{m_1}-2}\},
{\quad}S_2^K = \sum_{n=1}^Kn^{\sigma_{m_2}-4}\leq C\max\{1,K^{\sigma_{m_2}-3}\},\\
 S_3^K=& \sum_{n=1}^Kn^{\sigma_{m_3}-3-\alpha}\leq C\max\{1,K^{\sigma_{m_3}-2-\alpha}\}.
\end{aligned}\end{equation}
{Note from  \eqref{c-n}, \eqref{w-n}, and  Lemma \ref{lm4.2} that
\begin{eqnarray}\label{scheme2-1-5}
|W_{n,r}| &\leq& Cn^{\sigma_{m_3}-3-\alpha}, \,|v_{n,r}| \leq  Cn^{\sigma_{m_2}-4},\,
|u_{n,r}| \leq  Cn^{\sigma_{m_1}-3},  \, C_n = O(n^{-\alpha}).
\end{eqnarray}
Hence, we derive from  \eqref{scheme2-1-6}--\eqref{scheme2-1-6-3} }
\begin{eqnarray*}
&&\|v_N^{K+1}\|^2 +\mu\|\px[x]u_N^{K+1}\|^2
\leq C(\|v_N^{0}\|^2 +\mu\|\px[x]u_N^{0}\|^2+\|\px[x]v_N^{0}\|^2)\nonumber\\
&&+C\sum_{n=1}^{K+1}\Big(\tau^{1-\alpha}n^{-\alpha}+\tau
+\tau^{2-\alpha}n^{\sigma_{m_3}-3-\alpha}
+\tau  n^{\sigma_{m_2}-4}  \Big)\|v_N^{n}\|^2 \nonumber\\
&&+C\tau\sum_{n=1}^{K+1} n^{\sigma_{m_1}-3}\|\px[x]u_N^{n}\|^2
+ C\tau^{1-\alpha}\|v_N^{0}\|^2\sum_{n=0}^K|C_n|+C\tau\sum_{n=0}^{K+1}\|f^{n}\|^2\nonumber\\
&&+C \tau\sum_{r=1}^{m_2} \|\delta_tv_N^{r-\mfrac}\|^2S_2^K
+C \tau^{2-\alpha}\sum_{r=1}^{m_3} \|\delta_tv_N^{r-\mfrac}\|^2S_3^K
+C \tau\sum_{r=1}^{m_1} \|\px[x]\delta_tu_N^{r-\mfrac}\|^2S_1^K.\label{scheme2-1-7}
\end{eqnarray*}
For small enough $\tau$, using the assumption $\sigma_{m_1}\leq 3$,
$\sigma_{m_2},\sigma_{m_3}\leq 4$, and \eqref{scheme2-1-6-3},
we can get from the above inequality
\begin{equation}\label{scheme2-1-8}
\|v_N^{K+1}\|^2 +\mu\|\px[x]u_N^{K+1}\|^2\leq
\rho^K+ C\sum_{n=1}^{K}d_n(\|v_N^{n}\|^2+\mu\|\px[x]u_N^{n}\|^2),
\end{equation}
where
$d_n= \tau^{1-\alpha}n^{-\alpha}+\tau+\tau^{2-\alpha} n^{\sigma_{m_3}-3-\alpha}
+\tau  n^{\sigma_{m_2}-4} + \tau n^{\sigma_{m_1}-3}$, and
\begin{eqnarray}
\rho^K=& C\bigg(\|v_N^{0}\|^2 +\mu\|\px[x]u_N^{0}\|^2+\|\px[x]v_N^{0}\|^2
+ \tau\sum_{r=1}^{m_1} \|\px[x]\delta_tu_N^{r-\mfrac}\|^2S_1^K\nonumber\\
&+\tau\sum_{r=1}^{m_2}\|\delta_tv_N^{r-\mfrac}\|^2 S_2^K
+\tau^{2-\alpha}\sum_{r=1}^{m_3}\|\delta_tv_N^{r-\mfrac}\|^2 S_3^K
+\tau\sum_{n=0}^{K+1}\|f^{n}\|^2\bigg).\label{scheme2-1-9}
\end{eqnarray}
Applying Gronwall's inequality (Lemma \ref{lm4.3}) yields
\begin{equation}\label{scheme2-1-10}\begin{aligned}
&\|v_N^{K+1}\|^2 +\mu\|\px[x]u_N^{K+1}\|^2\leq
\rho^K\exp\left(C\sum_{n=1}^Kd_n\right).
\end{aligned}\end{equation}
Using the condition $\sigma_{m_1}\leq 3$ and $\sigma_{m_2},\sigma_{m_3}\leq 4$
leads to
\begin{equation}\label{scheme2-1-11}\begin{aligned}
\sum_{n=1}^Kd_n\leq C(T^{1-\alpha}+T^{2-\alpha}+T+1).
\end{aligned}\end{equation}
Applying \eqref{scheme2-1-6-3},
$\sigma_{m_1}\leq 3$ and $\sigma_{m_2},\sigma_{m_3}\leq 4$  yields
\begin{eqnarray}
\rho^K&\leq& C(\|v_N^{0}\|^2 +\mu\|\px[x]u_N^{0}\|^2+\|\px[x]v_N^{0}\|^2)
+C\sum_{r=1}^{m_1} \|\px[x]\delta_tu_N^{r-\mfrac}\|^2\nonumber\\
&&+C \sum_{r=1}^{m_2}\|\delta_tv_N^{r-\mfrac}\|^2+C \sum_{r=1}^{m_3}\|\delta_tv_N^{r-\mfrac}\|^2
+C\tau\sum_{n=0}^{K+1}\|f^{n}\|^2.\label{scheme2-1-12}
\end{eqnarray}
Combining \eqref{scheme2-1-10}--\eqref{scheme2-1-12}
reaches the conclusion.
This completes the proof.
 \end{proof}

\subsubsection{Proof of  Theorem \ref{thm:convergence}}
We now focus on the convergence of the scheme \eqref{scheme2-1}--\eqref{scheme2-3}.
Introduce the projector $P_N^{1,0}$: $H^1_0(\Omega)\to V_N^0$ as
\begin{equation}
(\px (P_N^{1,0}u-u),\px v)=0,{\quad}u\in H^1_0(\Omega),{\quad}\forall v \in V_N^0.
\end{equation}
The properties of the interpolation and  projection operators are listed below.
\begin{lemma}[\cite{CanHusetal-B06}]\label{lem2-1}
If $u{\,\in\,}{H^{r}(\Omega)},\,r{\,\geq\,}0$, then
we have
\begin{equation} \nonumber
\big\|{\px^{l}(u-I_Nu)}\big\|
{\leq}Ch^{r-l}\|{u}\|_{H^{r}(\Omega)},{\quad}l=0,1.
\end{equation}
\end{lemma}

\begin{lemma}[\cite{CanHusetal-B06}]\label{lem2-2}
If $u{\,\in\,}{H^{r}(\Omega)}\cap{H^1_0(\Omega)},\,r{\,\geq\,}1$,
then
\begin{equation} \nonumber
\big\|{\px^{l}(u-P_N^{1,0}u)}\big\|
{\leq}Ch^{r-l}\|{u}\|_{H^{r}(\Omega)},{\quad}l=0,1.
\end{equation}
\end{lemma}

Denote by $e_u=u_N-P_N^{1,0}U$, $\eta_u=U-P_N^{1,0}U$,
$e_v=v_N-P_N^{1,0}V$, and $\eta_v=V-P_N^{1,0}V$. Then we get the error equation of  \eqref{scheme2-1}--\eqref{scheme2-3}
as follows
\begin{eqnarray}
&&(\delta_te_v^{n+\mfrac},v)+\frac{1}{\tau}\sum_{r=1}^{m_2} v_{n,r}(e_v^r-e_v^0,v)
+\frac{\nu}{2}\left[\left(\mathcal{{A}}_{0,-1}^{\alpha,n+1,m_3}(e_v-e_v^0),v\right)
+\left(\mathcal{{A}}_{0,-1}^{\alpha,n,m_3}(e_v-e_v^0),v\right)\right]\nonumber\\
&&{\qquad\qquad\quad}+\mu\,(\px e_u^{n+\mfrac},\px v)=(G^n,v),\label{err-3}\\
&&(\delta_t\px[x]e_u^{n+\mfrac},\px[x]u)
+\frac{1}{\tau}\sum_{r=1}^{m_1} u_{n,r}(\px[x](e_u^r-e_u^0-t_re_v^0),v)
- (\px e_v^{n+\mfrac},\px u)=(H^n,u),\label{err-4}
\end{eqnarray}
where $u,v\in V_N^0$,  $H^n=O(\tau^2t_n^{\sigma_{m_1+1}-3})$,  and
\begin{eqnarray*}
G^n&=&\delta_t\eta_v^{n+\mfrac}
+\frac{1}{\tau}\sum_{r=1}^{m_2} v_{n,r}(\eta_v^r-\eta_v^0)+\frac{\nu}{2}
\left(\mathcal{{A}}_{0,-1}^{\alpha,n+1,m_3}(\eta_v-\eta_v^0)
+\mathcal{{A}}_{0,-1}^{\alpha,n,m_3}(\eta_v-\eta_v^0)\right)\\
&&+I_Nf^{n+1}-f^{n+1}+O(\tau^2t_n^{\sigma_{m_2+1}-4})+O(\tau^2t_n^{\sigma_{m_3+1}-3-\alpha}).
\end{eqnarray*}
\begin{proof}
From Theorem \ref{thm:stability}, \eqref{err-3}--\eqref{err-4}, and
the properties $e_v^0=e_u^0=0$, we can similarly derive
\begin{equation}\label{s5:eq-11}\begin{aligned}
\|e_v^{n}\|^2 +\mu\|\px[x]e_u^{n}\|^2\leq& C R^{m_1,m_2,m_3}(e)
+C\tau\sum_{k=0}^{n}\left(\|G^{k}\|^2+\|H^k\|^2\right),
\end{aligned}\end{equation}
where $$R^{m_1,m_2,m_3}(e)=\sum_{k=1}^{m_1} \|\px[x]\delta_te_u^{k-\mfrac}\|^2
+ \sum_{k=1}^{m_2} \|\delta_te_v^{k-\mfrac}\|^2+ \sum_{k=1}^{m_3} \|\delta_te_v^{k-\mfrac}\|^2.$$


It is easy to obtain $\|H^n\|\leq Cn^{\sigma_{m_1+1}-3}\tau^{\sigma_{m_1+1}-1}$ and
$$\|G^n\|\leq C\left(h^r
+n^{\sigma_{m_2+1}-4}\tau^{\sigma_{m_2+1}-2}
+n^{\sigma_{m_3+1}-3-\alpha}\tau^{\sigma_{m_3+1}-1-\alpha}\right).$$

From \eqref{s5:eq-11}  and the boundedness of $\|G^n\|$ and $\|H^n\|$,
we derive
\begin{eqnarray}\label{s5:eq-12}
\|e_v^{n}\|^2 &+&\mu\|\px[x]e_u^{n}\|^2\leq C\bigg(  h^{2r}
+R^{m_1,m_2,m_3}(e)+\tau^{2\sigma_{m_1+1}-1}\sum_{k=1}^{n}k^{2(\sigma_{m_1+1}-3)}\notag\\
&&+\tau^{2\sigma_{m_2+1}-3}\sum_{k=1}^{n}k^{2(\sigma_{m_2+1}-4)}
+\tau^{2\sigma_{m_3+1}-1-2\alpha}\sum_{k=1}^{n}k^{2(\sigma_{m_2+1}-3-\alpha)}\bigg).
 \end{eqnarray}
Using the assumption  $\sigma_{m_1}\leq 3,\sigma_{m_2},\sigma_{m_3}\leq4$ and
applying the property $\sum_{k=1}^nk^{\beta}\leq C\max\{1,n^{1+\beta}\},\beta\in \mathbb{R},\beta\neq-1$,
we obtain
\begin{eqnarray*}
\|e_v^{n}\|^2 &&+\mu\|\px[x]e_u^{n}\|^2
\leq C\Big(h^{2r}+R^{m_1,m_2,m_3}(e)
 +\tau^{\min\{4,2{\sigma_{m_1+1}-1},2\sigma_{m_2+1}-3,2\sigma_{m_3+1}-1-2\alpha\}}
\Big).\label{s5:eq-13}
\end{eqnarray*}

We can get
$R^{m_1,m_2,m_3}(e) \leq C (h^{2r-2}+\tau^{\min\{4,2{\sigma_{m_1+1}-1},2\sigma_{m_2+1}-3,2\sigma_{m_3+1}-1-2\alpha\}} )$ from the assumption, hence
\begin{equation*}\begin{aligned}
\max\{\|e_v^{n}\|, \|\px[x]e_u^{n}\|\}\leq& C\Big(
\tau^{\min\{2,{\sigma_{m_1+1}-0.5},\sigma_{m_2+1}-1.5,\sigma_{m_3+1}-0.5-\alpha\}}+ h^{r-1}\Big).
\end{aligned}\end{equation*}
Applying the triangle inequality
$\|\px(u_N^{n}-U(t_n))\|\leq \|\px e_u^{n}\|+\|\px \eta_u^{n}\|$ and Lemma \ref{lem2-2}
yields the desired result.
\end{proof}

\subsection{Multi-term time-fractional subdiffusion equation}
Consider the following multi-term time-fractional subdiffusion equation
\begin{equation}\label{subeq}
\left\{\begin{aligned}
&{}_{C}D^{\alpha_1}_{0,t}U+\nu\,{}_{C}D^{\alpha_2}_{0,t}U=\mu\,\px^2U+f(x,t),
{\quad}(x,t){\,\in\,}\Omega{\times}(0,T],T>0,\\
&U(x,0)=\phi_0(x),{\quad}x{\,\in\,}\bar{\Omega},\\
&U(x,t)=0,{\quad}(x,t)\in\partial\Omega\times(0,T],
\end{aligned}\right.
\end{equation}
where $\Omega=(a,b)$, $0<\alpha_1,\alpha_2\leq 1$, $\nu\geq0$, and $\mu>0$.
Here, we  extend  LGSEM \eqref{scheme2-1}--\eqref{scheme2-3}  to solve
\eqref{subeq}, which can be easily  extended to
more generalized multi-term time-fractional subdiffusion equations,
see e.g. \cite{JinLaz-etal15,Liu-etal13,RenSun14}.
We directly present the LGSEM   for \eqref{subeq} as:
Find $u_N^{n}\in V_N^0$ for $n=1,2,...,n_T$,  such that
\begin{eqnarray}
&&\left(\mathcal{{A}}_{0,-1}^{{\alpha_1},n,m_1} \hat{u}_N ,v\right)
+\nu\left(\mathcal{{A}}_{0,-1}^{{\alpha_2},n,m_2} \hat{u}_N ,v\right)
+\mu(\px u_N^{n},\px v)=(I_Nf^{n},v),\,\,\forall v\in V_N^0, \label{scheme3-1}\\
&&u_N^0=P_N^{1,0}U(0).\label{scheme3-3}
\end{eqnarray}
where $\hat{u}_N^{k}=u_N^{k}-u_N^0\,(k\geq 0)$, $m_1,m_2$ are suitable positive integers,
and $\mathcal{{A}}_{0,-1}^{\alpha_k,n,m_k}\,(k=1,2)$
are defined as in \eqref{s31-1-3}.

The  convergence of the scheme \eqref{scheme3-1}--\eqref{scheme3-3}  can be similarly
proven as that of Theorem \ref{thm:convergence}, which is given by
$$\left(\tau\sum_{k=0}^n\|\px (u_N^{k}-U(t_k))\|^2\right)^{1/2}
\leq C\left(\tau^{\min\{2,\sigma_{m_1+1}-{\alpha_1},\sigma_{m_2+1}-{\alpha_2}\}}
+ h^{r-1}\right),$$
where the analytical solution of \eqref{subeq} satisfies
$U(t)-U(0)=\sum_{r=1}^{m}c_rt^{\sigma_r}+u(t)t^{\sigma_{m+1}}$, $u(t)$ is uniformly bounded
for $t\in[0,T]$, $U(t)\in H_0^r(\Omega)$,
 and {$\sigma_{m_1},\sigma_{m_2}\leq 3$}.

\begin{remark}
If $m_1=m_2=0$ in \eqref{scheme3-1}--\eqref{scheme3-3}, then we have
\begin{eqnarray*}
\|\px (u_N^{n}-U(t_n))\|
&\leq& C\left(\tau^{\min\{2,\sigma_{1}-{\alpha_1},\sigma_{1}-{\alpha_2}\}}+ h^{r-1}\right),\\
\left(\tau\sum_{k=0}^n\|\px (u_N^{k}-U(t_k))\|^2\right)^{1/2}
&\leq& C\left(\tau^{\min\{2,\sigma_{1}-{\alpha_1}+0.5,\sigma_{1}-{\alpha_2}+0.5\}}+h^{r-1}\right).
\end{eqnarray*}
\end{remark}

\begin{example}\label{exm:throw-at-far-field}
Consider the following time-fractional subdiffusion equation
\begin{equation}\label{app-sec6:eq2}
{}_{C}D^{\alpha_1}_{0,t}U+{}_{C}D^{\alpha_2}_{0,t}U
=\px^2U+\exp(-t){\sin(\pi x)},{\quad}(x,t){\,\in\,}(0,1){\times}(0,T],T>0
\end{equation}
subject to the homogenous initial and boundary conditions, $\alpha_1=3/4$, and $\alpha_2 =1/2$.
\end{example}

Next, we use the scheme \eqref{scheme3-1}--\eqref{scheme3-3} with $m_1=m_2=m$
to solve \eqref{app-sec6:eq2}.
Here in space, we use two subdomains: $[0,1]=[0,1/2]\cup[1/2,1]$, and $N=(32,32)$. We observe
numerically that the resolution in space is fine enough and the total error is
dominated by errors from time discretization.

It is known in \cite[p. 183]{Diethelm-B10}   that the analytical solution $U(t)$
to \eqref{app-sec6:eq2} satisfies
$U(t)=\sum_{k=1}^{\infty}c_kt^{\sigma_k}$, where $\sigma_k=(2+k)/4$.
As we do not have the explicit form of the solution, we use reference solutions that are obtained
with smaller time stepsize $\tau=2^{-13}$.

In Table \ref{apx-tb2-1}, we observe that
the average  $L^2$ errors    become smaller  when $m$ increases.
We also observe second-order accuracy  when   $m=3$.
In the first column of Table \ref{apx-tb2-1}, we also  list numerical results from the scheme
of applying the L1 method in time \cite{JinLaz-etal15,RenSun14} with spatial discretization by the
spectral element method (L1-SEM). We  observe first-order accuracy of the L1-SEM,
with the corresponding errors much larger than those by our proposed schemes.

We observe that  we do not need to use the
correction terms to get second-order accuracy when  computing solutions at time far from $t=0$.
Hence, we can still  use
the method \eqref{scheme3-1}--\eqref{scheme3-3}   to solve \eqref{app-sec6:eq2}, but the operator
$\mathcal{{A}}_{0,-1}^{{\alpha_k},n,m_k}(k=1,2)$ can be replaced with
$\mathcal{{A}}_{0,-1}^{{\alpha_k},n}$  when $n\geq\lceil n_T/5\rceil$, i.e., $w_{n,r}=0$
in \eqref{s31-1-3} for $n\geq\lceil n_T/5\rceil$.
It is shown   in  Table  \ref{apx-tb2-2}
that similar error behaviors are obtained, compared to those results in Table  \ref{apx-tb2-1}.  This can be readily explained by the truncation  error defined in \eqref{s31-1}:
When $n\geq n_0$, $n_0$ is suitably large, the correction term in
$\mathcal{{A}}_{0,-1}^{{\alpha_k},n,m_k}\,(k=1,2)$
contributes little to   accuracy and
convergence rate of  the method \eqref{scheme3-1}--\eqref{scheme3-3}.

\begin{table}[!ht]
\caption{{The  average $L^{2}$ errors $\left(\tau\sum_{n=0}^{n_T}\|u_N^{n}-U(t_n)\|^2\right)^{1/2}$
for Example \ref{exm:throw-at-far-field}, $N=(16,16)$,  {$T=1$}.}}\label{apx-tb2-1}
\centering\footnotesize
\begin{tabular}{|c|c|c|c|c|c|c|c|c|}
\hline
 $\tau$    & L1-SEM & Order & $m=1$ & Order& $m=2$ & Order& $m=3$ & Order \\
 \hline
$2^{-7}$ &6.3514e-4&    &1.5330e-4&    &1.3581e-5&    &1.5120e-5&    \\
$2^{-8}$ &3.3779e-4&0.91&6.1717e-5&1.31&7.5292e-6&0.85&4.1216e-6&1.87\\
$2^{-9}$ &1.7322e-4&0.96&2.4066e-5&1.35&3.1527e-6&1.25&8.6364e-7&2.25\\
$2^{-10}$&8.4504e-5&1.03&9.1040e-6&1.40&1.1281e-6&1.48&1.2301e-7&2.81\\
$2^{-11}$&3.7468e-5&1.17&3.2316e-6&1.49&3.5344e-7&1.67&2.3307e-8&2.39\\
\hline
\end{tabular}
\end{table}

\begin{table}[!ht]
\caption{{The  average $L^{2}$ errors $\left(\tau\sum_{n=0}^{n_T}\|u_N^{n}-U(t_n)\|^2\right)^{1/2}$
for Example \ref{exm:throw-at-far-field}, $N=(16,16), {T=1}$, and  $w_{n,k}=0$ in \eqref{scheme3-1} for $n\geq\lceil n_T/5\rceil$,
see also $w_{n,k}$ in \eqref{s31-1-3}.}}\label{apx-tb2-2}
\centering\footnotesize
\begin{tabular}{|c|c|c|c|c|c|c|c|c|}
\hline
 $\tau$ &$m=0$ &Order&$m=1$ & Order& $m=2$ & Order  & $m=3$& Order\\
 \hline
$2^{-7}$ &6.3706e-4&    &1.5347e-4&    &1.3536e-5&    &1.5065e-5&    \\
$2^{-8}$ &3.3836e-4&0.91&6.1769e-5&1.31&7.5037e-6&0.85&4.1149e-6&1.87\\
$2^{-9}$ &1.7304e-4&0.96&2.4081e-5&1.35&3.1433e-6&1.25&8.6822e-7&2.24\\
$2^{-10}$&8.4077e-5&1.04&9.1086e-6&1.40&1.1249e-6&1.48&1.3063e-7&2.73\\
$2^{-11}$&3.7071e-5&1.18&3.2329e-6&1.49&3.5248e-7&1.67&2.8196e-8&2.21\\
\hline
\end{tabular}
\end{table}


\subsection{The $L1$ method and fractional trapezoidal rule used in Example 5.2}
\begin{itemize}
  \item L1 method:  The Caputo derivatives in \eqref{sec6:eq0-2}  are discretized by the L1 method, and the corresponding scheme is given by
\begin{equation}\label{L1}
\sum_{j=1}^2\sum_{k=0}^{n-1}b^{(\alpha_j)}_{n-k-1}(y^{k+1}-y^k)
= f(t_n,y^n),{\quad}y^0=Y_0,
\end{equation}
where  $f(t,Y)=Y(1-Y^2)+\cos(t)$ and
$b^{(\alpha)}_k=\frac{\tau^{-\alpha}}{\Gamma(2-\alpha)}[(k+1)^{1-\alpha}-k^{1-\alpha}]$.

  \item Trapezoidal rule method: We   transform \eqref{sec6:eq0-2} into its integral form as $Y(t)+D^{-(\alpha_1-\alpha_2)}_{0,t}(Y(t)-Y(0))= Y(0)+D^{-(\alpha_1-\alpha_2)}_{0,t}(Y(t)(1-Y^2(t))+\cos(t))$, then the trapezoidal rule is applied to the fractional integrals. The corresponding   scheme is given by
\begin{equation}\label{trapezoidal}
(y^n -y^0)+ \sum_{k=0}^{n}a^{(\alpha_1-\alpha_2)}_{n,k}(y^k-y^0)
= \sum_{k=0}^{n}a^{(\alpha_1)}_{n,k}f(t_k,y^k),{\quad}y^0=Y_0,
\end{equation}
where   $f(t,Y)=Y(1-Y^2)+\cos(t)$ and $a^{(\alpha)}_{n,k}$ is given by
\begin{equation*}
a^{(\alpha)}_{n,k}=\frac{\tau^{\alpha}}{\Gamma(2+\alpha)}\left\{\begin{aligned}
&(n-1)^{\alpha+1}-(n-1-\alpha)n^{\alpha},{\quad}k=0,\\
&(n-k+1)^{\alpha+1}-2(n-k)^{\alpha+1}+(n-k-1)^{\alpha+1},\\
&{\qquad\qquad}1\leq k \leq n-1,\\
&1,{\quad}k=n.
\end{aligned}\right.
\end{equation*}
\end{itemize}


\begin{thebibliography}{10}


\bibitem{CaiCZh14}
{\sc W.~Cai, W.~Chen, and X.~Zhang}, {\em A {M}atlab toolbox for positive
  fractional time derivative modeling of arbitrarily frequency-dependent
  viscosity}, J. Vib. Control, 20 (2014), pp.~1009--1016.

\bibitem{CanHusetal-B06}
{\sc C.~Canuto, M.~Y. Hussaini, A.~Quarteroni, and T.~A. Zang}, {\em Spectral
  methods}, Scientific Computation, Springer-Verlag, Berlin, 2006.
\newblock Fundamentals in single domains.

\bibitem{CaoHX03}
{\sc Y.~Cao, T.~Herdman, and Y.~Xu}, {\em A hybrid collocation method for
  {V}olterra integral equations with weakly singular kernels}, SIAM J. Numer.
  Anal., 41 (2003), pp.~364--381 (electronic).

\bibitem{CelikDuman12}
{\sc C.~{\c{C}}elik and M.~Duman}, {\em Crank-{N}icolson method for the
  fractional diffusion equation with the {R}iesz fractional derivative}, J.
  Comput. Phys., 231 (2012), pp.~1743--1750.

\bibitem{ChenLTA07}
{\sc C.-M. Chen, F.~Liu, I.~Turner, and V.~Anh}, {\em A {F}ourier method for
  the fractional diffusion equation describing sub-diffusion}, J. Comput.
  Phys., 227 (2007), pp.~886--897.

\bibitem{ChenDeng14}
{\sc M.~Chen and W.~Deng}, {\em Fourth order accurate scheme for the space
  fractional diffusion equations}, SIAM J. Numer. Anal., 52 (2014),
  pp.~1418--1438.

\bibitem{ChenSW14}
{\sc S.~Chen, J.~Shen, and L.-L. Wang}, {\em Generalized {J}acobi functions and
  their applications to fractional differential equations},
  Math.   Comp., 2015,  http://dx.doi.org/10.1090/mcom3035.

\bibitem{CueLubPal06}
{\sc E.~Cuesta, C.~Lubich, and C.~Palencia}, {\em Convolution quadrature time
  discretization of fractional diffusion-wave equations}, Math. Comp., 75
  (2006), pp.~673--696 (electronic).

\bibitem{Diethelm-B10}
{\sc K.~Diethelm}, {\em The Analysis of Fractional Differential Equations}, Springer-Verlag, Berlin, 2010.


\bibitem{DieFord06}
{\sc K.~Diethelm, J.~M. Ford, N.~J. Ford, and M.~Weilbeer}, {\em Pitfalls in
  fast numerical solvers for fractional differential equations}, J. Comput.
  Appl. Math., 186 (2006), pp.~482--503.

\bibitem{DieFF04}
{\sc K.~Diethelm, N.~J. Ford, and A.~D. Freed}, {\em Detailed error analysis
  for a fractional {A}dams method}, Numer. Algorithms, 36 (2004), pp.~31--52.


\bibitem{DingLC14a}
{\sc H.~Ding, C.~Li, and Y.~Chen}, {\em High-order algorithms for {R}iesz
  derivative and their applications {(I)}}, Abstr. Appl. Anal.,  (2014),
  pp.~Art. ID 653797, 17.

\bibitem{DingLC14b}
\leavevmode\vrule height 2pt depth -1.6pt width 23pt, {\em High-order
  algorithms for {R}iesz derivative and their applications ({II})}, J. Comput.
  Phys., 293 (2015), pp.~218--237.

\bibitem{EsmSL11}
{\sc S.~Esmaeili, M.~Shamsi, and Y.~Luchko}, {\em Numerical solution of
  fractional differential equations with a collocation method based on
  {M}\"untz polynomials}, Comput. Math. Appl., 62 (2011), pp.~918--929.

\bibitem{FordCon09}
{\sc N.~J. Ford and J.~A. Connolly}, {\em Systems-based decomposition schemes
  for the approximate solution of multi-term fractional differential
  equations}, J. Comput. Appl. Math., 229 (2009), pp.~282--391.

\bibitem{ForMR13}
{\sc N.~J. Ford, M.~L. Morgado, and M.~Rebelo}, {\em Nonpolynomial collocation
  approximation of solutions to fractional differential equations}, Fract.
  Calc. Appl. Anal., 16 (2013), pp.~874--891.

\bibitem{GaoSS15}
{\sc G.-H. Gao, H.-W. Sun, and Z.-Z. Sun}, {\em Stability and convergence of
  finite difference schemes for a class of time-fractional sub-diffusion
  equations based on certain superconvergence}, J. Comput. Phys., 280 (2015),
  pp.~510--528.



\bibitem{JiangLiu-etal12b}
{\sc H.~Jiang, F.~Liu, I.~Turner, and K.~Burrage}, {\em Analytical solutions
  for the multi-term time-fractional diffusion-wave/diffusion equations in a
  finite domain}, Comput. Math. Appl., 64 (2012), pp.~3377--3388.

\bibitem{JiangLiu-etal12}
\leavevmode\vrule height 2pt depth -1.6pt width 23pt, {\em Analytical solutions
  for the multi-term time-space {C}aputo-{R}iesz fractional advection-diffusion
  equations on a finite domain}, J. Math. Anal. Appl., 389 (2012),
  pp.~1117--1127.

\bibitem{JinLaz-etal15}
{\sc B.~Jin, R.~Lazarov, Y.~Liu, and Z.~Zhou}, {\em The {G}alerkin finite
  element method for a multi-term time-fractional diffusion equation}, J.
  Comput. Phys., 281 (2015), pp.~825--843.

\bibitem{JinZhou14}
{\sc B.~Jin and Z.~Zhou}, {\em A finite element method with singularity
  reconstruction for fractional boundary value problems},
  ESAIM: M2AN, 49 (2015), pp. 1261--1283.

\bibitem{LiZeng13}
{\sc C.~Li and F.~Zeng}, {\em The finite difference methods for fractional
  ordinary differential equations}, Numer. Funct. Anal. Optim., 34 (2013),
  pp.~149--179.


\bibitem{LiLiu14}
{\sc Z.~Li, Y.~Liu, and M.~Yamamoto}, {\em Initial-boundary value problems for
  multi-term time-fractional diffusion equations with positive constant
  coeffcients}, arXiv:1312.2112v2,  (2014).

\bibitem{LinXu07}
{\sc Y.~Lin and C.~Xu}, {\em Finite difference/spectral approximations for the
  time-fractional diffusion equation}, J. Comput. Phys., 225 (2007),
  pp.~1533--1552.



\bibitem{Liu-etal13}
{\sc F.~Liu, M.~M. Meerschaert, R.~J. McGough, P.~Zhuang, and Q.~Liu}, {\em
  Numerical methods for solving the multi-term time-fractional wave-diffusion
  equation}, Fract. Calc. Appl. Anal., 16 (2013), pp.~9--25.


\bibitem{Lub86}
{\sc C.~Lubich}, {\em Discretized fractional calculus}, SIAM J. Math. Anal., 17
  (1986), pp.~704--719.

\bibitem{Lub86b}
\leavevmode\vrule height 2pt depth -1.6pt width 23pt, {\em A stability analysis
  of convolution quadratures for {A}bel-{V}olterra integral equations}, IMA J.
  Numer. Anal., 6 (1986), pp.~87--101.

\bibitem{Luchko11}
{\sc Y.~Luchko}, {\em Initial-boundary problems for the generalized multi-term
  time-fractional diffusion equation}, J. Math. Anal. Appl., 374 (2011),
  pp.~538--548.


\bibitem{MaoShen16}
{\sc Z. P. Mao and J. Shen}, {\em Efficient spectral--{G}alerkin methods for fractional partial
differential equations with variable coefficients}, J. Comput. Phys., 307 (2016), pp.~243--261.


\bibitem{McLMus07}
{\sc W.~McLean and K.~Mustapha}, {\em A second-order accurate numerical method
  for a fractional wave equation}, Numer. Math., 105 (2007), pp.~481--510.

\bibitem{MeeTad04}
{\sc M.~M. Meerschaert and C.~Tadjeran}, {\em Finite difference approximations
  for fractional advection-dispersion flow equations}, J. Comput. Appl. Math.,
  172 (2004), pp.~65--77.

\bibitem{MusMcL13}
{\sc K.~Mustapha and W.~McLean}, {\em Superconvergence of a discontinuous
  {G}alerkin method for fractional diffusion and wave equations}, SIAM J.
  Numer. Anal., 51 (2013), pp.~491--515.

\bibitem{PdeasTamme11}
{\sc A.~Pedas and E.~Tamme}, {\em Spline collocation methods for linear
  multi-term fractional differential equations}, J. Comput. Appl. Math., 236
  (2011), pp.~167--176.

\bibitem{Pod-B99}
{\sc I.~Podlubny}, {\em Fractional differential equations}, Academic Press,
  Inc., San Diego, CA, 1999.

\bibitem{QuaVal94}
{\sc A.~Quarteroni and A.~Valli}, {\em Numerical approximation of partial
  differential equations}, vol.~23 of Springer Series in Computational
  Mathematics, Springer-Verlag, Berlin, 1994.

\bibitem{QuiYus13}
{\sc J.~Quintana-Murillo and S.~Yuste}, {\em A finite difference method with
  non-uniform timesteps for fractional diffusion and diffusion-wave equations},
  The European Physical Journal Special Topics, 222 (2013), pp.~1987--1998.

\bibitem{RenSun14}
{\sc J.~Ren and Z.-z. Sun}, {\em Efficient and stable numerical methods for
  multi-term time fractional sub-diffusion equations}, East Asian J. Appl.
  Math., 4 (2014), pp.~242--266.

\bibitem{Sousa12}
{\sc E.~Sousa}, {\em How to approximate the fractional derivative of order
  {$1<\alpha\leq2$}}, Internat. J. Bifur. Chaos Appl. Sci. Engrg., 22 (2012),
  pp.~1250075, 13.

\bibitem{TianZD14}
{\sc W.~Tian, H.~Zhou, and W.~Deng}, {\em A class of second order difference
  approximation for solving space fractional diffusion equations}, Math. Comp.,
  84 (2015), pp. 1703--1727.

\bibitem{WangZhang15}
{\sc H.~Wang and X.~Zhang}, {\em A high-accuracy preserving spectral {G}alerkin
  method for the {D}irichlet boundary-value problem of variable-coefficient
  conservative fractional diffusion equations}, J. Comput. Phys., 281 (2015),
  pp.~67--81.

\bibitem{WangVong14b}
{\sc Z.~Wang and S.~Vong}, {\em Compact difference schemes for the modified
  anomalous fractional sub-diffusion equation and the fractional diffusion-wave
  equation}, J. Comput. Phys., 277 (2014), pp.~1--15.

\bibitem{YeLiu14}
{\sc H.~Ye, F.~Liu, V.~Anh, and I.~Turner}, {\em Maximum principle and
  numerical method for the multi-term time-space {R}iesz-{C}aputo fractional
  differential equations}, Appl. Math. Comput., 227 (2014), pp.~531--540.




\bibitem{Yuste06}
{\sc S.~B. Yuste}, {\em Weighted average finite difference methods for
  fractional diffusion equations}, J. Comput. Phys., 216 (2006), pp.~264--274.

\bibitem{ZayKar14b}
{\sc M.~Zayernouri and G.~E. Karniadakis}, {\em Fractional spectral collocation
  method}, SIAM J. Sci. Comput., 36 (2014), pp.~A40--A62.

\bibitem{Zeng14}
{\sc F.~Zeng}, {\em Second-order stable finite difference schemes for the
  time-fractional diffusion-wave equation}, J. Sci. Comput.,  65 (2015), pp. 411--430.

\bibitem{ZengLLT13}
{\sc F.~Zeng, C.~Li, F.~Liu, and I.~Turner}, {\em The use of finite
  difference/element approaches for solving the time-fractional subdiffusion
  equation}, SIAM J. Sci. Comput., 35 (2013), pp.~A2976--A3000.

\bibitem{ZengLLT15}
\leavevmode\vrule height 2pt depth -1.6pt width 23pt, {\em Numerical algorithms
  for time-fractional subdiffusion equation with second-order accuracy}, SIAM
  J. Sci. Comput., 37 (2015), pp.~A55--A78.

\bibitem{ZengZK15}
{\sc F.~Zeng, Z.~Zhang, and G.~E. Karniadakis}, {\em A generalized spectral
  collocation method with tunable accuracy for variable-order fractional
  differential equations}, SIAM J. Sci. Comput. 37 (2015), pp.~A2710--A2732.

\bibitem{ZengZK16}
\leavevmode\vrule height 2pt depth -1.6pt width 23pt, {\em Fast difference schemes for solving high-dimensional time-fractional subdiffusion equations}, J. Comput. Phys. 307 (2016), pp.~15--33.






\bibitem{ZhangSunLiao14}
{\sc Y.-n. Zhang, Z.-z. Sun, and H.-l. Liao}, {\em Finite difference methods
  for the time fractional diffusion equation on non-uniform meshes}, J. Comput.
  Phys., 265 (2014), pp.~195--210.

\bibitem{ZhangZK15}
{\sc Z.~Zhang, F.~Zeng, and G.~E. Karniadakis}, {\em Optimal error estimates
  for spectral {P}etrov-{G}alerkin and collocation methods for initial value
  problems for fractional differential equations}, SIAM J. Numer. Anal., 53 (2015),  pp. 2074--2096.


\bibitem{ZhaoSH14}
{\sc X.~Zhao, Z.-z. Sun, and Z.-p. Hao}, {\em A fourth-order compact {ADI}
  scheme for two-dimensional nonlinear space fractional {S}chr\"odinger
  equation}, SIAM J. Sci. Comput., 36 (2014), pp.~A2865--A2886.


\bibitem{ZhengLATS15}
{\sc M. Zheng, F. Liu, V. Anh, and I. Turner}, {\em A high order spectral method for the multi-term time-fractional diffusion equations}, Appl. Math. Modelling, 2015, in press.



\bibitem{ZhouTD13}
{\sc H.~Zhou, W.~Tian, and W.~Deng}, {\em Quasi-compact finite difference
  schemes for space fractional diffusion equations}, J. Sci. Comput., 56
  (2013), pp.~45--66.

%
%
%
\end{thebibliography}
\end{document}